\documentclass{elsarticle}
\usepackage{fullpage}
\usepackage{enumitem}
\usepackage{multirow,multicol}
\usepackage{upgreek}
\usepackage{color}
\usepackage[usenames,dvipsnames,svgnames,table]{xcolor}
\usepackage{amsmath,amssymb,amsopn,mathtools}
\usepackage{graphicx}
\usepackage{pgfplots,pgfplotstable}

\usepackage{subcaption}

\usepackage{hyperref}       
\hypersetup{
	colorlinks=true,   
	citecolor=blue,     
	filecolor=blue,     
	linkcolor=red,    
	urlcolor=blue}      

\usepackage{relsize}
\numberwithin{equation}{section}

\newcommand{\reviewerA}[1]{{#1}}
\newcommand{\reviewerB}[1]{{#1}}

\newcommand{\bds}{\boldsymbol}
\newcommand{\bmat}[1]{\begin{bmatrix}#1\end{bmatrix}} 
\newcommand{\sign}[1]{\mathrm{sign}(#1)}

\newtheorem{theorem}{Theorem}[section]
\newtheorem{corollary}{Corollary}[section]
\newtheorem{lemma}{Lemma}[section]
\newtheorem{remark}{Remark}
\newtheorem{definition}{Definition}
\newtheorem{proposition}{Proposition}[section]
\newenvironment{proof}{\paragraph{Proof}}{$\square$}

\newcommand{\RR}[1]{\ensuremath{\mathbb{R}^{ #1 }}}
\newcommand{\RRstar}[1]{\ensuremath{\mathbb{R}_\star^{ #1 }}}
\newcommand{\NN}{\mathbb{N}}
\newcommand{\RRplus}[1]{\ensuremath{\mathbb{R}_+^{ #1 }}}
\newcommand{\natNo}{\NN}
\newcommand{\nat}[1]{\natNo(#1)}

\newcommand{\innat}[1]{\in\nat{#1}}
\newcommand{\innatSeq}[1]{=1,\ldots,#1}

\newcommand{\defeq}{:=}
\newcommand{\range}[1]{\mathrm{Ran}\left(#1\right)}
\newcommand{\rank}[1]{\mathrm{rank}\left(#1\right)}
\newcommand{\diag}[1]{\mathrm{diag}\left(#1\right)}
\newcommand{\cond}[1]{\mathrm{cond}\left(#1\right)}
\newcommand{\abs}[1]{|#1|}
\newcommand{\norm}[1]{\|#1\|}

\newcommand{\timeSymb}{t}
\newcommand{\dt}{\Delta \timeSymb}
\newcommand{\timestepit}{n}
\newcommand{\timeArg}[1]{{\timeSymb^{#1}}}
\newcommand{\timeDummy}{\tau}

\newcommand{\solSymb}{x}
\newcommand{\sol}{\bds \solSymb}
\newcommand{\state}{\sol}
\newcommand{\stateEntry}[1]{\solSymb_{#1}}
\newcommand{\stateArg}[1]{\state^{#1}}
\newcommand{\stateInit}{\stateArg{0}}
\newcommand{\stateInitEntry}[1]{\solSymb^0_{#1}}

\newcommand{\decomp}[1]{\bar{#1}}
\newcommand{\solDecompSymb}{\decomp{\solSymb}}
\newcommand{\stateDecomp}{\decomp{\state}}

\newcommand{\stateDecompEntry}[1]{\solDecompSymb_{#1}}

\newcommand{\unknownSymb}{w}
\newcommand{\unknown}{\bds{\unknownSymb}}
\newcommand{\unknownEntry}[1]{\unknownSymb_{#1}}
\newcommand{\unknownRed}{\hat\unknown}

\newcommand{\error}{\mathcal E}
\newcommand{\errorInst}{\varepsilon}
\newcommand{\relativeError}{{\error_{\state}}}
\newcommand{\relativeErrorInst}[1]{\errorInst_{\state}^{#1}}
\newcommand{\relativeGlobalError}{\error_{\state,\text{global}}}
\newcommand{\relativeGlobalErrorInst}[1]{\errorInst_{\state,\text{global}}^{#1}}
\newcommand{\violation}{\error_{\res,\text{global}}}
\newcommand{\violationInst}[1]{\errorInst_{\res,\text{global}}^{#1}}

\newcommand{\stateApprox}{\tilde\state}
\newcommand{\stateApproxArg}[1]{\stateApprox^{#1}}

\newcommand{\nspacedof}{N}
\newcommand{\ndof}{\nspacedof}
\newcommand{\ndofFace}{\nConservation\nFaces}
\newcommand{\ndofDecomp}{\decomp{\nspacedof}}
\newcommand{\ndofDecompDecomp}{\decomp{\decomp{\nspacedof}}}

\newcommand{\spaceVar}{\vec x}
\newcommand{\spaceVarScalar}{ x}
\newcommand{\nspaceDim}{d}
\newcommand{\spaceDomain}{\Omega}
\newcommand{\dVol}{\, \mathrm{d}\spaceVar}
\newcommand{\dVolScalar}{\, \mathrm{d}\spaceVarScalar}
	\newcommand{\dInterface}{\, \mathrm{d}\vec s(\spaceVar)}
	\newcommand{\dInterfaceScalar}{\, \mathrm{d} s(\spaceVarScalar)}

\newcommand{\weightingMatrix}{}
\newcommand{\weightingMatrixT}{}
\newcommand{\weightingMatPrint}{\bds{A}}

\newcommand{\weightingMatrixArgs}[2]{\weightingMatrix}

\newcommand{\objGalOneArgs}[3]{f_\mathrm{G,I}(#1,#2;#3)}
\newcommand{\objGalTwoArgs}[3]{f_\mathrm{G,II}(#1,#2;#3)}
\newcommand{\objGalGenArgs}[4]{f_\mathrm{G,#4}(#1,#2;#3)}
\newcommand{\objLSPGOneArgs}[2]{f^n_\mathrm{P,I}(#1;#2)}
\newcommand{\objLSPGTwoArgs}[2]{f^n_\mathrm{P,II}(#1;#2)}
\newcommand{\objLSPGGenArgs}[3]{f^n_\mathrm{P,#3}(#1;#2)}
\newcommand{\constraintGalOneArgs}[3]{\bds{c}_\mathrm{G,I}(#1,#2;#3)}
\newcommand{\constraintGalTwoArgs}[3]{\bds{c}_\mathrm{G,II}(#1,#2;#3)}
\newcommand{\constraintGalGenArgs}[4]{\bds{c}_\mathrm{G,#4}(#1,#2;#3)}
\newcommand{\constraintLSPGOneArgs}[2]{\bds{c}^n_\mathrm{P,I}(#1;#2)}
\newcommand{\constraintLSPGTwoArgs}[2]{\bds{c}^n_\mathrm{P,II}(#1;#2)}
\newcommand{\constraintLSPGGenArgs}[3]{\bds{c}^n_\mathrm{P,#3}(#1;#2)}

\newcommand{\testBasis}{\bds{\Psi}}
\newcommand{\testBasisArgs}[3]{\testBasis^{#1}(#2;#3)}
\newcommand{\testBasisArgsTwo}[2]{\testBasis^{#1}(#2)}

\newcommand{\paramSymb}{\mu}
\newcommand{\params}{\bds \paramSymb}
\newcommand{\param}{\params}
\newcommand{\paramsDummy}{\bds \nu}
\newcommand{\paramDomain}{\mathcal D}
\newcommand{\nparam}{n_\mu}

\newcommand{\timeVar}{t}
\newcommand{\finaltime}{T}
\newcommand{\totaltime}{\finaltime}
\newcommand{\timeDomain}{[0,\finaltime]}
\newcommand{\ntimedof}{{N_T}}

\newcommand{\indicatorSymb}{I}
\newcommand{\indicator}[1]{\indicatorSymb(#1)}
\newcommand{\velocitySymb}{f}
\newcommand{\velocityFluxSymb}{\velocitySymb^\fluxSymb}
\newcommand{\velocitySourceSymb}{\velocitySymb^\sourceSymb}
\newcommand{\velocityDecompSymb}{\decomp f}

\newcommand{\velocityHyperSymb}{\hyper{f}}
\newcommand{\velocityHyper}{\bds{\velocityHyperSymb}}
\newcommand{\velocityHyperArgs}[3]{\velocityHyper(#1,#2;#3)}
\newcommand{\velocityDecompFluxSymb}{\decomp \velocitySymb^\fluxSymb}
\newcommand{\velocityDecompSourceSymb}{\decomp \velocitySymb^\sourceSymb}

\newcommand{\velocity}{\bds{\velocitySymb}}

\newcommand{\velocityArgs}[3]{\velocity(#1,#2;#3)}
\newcommand{\velocityParams}{\velocity(\params)}

\newcommand{\velocityEntry}[1]{\velocitySymb_{#1}}

\newcommand{\velocityOptDummy}{\bds{v}}
\newcommand{\velocityOptDummyTwo}{\bds{w}}
\newcommand{\velocityOptDummyRed}{\hat\velocityOptDummy}
\newcommand{\velocityOptDummyRedOne}{\velocityOptDummyRed_1}
\newcommand{\velocityOptDummyRedTwo}{\velocityOptDummyRed_2}
\newcommand{\feasibleSet}{\mathcal F}
\newcommand{\feasibleSetGal}{\feasibleSet_\mathrm{G}}
\newcommand{\feasibleSetGalArgs}[3]{\feasibleSetGal(#1,#2;#3)}
\newcommand{\feasibleSetLSPG}{\feasibleSet_\mathrm{P}}
\newcommand{\feasibleSetLSPGArgs}[2]{\feasibleSetLSPG^{#1}(#2)}

\newcommand{\leftSing}{\bds{U}_\mathrm{G}}
\newcommand{\Sing}{\bds{\Sigma}_\mathrm{G}}
\newcommand{\rightSing}{\bds{V}_\mathrm{G}}
\newcommand{\nullBasis}{\bds{Z}_\mathrm{G}}
\newcommand{\rightSingFull}{\bar{\bds{V}}_\mathrm{G}}
\newcommand{\nullBasisFull}{\bar{\bds{Z}}_\mathrm{G}}

\newcommand{\leftSingP}{\bds{U}_\mathrm{P}^n}
\newcommand{\leftSingPT}{[\bds{U}_\mathrm{P}^n]^T}
\newcommand{\SingP}{\bds{\Sigma}_\mathrm{P}^n}
\newcommand{\SingPInv}{[\bds{\Sigma}_\mathrm{P}^n]^{-1}}
\newcommand{\rightSingP}{{\bds{V}_\mathrm{P}^n}}
\newcommand{\rightSingPT}{[\bds{V}_\mathrm{P}^n]^T}

\newcommand{\rightSingPFull}{\bar{\bds{V}}_\mathrm{P}^n}
\newcommand{\rightSingPFullT}{[\bar{\bds{V}}_\mathrm{P}^n]^T}
	\newcommand{\nullBasisPFull}{\bar{\bds{Z}}_\mathrm{P}^n}
	\newcommand{\nullBasisPFullT}{[\bar{\bds{Z}}_\mathrm{P}^n]^T}
	\newcommand{\normedQuantityP}{\zeta_\mathrm{P}^n}
	\newcommand{\normedQuantityG}{\zeta_\mathrm{G}}

\newcommand{\SingC}{{\bds{\Sigma}}_\meshMapping}
\newcommand{\leftSingC}{{\bds{U}}_\meshMapping}
\newcommand{\rightSingCFull}{{\bds{V}}_\meshMapping}
\newcommand{\nullBasisCFull}{{\bds{Z}}_\meshMapping}
\newcommand{\pinv}[1]{(#1)^+}

\newcommand{\orthomat}[2]{\RR{#2\times #1}} 

\newcommand{\subspaceRef}{\bar\stateDummy}
\newcommand{\subspaceOrthoBasis}{\bds{Q}}

\newcommand{\velocityFlux}{\bds{\velocitySymb}^{\fluxSymb}}
\newcommand{\velocityFluxEntryArgs}[4]{\velocityFluxSymb_{#1}(#2,#3;#4)}
\newcommand{\velocityFluxVecArgs}[3]{\velocityFlux(#1,#2;#3)}
\newcommand{\velocityFluxHyper}{\hyper{\bds{\velocitySymb}}^{\fluxSymb}}

\newcommand{\velocitySource}{\bds{\velocitySymb}^{\sourceSymb}}
\newcommand{\velocitySourceEntryArgs}[4]{\velocitySourceSymb_{#1}(#2,#3;#4)}
\newcommand{\velocitySourceVecArgs}[3]{\velocitySource(#1,#2;#3)}
\newcommand{\velocitySourceHyper}{\hyper{\bds{\velocitySymb}}^{\sourceSymb}}

\newcommand{\sVal}[1]{\sigma_{#1}}

\newcommand{\velocityDecompSymbol}{\bds{\velocityDecompSymb}}
\newcommand{\velocityDecomp}{\meshMapping\velocity}
\newcommand{\velocityDecompDecomp}{\meshMappingDecomp\velocity}
\newcommand{\velocityDecompRed}{\hat{\decomp{\bds f}}}
\newcommand{\velocityDecompRedArgs}[3]{\velocityDecompRed(#1,#2;#3)}

\newcommand{\velocityDecompArgs}[3]{\velocityDecomp(#1,#2;#3)}

\newcommand{\velocityDecompFlux}{\bds{\velocityDecompSymb}^{\fluxSymb}}
\newcommand{\velocityDecompFluxEntryArgs}[4]{\velocityDecompFluxSymb_{#1}(#2,#3;#4)}
\newcommand{\velocityDecompFluxVecArgs}[3]{\velocityDecompFlux(#1,#2;#3)}

\newcommand{\velocityDecompSource}{\bds{\velocityDecompSymb}^{\sourceSymb}}
\newcommand{\velocityDecompSourceEntryArgs}[4]{\velocityDecompSourceSymb_{#1}(#2,#3;#4)}
\newcommand{\velocityDecompSourceVecArgs}[3]{\velocityDecompSource(#1,#2;#3)}

\newcommand{\faceFluxSymb}{h}
\newcommand{\faceFluxEntryArgs}[4]{\faceFluxSymb_{#1}(#2,#3;#4)}
\newcommand{\faceFluxVec}{\bds{\faceFluxSymb}}
\newcommand{\faceFluxVecArgs}[3]{\faceFluxVec(#1,#2;#3)}
\newcommand{\faceFluxHyperVec}{\hyper{\bds{\faceFluxSymb}}}

\newcommand{\sampleMat}{\bds{P}}

\newcommand{\nsamplesres}{n_{p,\resSymb}}
\newcommand{\nsamplesvelocity}{n_{p,\velocitySymb}}
\newcommand{\nsamplessource}{n_{p,\sourceSymb}}
\newcommand{\nsamplesflux}{n_{p,\faceFluxSymb}}

\newcommand{\polyArg}[4]{I^{#1}_{#2}(#3;#4)}
\newcommand{\polyGenArg}[2]{I(#1;#2)}
\newcommand{\polyArgFew}[2]{I^{#1}_{#2}}
\newcommand{\polyGenArgFew}{I}

\newcommand{\lagrangebasis}[4]{\prod_{j=1, j\neq
#3}^#2\frac{#4-#4^{n+1-j}}{#4^{#1+1-#3}-#4^{#1+1-j}}}

\newcommand{\stateDummy}{\bds{\xi}}
\newcommand{\stateDummyArg}[1]{\stateDummy^{#1}}

\newcommand{\stateOptDummy}{\bds{z}}

\newcommand{\stateOptDummyRed}{\hat\stateOptDummy}

\newcommand{\stateFOM}[1]{\state_\star^{#1}}
\newcommand{\stateGal}[1]{\stateRed_\mathrm{G}^{#1}}
\newcommand{\stateLSPG}[1]{\stateRed_\mathrm{P}^{#1}}

\newcommand{\errorGal}[1]{\delta\stateRed_\mathrm{G}^{#1}}
\newcommand{\errorLSPG}[1]{\delta\stateRed_\mathrm{P}^{#1}}

\newcommand{\f}[2]{\bds{\velocitySymb}(#1,#2)}
\newcommand{\fSimple}[2]{\bds{\velocitySymb}(#1)}

\newcommand{\testBasisAAShortn}{{\bds{\Psi}}^n}
\newcommand{\testBasisAAShortnT}{({\bds{\Psi}}^n)^T}

\newcommand{\testBasisShortnT}{({\bds{\Psi}}^n)^T}
\newcommand{\resPrev}[1]{\bar\res^n[#1{n-k},\ldots,#1{n-1}]}
\newcommand{\resFOMPrev}{\resPrev{\stateFOM}}
\newcommand{\resGalPrev}{\podstate^T\resPrev{\podstate\stateGal}}
\newcommand{\resGalPrevC}{\resPrev{\podstate\stateGal}}
\newcommand{\resGalDiffC}{\delta\bar\res_\mathrm{G}^{n-1}}

\newcommand{\resLSPGPrev}{(\testBasisAAShortnT\podstate)^{-1}\testBasisAAShortnT\resPrev{\podstate\stateLSPG}}
\newcommand{\resLSPGDiff}{\delta\res_\mathrm{P}^{n-1}}
\newcommand{\resLSPGDiffP}{\delta\res_\mathrm{P,\star}^{n-1}}

\newcommand{\lipschitzConstant }{\kappa}
\newcommand{\normtwo }[1]{\|#1\|_2}
\newcommand{\GalProj }{\mathbb V}
\newcommand{\LSPGProj }{\mathbb P^n}
\newcommand{\LSPGProjT }{[\mathbb P^n]^T}
\newcommand{\compositeProj}{(\nullBasisPFull\nullBasisPFullT\LSPGProj+\rightSingPFull\SingPInv\leftSingPT\meshMapping)}
\newcommand{\ProjDiff }{ \Delta^n}
\newcommand{\hCoeff }[1]{h^{#1}}
\newcommand{\hCoeffC }[1]{\bar h^{#1}}

\newcommand{\epsilonCoeff }[2]{\varepsilon_{#1}^{#2}}
\newcommand{\epsilonCoeffP }[2]{\varepsilon_{\mathrm{P},#1}^{#2}}

\newcommand{\gammaCoeff }[2]{\gamma_{#1}^{#2}}
\newcommand{\epsilonCoeffC}[2]{\bar\varepsilon_{#1}^{#2}}
\newcommand{\gammaCoeffC}[2]{\bar\gamma_{#1}^{#2}}

\newcommand{\nConservation}{\ensuremath{n_u}}
\newcommand{\nControlVol}{{N_\Omega}}

\newcommand{\nConstraints}{\ensuremath{\decomp \ndof}}
\newcommand{\nSubdomains}{\ensuremath{{N_{\decomp\Omega}}}}
\newcommand{\nSubdomainsDecomp}{\ensuremath{{N_{\decomp{\decomp\Omega}}}}}
\newcommand{\nsubdomains}{\nSubdomains}

\newcommand{\penaltyParam}{\rho}

\newcommand{\consMapNo}{\mathcal I}
\newcommand{\consMap}[2]{\consMapNo(#1,#2)}
\newcommand{\consMapDecompNo}{\decomp{\consMapNo}}
\newcommand{\consMapDecomp}[2]{\consMapDecompNo(#1,#2)}
\newcommand{\consMapDecompDecompNo}{\decomp{\decomp{\consMapNo}}}
\newcommand{\consMapDecompDecomp}[2]{\consMapDecompDecompNo(#1,#2)}

\newcommand{\faceMapNo}{\mathcal J}
\newcommand{\faceMap}[2]{\faceMapNo(#1,#2)}

\newcommand{\nFaces}{N_\face}

\newcommand{\alignment}{\bds{B}}
\newcommand{\alignmentEntry}[2]{b_{#1,#2}}
\newcommand{\alignmentDecomp}{\decomp{\alignment}}
\newcommand{\alignmentDecompEntry}[2]{\decomp b_{#1,#2}}

\newcommand{\kronecker}[2]{\delta_{#1#2}}

\newcommand{\conservedQuantityNo}{u}
\newcommand{\conservedQuantityVec}{\bds{u}}
\newcommand{\conservedQuantity}[1]{\conservedQuantityNo_{#1}}
\newcommand{\conservedQuantityUnknown}[1]{\tilde\conservedQuantityNo_{#1}}
\newcommand{\conservedQuantityInitNo}{u^0}
\newcommand{\conservedQuantityInit}[1]{\conservedQuantityInitNo_{#1}}

\newcommand{\constant}{a}
\newcommand{\resSymb}{r}
\newcommand{\resEntry}[1]{\resSymb_{#1}}
\newcommand{\res}{\bds{r}}
\newcommand{\resGal}{\res}
\newcommand{\resGalArgs}[4]{\resGal(#1,#3,#2;#4)}
\newcommand{\resGalHyper}{\hyper\resGal}
\newcommand{\resGalHyperArgs}[4]{\resGalHyper(#1,#3,#2;#4)}
\newcommand{\resGalHyperDecomp}{\hyper{\hyper\resGal}}
\newcommand{\resGalHyperDecompArgs}[4]{\resGalHyperDecomp(#1,#3,#2;#4)}
\newcommand{\resTime}[1]{\res^{#1}}
\newcommand{\resArgs}[3]{\resTime{#1}(#2;#3)}
\newcommand{\resEntryTime}[2]{\resSymb_{#1}^{#2}}

\newcommand{\measure}[1]{\mathrm{meas}(#1)}

\newcommand{\resDecomp}{\meshMapping{\res}}
\newcommand{\resDecompDecomp}{\meshMappingDecomp{{\res}}}
\newcommand{\resDecompTime}[1]{\resDecomp^{#1}}
\newcommand{\resDecompDecompTime}[1]{\resDecompDecomp^{#1}}
\newcommand{\resDecompArgs}[3]{\resDecompTime{#1}(#2;#3)}
\newcommand{\resDecompDecompArgs}[3]{\resDecompDecompTime{#1}(#2;#3)}

\newcommand{\hyper}[1]{\tilde{#1}}

\newcommand{\galSymb}{G}
\newcommand{\lagrangianGal}{\mathcal L_\galSymb}
\newcommand{\lagrangianGalArgs}[4]{\lagrangianGal(#1,#2,#3;#4)}
\newcommand{\lspgSymb}{L}
\newcommand{\lagrangianLSPG}{\mathcal L_\lspgSymb}
\newcommand{\lagrangianLSPGTime}[1]{\lagrangianLSPG^{#1}}
\newcommand{\lagrangianLSPGArgs}[4]{\lagrangianLSPGTime{#1}(#2,#3;#4)}

\newcommand{\deltaVelocityGal}{\delta\velocity_\mathrm{G}}
\newcommand{\deltaVelocityGalArgs}[3]{\deltaVelocityGal(#1,#2;#3)}

\newcommand{\resHyper}{\tilde\res}
\newcommand{\resHyperTime}[1]{\resHyper^{#1}}
\newcommand{\resHyperArgs}[3]{\resHyperTime{#1}(#2;#3)}
\newcommand{\resHyperDecomp}{\tilde{\tilde\res}}
\newcommand{\resHyperDecompTime}[1]{\resHyperDecomp^{#1}}
\newcommand{\resHyperDecompArgs}[3]{\resHyperDecompTime{#1}(#2;#3)}
\newcommand{\podres}{\podstate_\resSymb}
\newcommand{\podflux}{\podstate_\faceFluxSymb}
\newcommand{\podsource}{\podstate_\sourceSymb}

\newcommand{\spatialDimension}{d}

\newcommand{\fluxSymb}{g}

\newcommand{\fluxVec}{\bds{\fluxSymb}}
\newcommand{\fluxVecArg}[1]{\fluxVec_{#1}}
\newcommand{\fluxScalar}{{\fluxSymb}}
\newcommand{\fluxScalarArg}[1]{\fluxScalar_{#1}}

\newcommand{\fluxApproxVecArg}[1]{\bds{\fluxSymb}_{#1}^\mathrm{FV}}

\newcommand{\fluxApprox}[2]{\bds{\fluxSymb}^\mathrm{FV}}

\newcommand{\normalVec}{\bds{n}}
\newcommand{\normalVecArg}[1]{\normalVec_{#1}}
\newcommand{\normalVecDecompArg}[1]{\decomp{\normalVec}_{#1}}

\newcommand{\normalVecFace}{\normalVec^{\face}}
\newcommand{\normalVecFaceArg}[1]{\normalVecFace_{#1}}
\newcommand{\normalVecScalar}{{n}}

\newcommand{\podstate}{\bds{\Phi}}
\newcommand{\podstateVec}[1]{\bds\phi_{#1}}

\newcommand{\nstate}{N}
\newcommand{\stateRed}{\hat\state}

\newcommand{\stateRedGalTime}[1]{\stateRed_\mathrm{G}^{#1}}

\newcommand{\stateRedTime}[1]{\stateRed^{#1}}
\newcommand{\stateRedOptDummy}{\hat\stateOptDummy}

\newcommand{\nstateRed}{p}

\newcommand{\linesearchParam}[1]{\eta^{#1}}

\newcommand{\zero}{\bds{0}}
\newcommand{\identity}{\bds{I}}
\newcommand{\lagrangeMultipliersDummy}{\bds{\gamma}}
\newcommand{\lagrangeMultipliers}{\bds{\lambda}}
\newcommand{\lagrangeMultipliersGal}{\lagrangeMultipliers_\text{G}}
\newcommand{\lagrangeMultipliersGalTime}[1]{\lagrangeMultipliersGal^{#1}}
\newcommand{\lagrangeMultipliersGalSolve}{\frac{d\lagrangeMultipliersGal}{dt}}

\newcommand{\lagrangeMultipliersLSPG}{\lagrangeMultipliers_\text{P}}
\newcommand{\lagrangeMultipliersLSPGSolveTime}[1]{\lagrangeMultipliersLSPG^{#1}}

\newcommand{\podvelocity}{\podstate_{\velocitySymb}}
\newcommand{\nresRed}{\nstateRed_{\resSymb}}
\newcommand{\nvelocityRed}{\nstateRed_{\velocitySymb}}
\newcommand{\nsourceRed}{\nstateRed_{\sourceSymb}}
\newcommand{\nfluxRed}{\nstateRed_{\faceFluxSymb}}

\newcommand{\sampleMatres}{\sampleMat_{\resSymb}}
\newcommand{\sampleMatvelocity}{\sampleMat_{\velocitySymb}}
\newcommand{\sampleMatflux}{\sampleMat_{\faceFluxSymb}}
\newcommand{\sampleMatsource}{\sampleMat_{\sourceSymb}}
\newcommand{\volume}[1]{|\Omega_{#1}|}

\newcommand{\controlVolSize}[1]{|{\Omega_{#1}}|}
\newcommand{\subdomainSymb}{\decomp{\Omega}}
\newcommand{\subdomainSize}[1]{|\subdomainSymb_{#1}|}

\newcommand{\meshMapping}{{\decomp{\bds C}}}
\newcommand{\pinvmeshMapping}{{\decomp{\bds C}}^+}
\newcommand{\meshMappingGlobal}{\meshMapping_\text{global}}
\newcommand{\meshMappingDecomp}{\decomp{\decomp{\bds C}}}
\newcommand{\meshMappingEntry}[2]{\decomp c_{#1,#2}}

\newcommand{\volumes}{{\bds V}}                       %
\newcommand{\volumesEntry}[2]{{v}_{#1,#2}}            %
\newcommand{\volumesDecomp}{\decomp{\volumes}}        %
\newcommand{\volumesDecompEntry}[2]{\decomp{v}_{#1,#2}}
\newcommand{\volumesDecompDecomp}{\decomp{\decomp{\volumes}}}
\newcommand{\volumesDecompDecompEntry}[2]{\decomp{\decomp{v}}_{#1,#2}}
\newcommand{\aggregation}{\decomp{\bds{E}}}
\newcommand{\aggregationEntry}[2]{\decomp{e}_{#1,#2}}
\newcommand{\aggregationDecomp}{\decomp{\decomp{\bds{E}}}}
\newcommand{\aggregationDecompEntry}[2]{\decomp{\decomp{e}}_{#1,#2}}
\newcommand{\aggregationDecompCoarseToCoarser}{\decomp{\decomp{\bds{E}}}'}
\newcommand{\aggregationDecompCoarseToCoarserEntry}[2]{\decomp{\decomp{e}}'_{#1,#2}}

\newcommand{\sourceSymb}{s}
\newcommand{\sourceNo}{\sourceSymb}
\newcommand{\sourceEntry}[1]{\sourceNo_{#1}}
\newcommand{\sourceApproxNo}{\sourceSymb^\mathrm{FV}}
\newcommand{\sourceApproxEntry}[1]{\sourceApproxNo_{#1}}

\newcommand{\mesh}{\mathcal M}
\newcommand{\meshDecomp}{\bar\mesh}
\newcommand{\meshDecompGlobal}{{\bar\mesh}_\text{global}}
\newcommand{\meshDecompDecomp}{\bar{\bar\mesh}}

\newcommand{\generalSubdomain}{\omega}

\newcommand{\generalInterface}{\gamma}

\newcommand{\domain}{\Omega}
\newcommand{\controlVolArg}[1]{\domain_{#1}}
\newcommand{\completecontrolVolArg}[1]{\domain_{#1}}
\newcommand{\Interface}{\Gamma}
\newcommand{\InterfaceArg}[1]{\Interface_{#1}}
\newcommand{\face}{e}
\newcommand{\faceArg}[1]{\face_{#1}}
\newcommand{\faceSet}{\mathcal E}
\newcommand{\faceSetArg}[1]{\faceSet_{#1}}

\newcommand{\subdomainArg}[1]{\subdomainSymb_{#1}}
\newcommand{\subdomainDecompArg}[1]{\decomp{\subdomainSymb}_{#1}}
\newcommand{\completesubdomainArg}[1]{\subdomainSymb_{#1}}
\newcommand{\subdomainInterfaceArg}[1]{\decomp\Interface_{#1}}

\newcommand{\subdomainFaceSet}{\bar\faceSet}
\newcommand{\subdomainFaceSetArg}[1]{\subdomainFaceSet_{#1}}

\newcommand{\FOM}{FOM}
\newcommand{\Gal}{Galerkin}
\newcommand{\LSPG}{LSPG}
\newcommand{\LSPGcnstd}{LSPG-FV}
\newcommand{\GNAT}{GNAT}
\newcommand{\GNATcnstd}{GNAT-FV}
\newcommand{\GNATcnstdArg}[1]{GNAT-FV(#1)}

\newcommand{\densitySymb}{\rho}
\newcommand{\uvelocitySymb}{u}
\newcommand{\energydensity}{e}

\newcommand{\crossareaSymb}{A}
\newcommand{\pressureSymb}{p}
\newcommand{\specificheat}{\gamma}
\newcommand{\energypermass}{\epsilon}
\newcommand{\specificgasconstant}{R}
\newcommand{\speedofsoundSymb}{c}
\newcommand{\machSymb}{M}
\newcommand{\temperatureSymb}{T}
\newcommand{\domainLength}{L}
\newcommand{\paramDomainTrain}{\paramDomain_\text{train}}
\newcommand{\ntrain}{n_\text{train}}

\newcommand{\tensor}{\mathcal X}
\newcommand{\snapshots}{\bds{X}}
\newcommand{\tensorUnfold}[1]{\bds{X}_{(#1)}}
\newcommand{\tensorres}{\mathcal R}

\newcommand{\tensorvelocity}{\mathcal F}

\newcommand{\tensorflux}{\mathcal H}

\newcommand{\tensorsource}{\mathcal S}

\newcommand{\paramTrain}[1]{\params^{#1}_\text{train}}
\newcommand{\kmax}[3]{k_\text{max}(#2,\timeArg{#1};#3)}
\newcommand{\leftsingmatstate}{\bds{U}}
\newcommand{\rightsingmatstate}{\bds{V}}
\newcommand{\singvalmatstate}{\bds{\Sigma}}
\newcommand{\leftsingvecstatei}[1]{\bds{u}_{#1}}
\newcommand{\controlVolSet}{\mathcal K}
\newcommand{\subdomainSet}{\bar{\controlVolSet}}

\makeatletter
\def\@author#1{\g@addto@macro\elsauthors{\normalsize%
    \def\baselinestretch{1}%
    \upshape\authorsep#1\unskip\textsuperscript{%
      \ifx\@fnmark\@empty\else\unskip\sep\@fnmark\let\sep=,\fi
      \ifx\@corref\@empty\else\unskip\sep\@corref\let\sep=,\fi
      }%
    \def\authorsep{\unskip,\space}%
    \global\let\@fnmark\@empty
    \global\let\@corref\@empty  
    \global\let\sep\@empty}%
    \@eadauthor={#1}
}
\makeatother

\journal{Journal of Computational Physics}
\begin{document}

\begin{frontmatter}

\title{Conservative model reduction for finite-volume models}

\author{Kevin Carlberg\corref{cor1}\fnref{sandiacor}}
\ead{ktcarlb@sandia.gov}
\ead[url]{sandia.gov/~ktcarlb}
\cortext[cor1]{Corresponding author}
\fntext[sandiacor]{Extreme-scale Data Science and Analytics
Department, Sandia National Laboratories, Livermore, CA 94550.}

\author{Youngsoo Choi\fnref{sandiallnlcor}}
\ead{choi15@llnl.gov}
\ead[url]{people.llnl.gov/choi15}
\fntext[sandiallnlcor]{This work was performed while employed 
in the Extreme-scale Data Science and Analytics
Department, Sandia National Laboratories, Livermore, CA 94550.
 \textit{Current affiliation}: Lawrence Livermore National Laboratories. Lawrence Livermore National Laboratory is operated by Lawrence 
Livermore National Security, LLC, for the U.S.\ Department of Energy, National 
Nuclear Security Administration under Contract DE-AC52-07NA27344.}

\author{Syuzanna Sargsyan\fnref{uwcor}}
\ead{ssusie@uw.edu}
\fntext[uwcor]{This work was performed while pursing a Ph.D.\ at the University
of Washington during employment as a summer intern in the Extreme-scale Data
Science and Analytics Department, Sandia National Laboratories, Livermore, CA
94550.  \textit{Current affiliation}: HERE Technologies}
\address{Sandia National Laboratories}


\begin{abstract}
This work proposes a method for model reduction of finite-volume
models that guarantees the resulting reduced-order model is conservative,
thereby preserving the structure intrinsic to finite-volume discretizations.
The proposed reduced-order models associate with optimization problems
characterized by a minimum-residual objective function and nonlinear equality
constraints that explicitly enforce conservation over subdomains.  Conservative Galerkin projection arises from
formulating this optimization problem at the time-continuous level, while
conservative least-squares Petrov--Galerkin (LSPG) projection associates with
a time-discrete formulation. We equip these approaches with hyper-reduction
techniques in the case of nonlinear flux and source terms, and also provide approaches for handling infeasibility. In addition, we perform analyses that include
deriving conditions under which conservative Galerkin and conservative LSPG
are equivalent, as well as deriving \textit{a posteriori} error bounds. Numerical experiments performed on a
parameterized quasi-1D Euler equation demonstrate the ability of
the proposed method to ensure not only global conservation, but also
significantly lower state-space errors than nonconservative
reduced-order models such as standard Galerkin and LSPG projection.
\noindent 
\end{abstract}

  \begin{keyword}
	nonlinear model reduction \sep  structure preservation\sep finite-volume
	method \sep Galerkin projection \sep least-squares Petrov--Galerkin
	projection \sep conservative schemes


  \end{keyword}

  \end{frontmatter}
\section{Introduction}

The finite-volume method is commonly employed for discretizing systems of
partial differential equations (PDEs) that associate with conservation laws,
especially those in fluid dynamics. Rather than operating on the strong form
of the PDE, the finite-volume method operates on the integral form of the PDE
to numerically enforce conservation over each control volume comprising
the computational mesh. Thus, \textit{conservation} is the primary problem
structure imposed by finite-volume discretizations; this contrasts with other
discretization techniques that aim to preserve other properties, e.g.,
variational principles in the case of the finite-element discretizations.

Unfortunately, the computational burden imposed by high-fidelity finite-volume models is
often prohibitive, as (1) the fine spatiotemporal resolution typically needed
to ensure a verified, validated computational model can lead to extremely
\textit{large-scale models} whose simulations consume months on 
supercomputers, and (2) many engineering problems are \textit{real time} or
\textit{many query} in nature. Such problems require the (parameterized)
computational model to be simulated rapidly either due to a strict
time-to-solution constraint in the case of real-time problems (e.g., model
predictive control) or due to the need for hundreds or thousands of
simulations in the case of many-query problems (e.g., statistical inversion).

Reduced-order models (ROMs) have been developed to mitigate this burden.
These techniques first perform an \textit{offline} stage during which they
execute computationally costly training tasks (e.g., simulating the
high-fidelity model for several parameter instances) to compute a
low-dimensional `trial' basis for the state. Next, these methods execute a
computationally inexpensive \textit{online} stage during which they rapidly
compute approximate solutions for different points in the parameter space by
projection: they compute solutions in the span of the trial basis while
enforcing the high-fidelity model residual to be orthogonal to the subspace
spanned by a
low-dimensional `test' basis. In the presence of nonlinearities, these
techniques also introduce `hyper-reduction' approximations to ensure the cost
of simulating the ROM is independent of the high-fidelity-model dimension.

The most popular model-reduction approach for nonlinear dynamical systems such
as those arising from finite-volume discretizations is Galerkin projection
\cite{sirovich1987tad3,deane1991low,ma2002low}, wherein the test basis is set to be equal to the trial
basis\reviewerB{. The trial basis} is often computed via proper orthogonal decomposition (POD)
\cite{POD}\reviewerB{, but it can also be computed via the reduced-basis method; see Refs.~\cite{haasdonk2008FV,haasdonkExplicit,haasdonk2008reduced}, which apply the classical reduced-basis method to finite-volume problems.} More recently, the least-squares Petrov--Galerkin (LSPG) projection
method \cite{CarlbergGappy,carlbergJCP,carlbergGalDiscOpt} was proposed, which
has been computationally demonstrated to generate accurate and stable responses
for turbulent, compressible flow problems on which Galerkin projection yielded
unstable responses.  Unfortunately, neither Galerkin nor LSPG projection
directly preserves important problem structure related to conservation laws or
finite-volume models.

To address this, alternative projection techniques have been
developed for improving the performance of reduced-order models when applied to conservation laws, particularly those appearing in fluid dynamics. These include stabilizing inner
products applied to finite-difference \cite{rowley2004mrc} and finite-element
discretizations \cite{barone2009stable,kalashnikova2010stability};
introducing dissipation via closure models
\cite{aubry1988dynamics,sirisup2004spectral,Bergmann2009516,wang2012proper,san2013proper}
or numerical dissipation \cite{iollo2000stability}; performing nonlinear
Galerkin projection based on approximate inertial manifolds
\cite{marion1989nonlinear,shen1990long,jolly1991preserving}; including a
pressure-term representation \cite{noack2005need,galletti2004low}; modifying
the POD basis by including many modes (such that dissipative modes are
captured), changing the norm \cite{iollo2000stability}, enabling adaptivity
\cite{Bergmann2009516,carlberg2014adaptive}, or including basis functions that
resolve a range of scales \cite{balajewicz2012stabilization} or respect the
attractor's power balance \cite{balajewicz2013low}; modifying the projection by
adopting a constrained Galerkin \cite{reddy2017constrained,fick2017reduced}\reviewerB{, constrained} Petrov--Galerkin
\cite{Fang2013540}\reviewerB{, or $L^1$-norm minimizing projection \cite{abgrall2015robust}}; \reviewerA{developing approaches tailored to the incompressible Navier--Stokes equations by introducing} stabilizations based on supremizer-enriched
velocity spaces and a pressure Poisson equation 
\cite{stabile2017finite,stabile2017advances} \reviewerA{or by modifying the Galerkin projection 
\cite{lorenzi2016pod}}\reviewerB{; and improving
the ROM's ability to capture shocks
\cite{ohlberger2013nonlinear,gerbeau2014approximated,carlberg2014adaptive,taddei2015reduced}.}
Among these contributions, only a subset is applicable to finite-volume
discretizations. Further, no model-reduction method to date has been developed
to preserve the structure intrinsic to finite-volume
models: conservation.  In particular, none of the above methods ensures that
conservation holds over any subset of the computational domain, which can lead
to spurious growth or dissipation of quantities that should be conserved in
principle.

To this end, this work proposes a novel projection scheme for finite-volume
models that ensures the reduced-order model is conservative over subdomains of
the problem. The approach leverages the minimum-residual formulation of both
Galerkin and least-squares Petrov--Galerkin projection by equipping their
associated optimization problems with (generally nonlinear) equality
constraints that explicitly enforce conservation over subdomains. The resulting \textit{conservative} reduced-order models can be expressed as
the solution to time-dependent saddle-point problems. The approach does not
rely on a particular choice of reduced basis, although the reduced basis can
affect feasibility of the associated optimization problems. New contributions in this work include:
\begin{enumerate} 
	\item Conservative Galerkin (Section \ref{sec:consGal}) and conservative LSPG (Section
		\ref{sec:consLSPG}) projection techniques, which ensure that the reduced-order models
		are conservative over subdomains of the original computational mesh. These
		methods are equipped with
\begin{enumerate} 
\item techniques for handling 
		infeasible constraints (Section \ref{sec:infeasible}), and
	\item hyper-reduction techniques that respect the underlying finite-volume
		discretization to handle nonlinearities in the flux and source terms
		(Section \ref{sec:hyperreduction2}).
\end{enumerate}
	\item Analysis, which includes:
		\begin{enumerate} 
			\item demonstration that conservative Galerkin projection and time
			discretization are commutative (Theorem \ref{sec:commutative}),
			\item sufficient conditions for feasibility of conservative
				Galerkin (Proposition \ref{prop:sufficientFeasGal}) and conservative LSPG (Proposition
				\ref{prop:suffLSPG}) projection,
			\item  conditions under which conservative Galerkin and conservative
				LSPG projection are equivalent (Theorem \ref{thm:equivalent}), and
			\item \textit{a posteriori} bounds (Section \ref{sec:error}) for the error in the 
				quantities conserved over subdomains (Theorem \ref{thm:boundConserved}), in the 
				null space (Lemma \ref{thm:boundNull}) and
				row space (Lemma \ref{thm:boundRow})
				 of the constraints, in the full state (Theorem
				 \ref{thm:localaposteriori}), and in the conserved quantities (Lemma
				 \ref{thm:boundConserved} and Theorem \ref{thm:boundConserved}).
		\end{enumerate}
	\item Numerical experiments on a parameterized quasi-1D Euler equation
		associated with modeling inviscid compressible flow in a 
		converging--diverging nozzle (Section \ref{sec:experiments}). These experiments demonstrate the merits of the
		proposed method and illustrate the importance of ensuring reduced-order
		models are globally conservative.
\end{enumerate}
 We remark that this work was first presented
publically at the ``Recent Developments in Numerical Methods for Model
Reduction'' workshop at the Institut Henri Poincar\'e on  November 10, 2016.

Other works have also explored formulating reduced-order models that
associate with constrained optimization problems.  Zimmermann et al.\ \cite{zimmermann2014reduced} equip
equality `aerodynamic constraints' to ROMs applied to steady-state external flows, where the
constraints associate with matching experimental data or target performance metrics in a design setting.  Recently,
Reddy et al.\ \cite{reddy2017constrained} propose equipping the time-discrete
Galerkin ROM with inequality constraints that enforce solution positivity or a
bound on the gas-void fraction. Relatedly, Fick et al.\ \cite{fick2017reduced}
proposed a modified Galerkin optimization problem applicable to the
incompressible Navier--Stokes equations, where the inequality constraints
associate with bounds on the generalized coordinates; these bounds correspond
to the extreme values of the generalized coordinates arising during the
training simulations.

The remainder of this paper is organized as follows. Section \ref{sec:fvm}
describes finite-volume discretizations of conservation laws (Section
\ref{sec:spatial}) discretized in time with a linear multistep scheme (Section
\ref{sec:timeInt}). Section \ref{sec:ROM} describes the (standard)
nonlinear model-reduction methods of Galerkin (Section
\ref{sec:galerkin}) and LSPG (Section \ref{sec:lspg})
projection, as well as their hyper-reduced variants (Section \ref{sec:hyperreduction}) and interpretations when applied to finite-volume
models (Section \ref{sec:ROMconserv}). Section \ref{sec:proposed} describes the proposed methodology, which is based on
enforcing conservation over decompositions (Section \ref{sec:DD}) of the computational mesh. Here, Section
\ref{sec:consGal} describes the proposed conservative Galerkin projection
technique, Section \ref{sec:consLSPG} describes
the proposed conservative LSPG projection method, Section \ref{sec:infeasible} describes
approaches for handling constraint infeasibility, Section
\ref{sec:hyperreduction2} describes the application of hyper-reduction to the
constraints that respects the underlying finite-volume discretization, and
Section \ref{sec:training} describes briefly how the quantities required for
the proposed ROMs can be constructed from training data. Next, Section \ref{sec:analysis}
performs analysis, including proving sufficient conditions for feasibility (Section
\ref{sec:feasibility}), providing conditions under with the conservative Galerkin and
conservative LSPG models are equivalent (Section \ref{sec:equivalence}), and
deriving local \textit{a posteriori} error analysis (Section \ref{sec:error}). 
Section \ref{sec:experiments} demonstrates the benefits off the proposed method on a
parameterization of the one-dimensional (compressible) Euler equations applied to  a converging--diverging
nozzle. Finally, Section \ref{sec:conclusions} concludes the paper.

In this work,  matrices are denoted by capitalized bold
letters, vectors by lowercase bold letters, and scalars by unbolded letters. The
columns of a matrix $\bds{A}\in\RR{m\times {\timestepit}}$ are denoted by
$\bds{a}_i\in\RR{m}$, $i\innat{{\timestepit}}$ with $\nat{a}\defeq\{1,\ldots,
a\}$ such that $\bds{A}\equiv\left[\bds{a}_1\ \cdots\
\bds{a}_{\timestepit}\right]$. 
The scalar-valued matrix elements are denoted
by $a_{ij}\in\RR{}$ such that  $\bds{a}_j\equiv\left[a_{1j}\ \cdots\
a_{mj}\right]^T$, $j\innat{{\timestepit}}$. A superscript
denotes the value of a variable at that time instance, e.g.,
$\state^\timestepit$ is the value of $\state$ at time $\timestepit\dt$, where $\dt$ is the time
step.
\section{Finite-volume discretization}\label{sec:fvm}
This work considers parameterized systems of conservation laws. In integral
form, the associated governing equations correspond to
\begin{equation}\label{eq:PDE}
\frac{d}{dt}\int_{\generalSubdomain}{\conservedQuantity{i}}(\spaceVar,\timeVar
;\params)\dVol +
\int_{\generalInterface}\fluxVecArg{i}(\spaceVar,\timeVar
;\params)\cdot\normalVec(\spaceVar)\dInterface =
\int_{\generalSubdomain}{\sourceEntry{i}}(\spaceVar,\timeVar
;\params)\dVol,\quad i\innat{\nConservation},\
\forall\generalSubdomain\subseteq\domain,
\end{equation}
which is solved in
time domain $\timeDomain$ with final time $\finaltime\in\RRplus{}$,
and a (parameterized) initial condition denoted by 
$\conservedQuantityInit{i}:\domain\times\paramDomain\rightarrow\RR{}$
such that $\conservedQuantity{i}(\spaceVar,0
;\params) = \conservedQuantityInit{i}(\spaceVar
;\params)$. 
Here, $\generalSubdomain$ with
$\generalInterface\defeq\partial \generalSubdomain$ denotes any subset of the
spatial 
domain of interest
$\spaceDomain\subset\RR{\nspaceDim}$ with $\nspaceDim\leq 3$, whose boundary
is 
$\Interface\defeq\partial \spaceDomain$, 
$\dInterface$ denotes integration with respect to the boundary,
$\conservedQuantity{i}:\spaceDomain\times\timeDomain\times
\paramDomain\rightarrow \RR{}$, $i\innat{\nConservation}$ denotes the
$i$th conserved variable (per unit volume);
$\fluxVecArg{i}:\spaceDomain\times\timeDomain\times
\paramDomain\rightarrow \RR{\nspaceDim}$, $i\innat{\nConservation}$,
denotes the flux associated with the $i$th conserved
variable (per unit area per unit time);
$\normalVec:\generalInterface\rightarrow\RR{\nspaceDim}$
denotes the outward unit normal to $\generalSubdomain$;
$\sourceEntry{i}:\spaceDomain\times\timeDomain\times \paramDomain\rightarrow
\RR{}$, $i\innat{\nConservation}$ denotes the source associated with the
$i$th conserved variable (per unit volume per unit time);
and
 $\paramDomain\subseteq\RR{\nparam}$ denotes the parameter domain.
We assume the domain $\domain$ is independent of the
parameters $\params$ for notational simplicity.

\subsection{Spatial discretization}\label{sec:spatial}
We consider the particular case where the governing equations \eqref{eq:PDE}
have been discretized in space by a finite-volume method. This implies that
the spatial domain has been partitioned into 
a mesh $\mesh$ of
$\nControlVol\in\natNo$
non-overlapping (closed, connected) control volumes  $\controlVolArg{i}\subseteq\domain$, $i\innat{\nControlVol}$ such
that $\spaceDomain = \cup_{i=1}^{\nControlVol}\completecontrolVolArg{i}$,
which
intersect only on their ($d-1$)-dimensional interface, i.e., $\measure{\controlVolArg{i}\cap
\controlVolArg{j}}=0$ for $i\neq j$, where
$\measure{\generalSubdomain}\defeq\int_\generalSubdomain \dVol$,
$\forall \generalSubdomain\subseteq\domain$. 
We define the mesh as
$\mesh \defeq \{\controlVolArg{i}\}_{i=1}^{\nControlVol}$,
and we denote
the
boundary of the $i$th control volume by
$\InterfaceArg{i}\defeq\partial\controlVolArg{i}$. 
The $i$th control-volume boundary is
partitioned into a set of faces\footnote{We note that this is a set of faces
	for $\spatialDimension=3$, faces
for $\spatialDimension=2$, or simply extremities for $\spatialDimension=1$.} denoted
by $\faceSetArg{i}$ such that 
$\InterfaceArg{i} =
\{\spaceVar\, |\, \spaceVar\in \face,\ \forall \face\in\faceSetArg{i},\
i\innat{|\faceSetArg{i}|}\}$. Then the full set of
$\nFaces$ faces within the mesh is $\faceSet\equiv\{\faceArg{i}\}_{i=1}^{\nFaces}\defeq
\cup_{i=1}^{\nControlVol}\faceSetArg{i}$. Applying Eq.~\eqref{eq:PDE} to each
control volume in the mesh yields
\begin{equation}\label{eq:PDEaftermesh}
\frac{d}{dt}\int_{\controlVolArg{j}}{\conservedQuantity{i}}(\spaceVar,\timeVar
;\params)\dVol +
\int_{\InterfaceArg{j}}\fluxVecArg{i}(\spaceVar,\timeVar
;\params)\cdot \normalVecArg{j}(\spaceVar)\dInterface =
\int_{\controlVolArg{j}}{\sourceEntry{i}}(\spaceVar,\timeVar
;\params)\dVol,\quad i\innat\nConservation,\ j\innat\nControlVol,
\end{equation}
where
$\conservedQuantityVec\equiv\left(\conservedQuantity{1},\ldots,\conservedQuantity{\nConservation}\right)$
and $\normalVecArg{j}:\InterfaceArg{j}\rightarrow\RR{\nspaceDim}$
denotes the unit normal to control volume $\controlVolArg{j}$.
Finite-volume schemes complete the spatial discretization by 
introducing a 
state vector
$\state:\timeDomain\times \paramDomain\rightarrow
 \RR{\ndof}$ with $\ndof = \nControlVol\nConservation$ whose elements comprise
 \begin{equation}
\stateEntry{\consMap{i}{j}}(t;\params) =
\frac{1}{\controlVolSize{j}}\int_{\controlVolArg{j}}{\conservedQuantity{i}}(\spaceVar,\timeVar
;\params)\dVol,\quad i\innat\nConservation,\ j\innat\nControlVol, 
\label{eq:stateDefinition}
\end{equation}
where
$\consMapNo:\nat{\nConservation}\times
\nat{\nControlVol}\rightarrow\nat{\ndof}$ denotes a mapping from
conservation-law index and control-volume index to degree of freedom,
and a velocity vector
$\velocity:(\stateDummy,\timeDummy;\paramsDummy)\mapsto\velocityFlux(\stateDummy,\timeDummy;\paramsDummy) + \velocitySource(\stateDummy,\timeDummy;\paramsDummy)$
with $\velocityFlux,\velocitySource:\RR{\ndof}\times\timeDomain\times \paramDomain\rightarrow
 \RR{\ndof}$ whose elements consist of
\begin{align}\label{eq:velocityDefinition}
	\begin{split}
	\velocityFluxEntryArgs{\consMap{i}{j}}{\state}{t}{\params}&=
-\frac{1}{\controlVolSize{j}}\int_{\InterfaceArg{j}}\fluxApproxVecArg{i}(\state;\spaceVar,\timeVar
;\params)\cdot\normalVecArg{j}(\spaceVar)\dInterface \\
\velocitySourceEntryArgs{\consMap{i}{j}}{\state}{t}{\params}&=
\frac{1}{\controlVolSize{j}}\int_{\controlVolArg{j}}{\sourceApproxEntry{i}}(\state;\spaceVar,\timeVar
;\params)\dVol
	\end{split}
\end{align}
for $i\innat\nConservation,\ j\innat\nControlVol$.
Here,
$\fluxApproxVecArg{i}:\RR{\ndof}\times\spaceDomain\times\timeDomain\times
\paramDomain\rightarrow \RR{\nspaceDim}$, $i\innat{\nConservation}$ denotes the approximated (or reconstructed) flux associated
with the $i$th conserved variable (per unit area per unit
time); and
$\sourceApproxEntry{i}:\RR{\ndof}\times\spaceDomain\times\timeDomain\times \paramDomain\rightarrow
\RR{}$, $i\innat{\nConservation}$ denotes the approximated source associated with the
$i$th conserved variable (per unit volume per unit time), which may arise,
e.g., from applying a quadrature rule to evaluate the integral. We
emphasize that both
the approximated flux $\fluxApproxVecArg{i}$ and approximated source 
	$\sourceApproxEntry{i}$
will in general depend on the
entire state vector $\state$, e.g., due to high-order flux reconstructions or
reactions, respectively. 

Substituting
$\int_{\controlVolArg{j}}{\conservedQuantity{i}}(\spaceVar,\timeVar
;\params)\dVol\leftarrow
\controlVolSize{j}\stateEntry{\consMap{i}{j}}(t;\params)
$, $\fluxVecArg{i}\leftarrow \fluxApproxVecArg{i}$, and 
$\sourceEntry{i}\leftarrow\sourceApproxEntry{i}$ in Eq.~\eqref{eq:PDEaftermesh} 
and dividing by $\controlVolSize{j}$
yields
 \begin{equation} \label{eq:fom}
  \frac{d\state}{dt} = \velocityArgs{\state}{\timeVar}{\params},\quad\quad
  \state(0;\params) = \stateInit(\params),
 \end{equation} 
 where $\stateInitEntry{\consMap{i}{j}}(\params)
 \defeq\frac{1}{\volume{i}}\int_{\controlVolArg{j}}\conservedQuantityInit{i}(\spaceVar
;\params)\dVol$.
This is a parameterized system of nonlinear ordinary differential equations
(ODEs) characterizing an initial value problem, which we consider to be our 
full-order model (FOM). 

\begin{remark}[Full-order model ODE: finite-volume interpretation]\label{rem:fomODE}
	From the definitions of the state \eqref{eq:stateDefinition} and
	velocity \eqref{eq:velocityDefinition}, the full-order-model ODE
	residual element $
{d\stateEntry{\consMap{i}{j}}}/{dt} - 
\velocityEntry{\consMap{i}{j}}
$
can be interpreted as the (normalized) \emph{rate of violation of
conservation} in variable
$\conservedQuantity{i}$ in control volume $\controlVolArg{j}$ at time instance
$\timeVar$ under one approximation: the flux and source terms are
approximated using the finite-volume discretization (i.e.,
$\fluxVecArg{i}\leftarrow\fluxApproxVecArg{i}$, and 
$\sourceEntry{i}\leftarrow\sourceApproxEntry{i}$).
\end{remark}

\begin{remark}[Flux velocity from face
	fluxes]\label{rem:fluxVelocityFaceFluxes}
The elements of the flux velocity can be computed
from a vector of face fluxes
$\faceFluxVec:\RR{\ndof}\times\timeDomain\times \paramDomain\rightarrow
 \RR{\ndofFace}$, whose elements are 
\begin{equation} 
	\faceFluxEntryArgs{\faceMap{i}{j}}{\state}{t}{\params} =
	\int_{\faceArg{j}}\fluxApproxVecArg{i}(\state;\spaceVar,\timeVar
	;\params)\cdot\normalVecFaceArg{j}(\spaceVar)\dInterface ,\quad
	i\innat{\nConservation},\
	j\innat{\nFaces},
\end{equation} 
where $\normalVecFaceArg{j}:\faceArg{j}\rightarrow\RR{\nspaceDim}$
denotes the unit normal assigned to face $\faceArg{j}$ (using any 
convention) and 
$\faceMapNo:\nat{\nConservation}\times
\nat{\nFaces}\rightarrow\nat{\ndofFace}$ denotes a mapping from
conservation-law index and face index to degrees of freedom defined on the
faces. This mapping is provided by
\begin{equation} \label{eq:velocityFromFaces}
	\velocityFluxEntryArgs{\consMap{i}{j}}{\state}{t}{\params} =
	\sum_{k\, |\, \faceArg{k}\in\InterfaceArg{j}}
\alignmentEntry{\consMap{i}{j}}{\faceMap{i}{k}}
	\faceFluxEntryArgs{\faceMap{i}{k}}{\state}{t}{\params}\reviewerB{,}
\end{equation} 
where the elements of 
$\alignment\in\RR{\ndof\times \ndofFace}$ 
are 
\begin{align} 
	\alignmentEntry{\consMap{i}{j}}{\faceMap{\ell}{k}}=\begin{cases}
		-\kronecker{i}{\ell}/\controlVolSize{j},\quad &\faceArg{k}\in\InterfaceArg{j};\
		\normalVecArg{j}(\spaceVar) =
		\normalVecFaceArg{k}(\spaceVar),\ 
		\spaceVar\in\faceArg{k}\\
		\kronecker{i}{\ell}/\controlVolSize{j},\quad &\faceArg{k}\in\InterfaceArg{j};\
		\normalVecArg{j}(\spaceVar) =
		-\normalVecFaceArg{k}(\spaceVar),\ 
		\spaceVar\in\faceArg{k}\\
		0,\quad &\text{otherwise,}
	\end{cases}
\end{align} 
where $\kronecker{i}{j}$ denotes the Kronecker delta.
In matrix form, Eq.~\eqref{eq:velocityFromFaces} becomes
\begin{equation} \label{eq:velocityFromFacesMat}
	\velocityFluxVecArgs{\state}{t}{\params} =
	\alignment\faceFluxVecArgs{\state}{t}{\params}\reviewerB{.}
\end{equation} 
This formulation will be exploited in Section \ref{sec:proposed}, where we introduce the
proposed method. 
\end{remark}

The full-order model ODE \eqref{eq:fom} is typically the starting
point for developing reduced-order models for nonlinear dynamical systems. In
this work, we exploit the particular structure underlying the dynamical system
arising from the definitions of the state \eqref{eq:stateDefinition} and
velocity \eqref{eq:velocityDefinition}.

\subsection{Time discretization}\label{sec:timeInt}

A time discretization is required to solve \eqref{eq:fom} numerically. For
simplicity, we
restrict the focus in this work to linear multistep schemes, although other
time integrators could be considered; see, e.g.,
Ref.~\citep{carlbergGalDiscOpt}, which develops LSPG
reduced-order models for  explicit, fully implicit, and diagonally implicit
Runge--Kutta schemes.  Applying a linear $k$-step method to numerically solve
Eq.~\eqref{eq:fom} at a given parameter instance  $\params\in\paramDomain$ can
be written as
\begin{equation}\label{eq:linearMultistepRev}
\sum_{j=0}^k\alpha_j\state^{\timestepit -j} = \dt\sum_{j=0}^k\beta_j\velocityArgs{\state^{\timestepit -j}}{t^{\timestepit -j}}{\params}\reviewerB{,}
\end{equation}
where $\dt\in\RRplus{}$ denotes the time step,
$\state^{k}$ denotes the numerical approximation to $\state(t^k)$, i.e.,
\begin{equation}\label{eq:stateDefinitionDisc}
\stateEntry{\consMap{i}{j}}^k =
\frac{1}{\controlVolSize{j}}\int_{\controlVolArg{j}}{\conservedQuantity{i}^k}(\spaceVar
)\dVol,
\end{equation}
where $\conservedQuantity{i}^k(\spaceVar)$ denotes the 
numerical approximation to $\conservedQuantity{i}(\spaceVar,\timeArg{k})$.
The coefficients $\alpha_j$ and $\beta_j$ define a particular linear multistep
scheme, $\alpha_0\neq 0$ and $\sum_{j=0}^k\alpha_j = 0$ is necessary for
consistency, and
the method is implicit if $\beta_0\neq 0$.
For notational simplicity, we employ a uniform time grid
$\timeArg{k} =
\timeArg{k-1} + \dt$, $k\innat{\ntimedof}$  with $\timeArg{0}=0$ and
$\ntimedof\defeq\finaltime/\dt$.
The fully discrete full-order model, which is sometimes denoted as the FOM O$\Delta$E, 
  is characterized by the following
 system of algebraic equations to be solved at each time instance $\timestepit\in\nat{\ntimedof}$:
 \begin{equation} \label{eq:resLinMultiSolve}
	 \resArgs{\timestepit}{\state^\timestepit}{\params}=0,
 \end{equation}
 where  $\resTime{\timestepit}:\RR{\ndof}\rightarrow\RR{\ndof}$ denotes the 
 linear multistep residual, which is  defined as
 \begin{equation} \label{eq:resLinMulti}
	 \resArgs{\timestepit}{\unknown}{\paramsDummy} \defeq \alpha_0 \unknown -
	 \dt \beta_0 \velocityArgs{\unknown}{t^\timestepit}{\paramsDummy}+
 \sum_{j=1}^k\alpha_j\state^{\timestepit -j}(\paramsDummy) -
 \dt \sum_{j=1}^k\beta_j \velocityArgs{\state^{\timestepit -j}}{t^{\timestepit
 -j}}{\paramsDummy}.
  \end{equation} 
The unknown vector 
$\unknown\in\RR{\ndof}$ 
can be interpreted as
\begin{equation}\label{eq:unknownDefinitionDisc}
\unknownEntry{\consMap{i}{j}} =
\frac{1}{\controlVolSize{j}}\int_{\controlVolArg{j}}{\conservedQuantityUnknown{i}}(\spaceVar
)\dVol,
\end{equation}
and $\conservedQuantityUnknown{i}$ denotes an approximation to the $i$th conserved variable
$\conservedQuantity{i}(\spaceVar,\timeArg{\timestepit})$  when evaluating the residual
\eqref{eq:resLinMulti} at the $n$th time instance.

\noindent \textbf{Adams methods}.  
Adams methods consider the integrated form of Eq.~\eqref{eq:fom}
 \begin{equation} \label{eq:fomInt}
  \state^\timestepit = \state^{\timestepit-1}
+ \int_{\timeArg{\timestepit-1}}^\timeArg{\timestepit}
\velocityArgs{\state}{\timeVar}{ \params}dt,\quad \timestepit\innatSeq{\ntimedof}, 
 \end{equation} 
and apply a polynomial approximation to the integrand.
In particular, the $p$th-order Adams scheme employs coefficients $\alpha_0
= 1$, $\alpha_1 = -1$, and $\alpha_j=0$, $j>1$ and coefficients $\beta_j$ that
associate with a polynomial interpolation of the integrand. In the explicit
($\beta_0=0$)
case, these are Adams--Bashforth methods with
\begin{equation}\label{eq:AdamsBashforthInt}
\dt\sum_{j=1}^k\beta_j\velocityArgs{\state^{n-j}}{\timeArg{n-j}}{\params}=\int_{\timeArg{\timestepit-1}}^\timeArg{\timestepit}
\polyArg{n-1}{k}{\velocityParams}{t} dt,
\end{equation}
where 
$\velocityParams\defeq(\velocityArgs{\stateArg{0}}{\timeArg{0}}{\params},\ldots
\velocityArgs{\stateArg{\ntimedof}}{\timeArg{\ntimedof}}{\params}\reviewerB{)}
$
and the polynomial approximation (in time) of any time-grid-dependent
quantity 
$\stateDummy\defeq(\stateDummyArg{1},\ldots,\stateDummyArg{k})$ using
data at $(t^{n},\ldots, t^{n+1-k})$ (with $k\geq 1$) is 
\begin{equation}
\polyArg{n}{k}{\stateDummy}{t} \defeq
\sum_{i=1}^k\stateDummyArg{n+1-i}\lagrangebasis{n}{k}{i}{t}.
\end{equation}
In the implicit case (with $\beta_0\neq 0$), these are Adams--Moulton
methods with coefficients $\beta_j$ satisfying
\begin{equation}
\dt\sum_{j=0}^k\beta_j\velocity(\state^{n-j},\timeArg{n-j};\params)=\int_{\timeArg{\timestepit-1}}^\timeArg{\timestepit}
\polyArg{n}{k+1}{\velocityParams}{t} dt,
\end{equation}

Thus, the time-discrete residual \eqref{eq:resLinMulti} becomes
\begin{equation}\label{eq:adams}
	\resArgs{\timestepit}{\stateArg{n}}{\params} = \stateArg{n} - \state^{\timestepit-1} -
\int_{\timeArg{\timestepit-1}}^\timeArg{\timestepit}
\polyGenArg{\velocityParams}{t} dt,
\end{equation}
where $\polyGenArgFew = \polyArgFew{n-1}{k}$ in the explicit case and
$\polyGenArgFew = \polyArgFew{n}{k+1}$ in the implicit case.
Substituting the definitions of the time-discrete state
\eqref{eq:stateDefinitionDisc} and velocity \eqref{eq:velocityDefinition} in
\eqref{eq:adams} yields
\begin{align}\label{eq:adamsFV}
\begin{split}
	\resEntry{\consMap{i}{j}}^{\timestepit }(\stateArg{n};\params)= 
&\frac{1}{\controlVolSize{j}}\int_{\controlVolArg{j}}{\conservedQuantity{i}^{n}}(\spaceVar
)\dVol - 
\frac{1}{\controlVolSize{j}}\int_{\controlVolArg{j}}{\conservedQuantity{i}^{n-1}}(\spaceVar
)\dVol\\
&\reviewerB{+}\frac{1}{\controlVolSize{j}}\int_{\timeArg{\timestepit-1}}^{\timeArg{\timestepit}}\int_{\InterfaceArg{j}}
\polyGenArg{\fluxApproxVecArg{i}(\spaceVar;\params)}{t}
\cdot\normalVecArg{j}(\spaceVar)\dInterface d\timeVar
\reviewerB{-}
\frac{1}{\controlVolSize{j}}\int_{\timeArg{\timestepit-1}}^{\timeArg{\timestepit}}\int_{\controlVolArg{j}}\polyGenArg{\sourceApproxEntry{i}(\spaceVar
;\params)}{t}\dVol d\timeVar,
\end{split}
\end{align}
where
\begin{align}
	\begin{split}
		\fluxApproxVecArg{i}(\spaceVar;\params)&\defeq(\fluxApproxVecArg{i}(\stateArg{0};\spaceVar,\timeArg{0}
		;\params),\ldots,
\fluxApproxVecArg{i}(\stateArg{\ntimedof};\spaceVar,\timeArg{\ntimedof} ;\params)
)\\
\sourceApproxEntry{i}(\spaceVar;\params)&\defeq(\sourceApproxEntry{i}(\stateArg{0};\spaceVar,\timeArg{0}
;\params),\ldots,
\sourceApproxEntry{i}(\stateArg{\ntimedof};\spaceVar,\timeArg{\ntimedof} ;\params)
).
\end{split}
\end{align}
\begin{remark}[Full-order model O$\Delta$E: finite-volume interpretation for
	Adams methods]\label{rem:fomODeltaE}
	Eq.\
\eqref{eq:adamsFV}
shows that 
the full-order-model O$\Delta$E
	residual element
$ \resEntry{\consMap{i}{j}}^{\timestepit
} $ in the case of Adams methods can be interpreted as the (normalized)
\emph{violation of conservation} in variable
$\conservedQuantity{i}$ in control volume $\controlVolArg{j}$ over time
interval $[\timeArg{n-1},\timeArg{n}]$ under two approximations: (1) the flux and source terms are
approximated using the finite-volume discretization (i.e.,
$\fluxVecArg{i}\leftarrow\fluxApproxVecArg{i}$, and 
$\sourceEntry{i}\leftarrow\sourceApproxEntry{i}$), and (2) a polynomial
interpolation is used to approximate the integrand for time integration.
\end{remark}

\section{Reduced-order models}\label{sec:ROM}
\reviewerB{During the online stage,} projection-based reduced-order models compute an approximate solution
$\stateApprox\approx \state$ that lies in a low-dimensional affine trial subspace  $\stateApprox(\timeVar;\params)\in\stateInit(\params) + \range{\podstate}$, i.e.,
\begin{equation} \label{eq:stateApprox}
	\stateApprox(\timeVar;\params) = \stateInit(\params) + \podstate\stateRed(\timeVar;\params),
\end{equation} 
where $\podstate\in\RR{\nstate\times \nstateRed}$ is the reduced-basis matrix
of dimension $\nstateRed\leq \nstate$, which we assume without loss of generality satisfies
$\podstate^T\podstate=\identity$, $\stateRed:\timeDomain\times
\paramDomain\rightarrow \RR{\nstateRed}$ denotes the generalized coordinates,
and $\range{\boldsymbol{A}}$ denotes the range of a matrix $\boldsymbol{A}$.
This basis can be computed in a variety of ways \reviewerB{during the offline stage}, e.g., 
eigenmode analysis,
POD \cite{POD}, or the reduced-basis method
\cite{prud2002reliable,rozza2007reduced}. Substituting the approximation
$\state\leftarrow\stateApprox$ into governing equations \eqref{eq:fom} yields
an overdetermined system of $\nstate$ equations in $\nstateRed$ unknowns. To
compute a unique solution, reduced-order models must enforce the residual to
be orthogonal to a $\nstateRed$-dimensional test subspace. Galerkin and LSPG
projection differ in their choices of this subspace; each choice leads to an
approximate solution that exhibits a particular notion of optimality.

\subsection{Galerkin projection}\label{sec:galerkin}

Galerkin projection employs a test subspace of $\range{\podstate}$ and thus
enforces the residual to be orthogonal to $\range{\podstate}$, i.e.,
the Galerkin ODE is
\begin{equation} \label{eq:galODE}
	\frac{d\stateRed}{d\timeVar} = \podstate^T\velocityArgs{\stateInit(\params) +
	\podstate\stateRed}{\timeVar}{\params}, \quad \stateRed(0) = \zero.
\end{equation} 
		Applying a linear multistep scheme to integrate Eq.~\eqref{eq:galODE} in
		time yields the Galerkin O$\Delta$E 
\begin{equation} 
	\podstate^T\resArgs{n}{\stateInit(\params) +
	\podstate\stateRed^n(\params)}{\params}=0.
\end{equation} 
		 
As demonstrated, e.g., in Ref.~\cite{carlbergGalDiscOpt}, Galerkin projection
exhibits continuous optimality if the reduced basis is orthogonal, i.e.,
$\podstate^T\podstate = \identity$, as the Galerkin ROM computes the
approximated velocity that minimizes the $\ell^2$-norm of the FOM ODE residual
\eqref{eq:fom} over $\range{\podstate}$, i.e., 
\begin{equation}\label{eq:galOptimality}
	\frac{d\stateApprox}{d\timeVar}\left(\stateInit(\params) +
	\podstate\stateRed,t;\params\right) =
	\underset{\velocityOptDummy\in\range{\podstate}}{\arg\min}\norm{\resGalArgs{\velocityOptDummy}{\timeVar}{\stateInit
	+ \podstate\stateRed}{\params}}_2
\end{equation}
or equivalently
\begin{equation}\label{eq:galOptimality2}
	\frac{d\stateRed}{d\timeVar}\left(\stateInit(\params) +
	\podstate\stateRed,t;\params\right) =
	\underset{\velocityOptDummyRed\in\RR{\nstateRed}}{\arg\min}\norm{
	\resGalArgs{\podstate\velocityOptDummyRed}{\timeVar}{\stateInit + \podstate\stateRed}{\params}
}_2,
\end{equation}
where
\begin{equation}
	\resGalArgs{\velocityOptDummy}{\timeDummy}{\stateDummy}{\paramsDummy}\defeq
\velocityOptDummy - 
	\velocityArgs{\stateDummy}{\timeDummy}{\paramsDummy}
\end{equation}
denotes the FOM ODE residual.
\begin{remark}[Galerkin ROM ODE: finite-volume
	interpretation]\label{rem:GalODE}
	From the time-continuous optimality of the Galerkin ROM ODE
	\eqref{eq:galOptimality2} and the finite-volume interpretation of the FOM ODE
	in Remark \ref{rem:fomODE}, 
	the Galerkin ROM ODE \eqref{eq:galODE} can be interpreted as minimizing the sum of
	squared
(normalized) \emph{rates of violation of conservation} across
all variables
$\conservedQuantity{i}$, $i\innat{\nConservation}$ and control volumes
$\controlVolArg{j}$, $j\innat{\nControlVol}$ at time instance
$\timeVar$ under one approximation: the flux and source terms are
approximated using the finite-volume discretization (i.e.,
$\fluxVecArg{i}\leftarrow\fluxApproxVecArg{i}$, and 
$\sourceEntry{i}\leftarrow\sourceApproxEntry{i}$).
\end{remark}
\subsection{LSPG projection}\label{sec:lspg}
In contrast, LSPG projection associates with a minimum-residual formulation applied to the (time-discrete) O$\Delta$E \eqref{eq:resLinMultiSolve}, i.e.,
\begin{equation} \label{eq:LSPGODeltaE}
	\stateApproxArg{n}= 
	\underset{\stateOptDummy\in\stateInit(\params) +
	\range{\podstate}}{\arg\min}\norm{\weightingMatrixArgs{\stateOptDummy}{\params}\resArgs{n}{\stateOptDummy}{\params}}_2
\end{equation} 
or equivalently
\begin{equation} \label{eq:LSPGODeltaEred}
	\stateRedTime{n}= 
	\underset{\stateRedOptDummy\in\RR{\nstateRed}}{\arg\min}\norm{\weightingMatrixArgs{\stateInit(\params) +
	\podstate \stateRedOptDummy}{\params}\resArgs{n}{\stateInit(\params) +
\podstate \stateRedOptDummy}{\params}}_2.
\end{equation} 
The necessary optimality conditions for problem \eqref{eq:LSPGODeltaEred}
associate with stationarity of the objective function, i.e., the solution
$\stateRedTime{n}$ satisfies
 \begin{equation} \label{eq:LSPGODeltaEStationarity}
\testBasisArgs{n}{\stateRedTime{n}}{\params}^T\weightingMatrixT\weightingMatrix\resArgs{n}{\stateInit(\params) +
\podstate \stateRedTime{n}}{\params} = \zero,
 \end{equation} 
 where the LSPG test basis
 $\testBasis^n:\RR{\nstateRed}\times\paramDomain\rightarrow\RR{\nstate\times\nstateRed}$ is 
 \begin{align}\label{eq:testBasisDef}
\begin{split}
	\testBasisArgs{n}{\unknownRed}{\paramsDummy}&\defeq
\frac{\partial\resTime{n}}{\partial\unknown}
(\stateInit(\params) +\podstate\unknownRed; \params)\podstate\\
& =\left(\alpha_0\identity +
\beta_0\dt\frac{\partial\velocity}{\partial\reviewerB{\stateDummy}}
(\stateInit(\params) +\podstate\unknownRed,\timeArg{n}; \params)\right)\podstate.
\end{split}
\end{align}
Eq.~\eqref{eq:testBasisDef} reveals that LSPG projection adds the term $
\frac{\beta_0\dt}{\alpha_0}\frac{\partial\velocity}{\partial\reviewerB{\stateDummy}}
(\stateInit(\params) +\podstate\unknownRed,\timeArg{n}; \params)\podstate
$ to the test
basis employed by Galerkin projection.

\begin{remark}[LSPG ROM O$\Delta$E: finite-volume interpretation for Adams
	methods]\label{rem:lspgODeltaE}
From the time-discrete optimality of the LSPG ROM O$\Delta$E
\eqref{eq:LSPGODeltaEred} and the finite-volume interpretation of the FOM
O$\Delta$E for Adams methods in Remark \ref{rem:fomODeltaE}, the LSPG ROM
O$\Delta$E \eqref{eq:LSPGODeltaEred} can be interpreted as minimizing the 
sum of squared
(normalized) \emph{violation of conservation} across all variables
$\conservedQuantity{i}$, $i\innat{\nConservation}$ and control volumes
$\controlVolArg{j}$, $j\innat{\nConservation}$ over time
interval $[\timeArg{n-1},\timeArg{n}]$ under two approximations: (1) the flux and source terms are
approximated using the finite-volume discretization (i.e.,
$\fluxVecArg{i}\leftarrow\fluxApproxVecArg{i}$, and 
$\sourceEntry{i}\leftarrow\sourceApproxEntry{i}$), and (2) a polynomial
interpolation is used to approximate the integrand for time integration.
\end{remark}

\subsection{Hyper-reduction}\label{sec:hyperreduction}
In the case of nonlinear dynamical systems, projection is insufficient to
yield computational savings, as high-dimensional nonlinear quantities
$\resGal$ and $\resTime{n}$ must be repeatedly computed, projected as
$\podstate^T\resGal$ and $\testBasisShortnT\resTime{n}$, and differentiated
(in the case of implicit time integrators) for Galerkin and LSPG ROMs,
respectively. To reduce this computational bottleneck, 
several `hyper-reduction' techniques have been developed that require computing only a
sample of the elements of these nonlinear vector-valued functions.  These
techniques include collocation
\cite{astrid2007mpe,ryckelynck2005phm,LeGresleyThesis}, gappy POD
\cite{sirovichOrigGappy,bos2004als,astrid2007mpe,CarlbergGappy,carlbergJCP},
the empirical interpolation method (EIM)
\cite{barrault2004eim,chaturantabut2010journal,galbally2009non,drohmannEOI,
HAntil_MHeinkenschloss_DCSorensen_2013a}, reduced-order quadrature
\cite{HAntil_SField_RHNochetto_MTiglio_2013}, finite-element subassembly
methods \cite{an2008optimizing,farhat2014dimensional}, and
reduced-basis-sparsification techniques \cite{carlberg2012spd}.

In the present context, hyper-reduction can be achieved by replacing the
residuals appearing in the  objective functions of
\eqref{eq:galOptimality2} and
\eqref{eq:LSPGODeltaEred} by
$\resGalHyper(\approx\resGal)$ and $\resHyperTime{n}(\approx\resTime{n})$,
respectively, such that the hyper-reduced optimization problems for Galerkin
and LSPG projection become
\begin{equation}\label{eq:galOptimality2Hyper}
	\frac{d\stateRed}{d\timeVar}\left(\stateInit(\params) +
	\podstate\stateRed,\timeVar;\params\right) =
	\underset{\velocityOptDummyRed\in\RR{\nstateRed}}{\arg\min}\norm{
	\resGalHyperArgs{\podstate\velocityOptDummyRed}{\timeVar}{\stateInit + \podstate\stateRed}{\params}
}_2,
\end{equation}
and 
\begin{equation} \label{eq:LSPGODeltaEredHyper}
	\stateRedTime{n}(\params)= 
	\underset{\stateRedOptDummy\in\RR{\nstateRed}}{\arg\min}\norm{\weightingMatrixArgs{\stateInit(\params) +
	\podstate \stateRedOptDummy}{\params}\resHyperArgs{n}{\stateInit(\params) +
\podstate \stateRedOptDummy}{\params}}_2,
\end{equation} 
respectively. These residual approximations are typically constructed in one
of two
ways. Later, Section \ref{sec:hyperreduction2} proposes a third
technique tailored to finite-volume discretizations.
\begin{enumerate}[leftmargin=*]
	\item \label{hyper:res}\textbf{Residual hyper-reduction}. This approach amounts to 
	\begin{equation}
		\resGalHyper = \podres\pinv{\sampleMatres\podres}\sampleMatres\resGal,\quad
		\resHyperTime{n} = \podres\pinv{\sampleMatres\podres}\sampleMatres\resTime{n}
	\end{equation}
	in the case of gappy POD hyper-reduction, or
	simply
	\begin{equation}
		\resGalHyper = \sampleMatres^T\sampleMatres\resGal,\quad
		\resHyperTime{n} = \sampleMatres^T\sampleMatres\resTime{n}
	\end{equation}
	in the case of collocation. Here, $\sampleMatres\in\{0,1\}^{\nsamplesres\times
	\ndof}$ denotes a sampling matrix comprising selected rows of the
	$\ndof\times\ndof$ identity matrix, while
	$\podres\in\RRstar{\nstate\times\nresRed}$ denotes a
	$\nresRed(\leq\nstate)$-dimensional reduced-basis
	matrix constructed for the residual, a superscript $+$ denotes the
	Moore--Penrose pseudoinverse, and $\RRstar{m\times n}$ denotes the set of
	full-column rank $m\times n$ matrices (the non-compact Stiefel manifold). This approach has the
	advantage of associating hyper-reduced optimization problems \eqref{eq:galOptimality2Hyper} and \eqref{eq:LSPGODeltaEredHyper} with a
	weighted-norm variant of the original optimization problems \eqref{eq:galOptimality2}
	and\eqref{eq:LSPGODeltaEred}, i.e.,
\begin{equation} \label{eq:LSPGODeltaEredWeight}
	\frac{d\stateRed}{d\timeVar}\left(\stateInit(\params) +
	\podstate\stateRed\right) =
	\underset{\velocityOptDummyRed\in\RR{\nstateRed}}{\arg\min}\norm{
	\weightingMatPrint\resGalArgs{\podstate\velocityOptDummyRed}{\timeVar}{\stateInit + \podstate\stateRed}{\params}
}_2,
	\quad\stateRedTime{n}= 
	\underset{\stateRedOptDummy\in\RR{\nstateRed}}{\arg\min}\norm{\weightingMatPrint\resArgs{n}{\stateInit(\params) +
	\podstate \stateRedOptDummy}{\params}}_2.
\end{equation} 
	where $\weightingMatPrint = \pinv{\sampleMatres\podres}\sampleMatres$ and
	$\weightingMatPrint = \sampleMatres$ in the case of gappy POD and collocation,
	respectively. 
\item \label{hyper:velocity}\textbf{Velocity hyper-reduction}. This approach
	employs an approximated residual constructed from hyper-reduction performed
	on the velocity vector only, i.e.,
	\begin{align}\label{eq:resHyperCont}
		&\resGalHyperArgs{\velocityOptDummy}{\timeDummy}{\stateDummy}{\paramsDummy}
		= 
		\velocityOptDummy	 -
		\velocityHyperArgs{\stateDummy}{\timeDummy}{\paramsDummy}\\
		\label{eq:resHyperDisc}&\resHyperArgs{\timestepit}{\unknown}{\paramsDummy} = 
		\alpha_0 \unknown -
	 \dt \beta_0 \velocityHyperArgs{\unknown}{t^\timestepit}{\paramsDummy}+
 \sum_{j=1}^k\alpha_j\state^{\timestepit -j}(\paramsDummy) -
 \dt \sum_{j=1}^k\beta_j \velocityHyperArgs{\state^{\timestepit -j}}{t^{\timestepit
 -j}}{\paramsDummy}\reviewerB{,}
	\end{align}
	where
	\begin{equation}
		\velocityHyper =
		\podvelocity\pinv{\sampleMatvelocity\podvelocity}\sampleMatvelocity\velocity\quad\text{or}\quad
\velocityHyper =\sampleMatvelocity^T\sampleMatvelocity\velocity
	\end{equation}
	in the case of gappy POD or collocation,
	respectively.
Here, $\sampleMatvelocity\in\{0,1\}^{\nsamplesvelocity\times
	\ndof}$ denotes a sampling matrix comprising selected rows of the
	identity matrix, while
	$\podvelocity\in\RRstar{\nstate\times\nvelocityRed}$ denotes a
	$\nvelocityRed(\leq\nstate)$-dimensional reduced-basis
	matrix constructed for the velocity. This approach has the advantage of
	limiting the hyper-reduction approximation to the nonlinear component of the
	residual. 
 \end{enumerate}
	We note that the gappy POD approximations are equivalent to empirical
	interpolation when the number of samples is equal to the number of
	reduced-basis elements (i.e., $\nsamplesres = \nresRed$, $\nsamplesvelocity
	= \nvelocityRed$), as the
	pseudo-inverse is equal to the inverse and the approximation
	interpolates the nonlinear function at the sampled elements in this case. 
	Further,
	the POD--(D)EIM method \cite{chaturantabut2010journal} corresponds to Galerkin projection with gappy POD velocity
	hyper-reduction and $\nsamplesvelocity = \nvelocityRed$, in which case the hyper-reduced Galerkin ODE becomes
\begin{equation} \label{eq:galODEhyper}
	\frac{d\stateRed}{d\timeVar} =
	\podstate^T\podvelocity(\sampleMatvelocity\podvelocity)^{-1}\sampleMatvelocity\velocityArgs{\stateInit(\params) +
	\podstate\stateRed}{\timeVar}{\params}, \quad \stateRed(0) = \zero.
\end{equation} 
In addition, the GNAT method \cite{CarlbergGappy,carlbergJCP} corresponds to
LSPG projection with gappy POD residual hyper-reduction. In principle, the two
projection techniques and two hyper-reduction approaches above yield four
possible (hyper-reduced) reduced-order models that could be constructed.

\subsection{Lack of conservation}\label{sec:ROMconserv}

Remarks \ref{rem:GalODE} and \ref{rem:lspgODeltaE} demonstrated that Galerkin
and LSPG ROMs minimize the violation of conservation in the case of
finite-volume models in particular senses; Galerkin performs this minimization
at the time-continuous level, while LSPG does so at the time-discrete level.
While this is an attractive property, it does not guarantee that the model is
conservative in any sense: because the minimum value of the objective
functions in Eqs.~\eqref{eq:galOptimality} and \eqref{eq:LSPGODeltaE} may be
non-zero, conservation is generally violated by each of these approaches. We
interpret this as violating the structure intrinsic to
finite-volume models. This provides the motivation for this work: we aim to
develop reduced-order models that ensure the resulting model is conservative
globally and---more generally---over subdomains.

\section{Proposed method}\label{sec:proposed}
\reviewerB{This section describes the proposed method, which equips the
	optimization problems characterizing the online ROM solution with equality constraints
	that explicitly enforce conservation over subdomains. The approach requires
	no modification to the offline stage except when hyper-reduction is applied
	to the nonlinear terms appearing in the constraints.  Section \ref{sec:DD}
	introduces the concept of conservation over subdomains, Section
	\ref{sec:consGal} introduces conservative Galerkin projection, Section
	\ref{sec:consLSPG} describes conservative LSPG projection, Section
	\ref{sec:infeasible} described approaches for handling infeasibility, and
	Section \ref{sec:hyperreduction2} describes hyper-reduction techniques
applicable to objective function and constraint, and Section
\ref{sec:training} describes snapshot-based (offline) training procedures that
may be used for generating the reduced-basis matrices required by the method.}
\subsection{Domain decomposition}\label{sec:DD}

To begin, we decompose the mesh $\mesh$ into subdomains, each of which
comprises the union of control volumes. That is, we define a
decomposed mesh $\meshDecomp$ of $\nSubdomains(\leq\nControlVol)$
subdomains 
$\subdomainArg{i}= \cup_{j\in\controlVolSet\subseteq \nat{\nControlVol}}
\controlVolArg{j}$,
$i\innat{\nSubdomains}$ 
with $\meshDecomp \defeq \{\subdomainArg{i}\}_{i=1}^{\nSubdomains}$. We note
that the subdomains need not be non-overlapping, closed, or connected. Denoting
the boundary of the
$i$th subdomain by $\subdomainInterfaceArg{i}\defeq\partial\subdomainArg{i}$,
we have
$\subdomainInterfaceArg{i} = \{\spaceVar\,
|\, \spaceVar\in \face,\ \forall \face\in\subdomainFaceSetArg{i},\
i\innat{|\subdomainFaceSetArg{i}|}\}
\subseteq\cup_{j=1}^\nControlVol \InterfaceArg{j}$,
$i\innat\nSubdomains$
 with 
$\subdomainFaceSetArg{i}\subseteq\faceSet$ representing the set of faces
belonging to the $i$th subdomain.
We denote the full set of faces within the
decomposed mesh by $\subdomainFaceSet\defeq
\cup_{i=1}^{\nSubdomains}\subdomainFaceSetArg{i}\subseteq\faceSet$. 
Figure \ref{fig:decompmeshes} depicts several decompositions that satisfy the above
conditions.
We emphasize that the subdomains can overlap, their
		union need not correspond to the global domain, and the global domain
		can be considered by employing $\meshDecomp = \meshDecompGlobal$, which is
		characterized by 
	$\nSubdomains=1$ subdomain that corresponds to the global domain, i.e.,
$\subdomainArg{1}=\domain$ and $\subdomainInterfaceArg{1} =
\Interface$, as depicted in Figure \ref{fig:mesh3}.

\begin{figure}[htbp] 
      \begin{center}
				\begin{subfigure}{0.32\textwidth}\includegraphics[width=\textwidth]{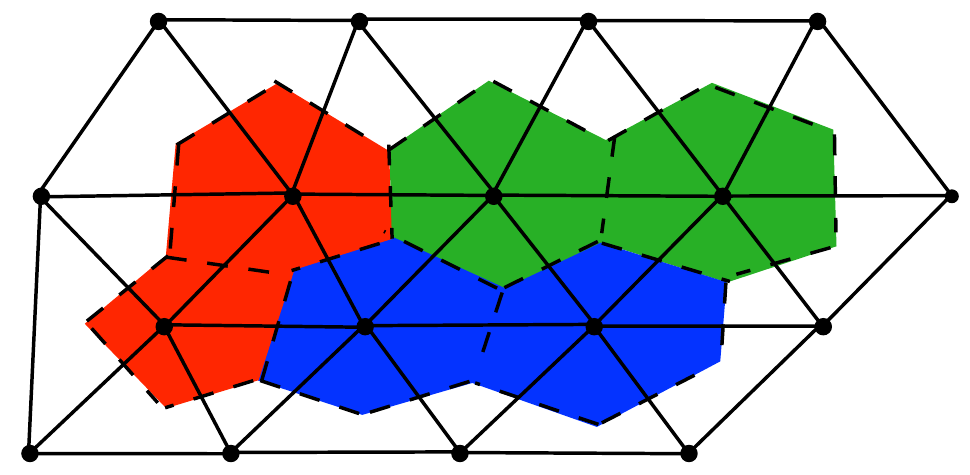}\caption{\centering Decomposed
						mesh $\meshDecomp$ with
				$\nSubdomains=3$}\label{fig:mesh1}\end{subfigure}
				\begin{subfigure}{0.32\textwidth}\includegraphics[width=\textwidth]{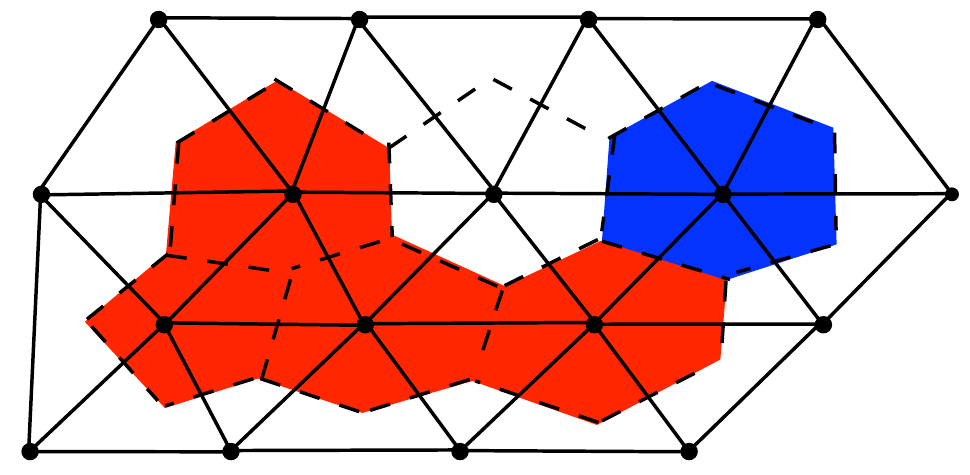}\caption{\centering Decomposed
						mesh $\meshDecomp$ with
				$\nSubdomains=2$}\label{fig:mesh2}\end{subfigure}
				\begin{subfigure}{0.32\textwidth}\includegraphics[width=\textwidth]{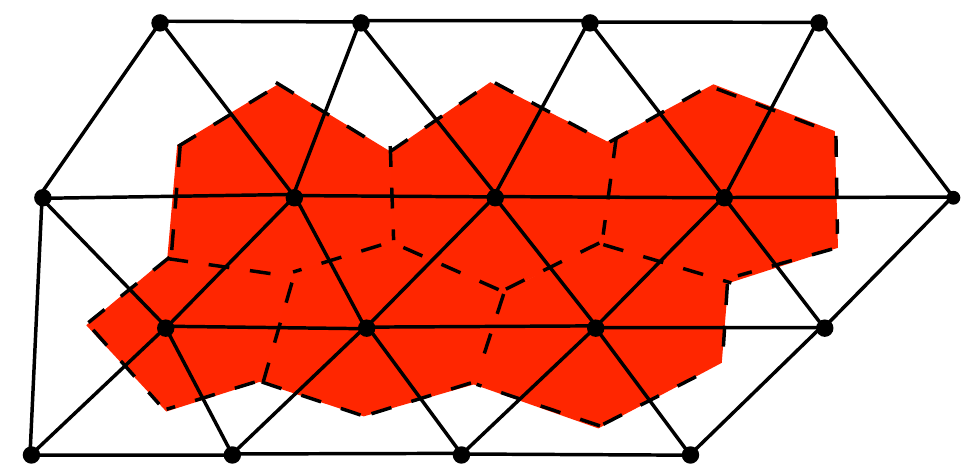}\caption{\centering Decomposed
						mesh $\meshDecompGlobal$ with
				$\nSubdomains=1$,
		$\subdomainArg{1}=\domain$, and $\subdomainInterfaceArg{1} =
\Interface$}\label{fig:mesh3}\end{subfigure}
      \end{center}
			\caption{Examples of decomposed meshes $\meshDecomp$ for a
				vertex-centered finite-volume model. Solid lines denote the primal mesh, and dashed lines the
				control-volume interfaces $\InterfaceArg{j}$ defining the dual mesh, and colors denote separate subdomains
			$\subdomainArg{i}$. }
	\label{fig:decompmeshes}
\end{figure}

Enforcing conservation
\eqref{eq:PDE} on each subdomain in the decomposed mesh yields
\begin{equation}\label{eq:PDEaftermeshDecomp}
\frac{d}{dt}\int_{\subdomainArg{j}}{\conservedQuantity{i}}(\spaceVar,\timeVar
;\params)\dVol +
\int_{\subdomainInterfaceArg{j}}\fluxVecArg{i}(\spaceVar,\timeVar
;\params)\cdot\normalVecDecompArg{j}(\spaceVar)\dInterface =
\int_{\subdomainArg{j}}{\sourceEntry{i}}(\spaceVar,\timeVar
;\params)\dVol,\quad i\innat\nConservation,\ j\innat\nSubdomains,
\end{equation}
where $\normalVecDecompArg{j}:\InterfaceArg{j}\rightarrow\RR{\nspaceDim}$
denotes the unit normal to subdomain $\subdomainArg{j}$\reviewerB{.}
We propose applying a finite-volume discretization to
Eq.~\eqref{eq:PDEaftermeshDecomp} that operates on the decomposed mesh
$\meshDecomp$. That is, we introduce a `decomposed' state vector
$\stateDecomp:\RR{\ndof}\times \timeDomain\times \paramDomain\rightarrow
 \RR{\ndofDecomp}$ with
 $\ndofDecomp = \nSubdomains\nConservation$ and
 elements
 \begin{equation}\label{eq:stateDecompDefinition}
\stateDecompEntry{\consMapDecomp{i}{j}}(\state,t;\params) =
\frac{1}{\subdomainSize{j}}\int_{\subdomainArg{j}}{\conservedQuantity{i}}(\spaceVar,\timeVar
;\params)\dVol,\quad i\innat\nConservation,\ j\innat\nSubdomains,
\end{equation}
where $\consMapDecompNo:\nat{\nConservation}\times
\nat{\nSubdomains}\rightarrow\nat{\ndofDecomp}$ denotes a mapping from
conservation-law index and subdomain index to decomposed degree of freedom.
The decomposed state vector can be computed from the state vector $\state$ as
\begin{equation}
	\stateDecompEntry{\consMapDecomp{i}{j}}(\state,t;\params)
	=\frac{1}{\subdomainSize{j}}\sum_{k\,|\,
		\controlVolArg{k}\subseteq\subdomainArg{j}}
\controlVolSize{k}
\stateEntry{\consMap{i}{k}}(t;\params)
\end{equation}
or equivalently
\begin{equation}
	\stateDecomp(\state) = \meshMapping\state,
\end{equation}
where
$\meshMapping\in\RRplus{\ndofDecomp\times\ndof}$
has elements $\meshMappingEntry{\consMapDecomp{i}{j}}{\consMap{\ell}{k}} =
\controlVolSize{k}/\subdomainSize{j}\kronecker{i}{\ell}\indicator{\controlVolArg{k}\subseteq\subdomainArg{j}}$,
where $\indicatorSymb$ is the indicator function, which evaluates to one if its
argument is true, and zero if its argument is false.
We note that this matrix can be decomposed as
$\meshMapping=\volumesDecomp^{-1}\aggregation\volumes$,
where the elements of
the volumetic matrices $\volumes\in\RR{\ndof\times\ndof}$,
$\volumesDecomp\in\RR{\ndofDecomp\times\ndofDecomp}$ 
and aggregation matrix $\aggregation\in\{0,1\}^{\ndofDecomp\times\ndof}$
comprise
\begin{align} 
	\volumesEntry{\consMap{i}{j}}{\consMap{\ell}{k}}=\kronecker{i}{\ell}\kronecker{j}{k}\controlVolArg{k},\quad
\volumesDecompEntry{\consMapDecomp{i}{j}}{\consMapDecomp{\ell}{k}}=
\kronecker{i}{\ell}\kronecker{j}{k}\subdomainArg{k},\quad\aggregationEntry{\consMapDecomp{i}{j}}{\consMap{\ell}{k}} =
\kronecker{i}{\ell}\indicator{\controlVolArg{k}\subseteq\subdomainArg{j}}.
\end{align} 
Similarly, we write the velocity
vector
$\velocityDecompSymbol:(\stateDummy,\timeDummy;\paramsDummy)\mapsto\velocityDecompFlux(\stateDummy,\timeDummy;\paramsDummy)
+ \velocityDecompSource(\stateDummy,\timeDummy;\paramsDummy)$
with $\velocityDecompFlux,\velocityDecompSource:\RR{\ndof}\times\timeDomain\times \paramDomain\rightarrow
 \RR{\ndofDecomp}$ whose elements consist of
\begin{align}\label{eq:velocityDecompDefinition}
	\velocityDecompFluxEntryArgs{\consMapDecomp{i}{j}}{\state}{t}{\params}&=
-\frac{1}{\subdomainSize{j}}\int_{\subdomainInterfaceArg{j}}\fluxApproxVecArg{i}(\state;\spaceVar,\timeVar
;\params)\cdot\normalVecDecompArg{j}(\spaceVar)\dInterface \\
\velocityDecompSourceEntryArgs{\consMapDecomp{i}{j}}{\state}{t}{\params}&=
\frac{1}{\subdomainSize{j}}\int_{\subdomainArg{j}}{\sourceApproxEntry{i}}(\state;\spaceVar,\timeVar
;\params)\dVol,
\end{align}
for $i\innat\nConservation,\ j\innat\nSubdomains$, which can be computed from
the underlying finite-volume model as
\begin{align}
	\velocityDecompSourceVecArgs{\state}{\timeVar}{\params} =
	\meshMapping\velocitySourceVecArgs{\state}{\timeVar}{\params},\quad
	\velocityDecompFluxVecArgs{\state}{\timeVar}{\params} = \alignmentDecomp\faceFluxVecArgs{\state}{t}{\params}
\end{align}
where the elements of $\alignmentDecomp\in\RR{\ndofDecomp\times
\ndofFace}$ are 
\begin{align} 
	\alignmentDecompEntry{\consMapDecomp{i}{j}}{\faceMap{\ell}{k}}=\begin{cases}
		-\kronecker{i}{\ell}/\subdomainSize{j},\quad &\faceArg{k}\in\subdomainInterfaceArg{j};\
		\normalVecDecompArg{j}(\spaceVar) =
		\normalVecFaceArg{k}(\spaceVar),\ 
		\spaceVar\in\faceArg{k}\\
		\kronecker{i}{\ell}/\subdomainSize{j},\quad &\faceArg{k}\in\subdomainInterfaceArg{j};\
		\normalVecDecompArg{j}(\spaceVar) =
		-\normalVecFaceArg{k}(\spaceVar),\ 
		\spaceVar\in\faceArg{k}\\
		0,\quad &\text{otherwise.}
	\end{cases}
\end{align} 
Critically, noting that $\alignmentDecomp = \meshMapping\alignment$ due to the fact that
neighboring control volumes have outward unit normals of opposite sign along a
shared face, we have
\begin{align}
	\velocityDecompFluxVecArgs{\state}{\timeVar}{\params} &= 
	\meshMapping\velocityFluxVecArgs{\state}{\timeVar}{\params}
\end{align}
such that
\begin{equation}
\velocityDecompSymbol(\state,t;\params)=\velocityDecomp(\state,t;\params).
\end{equation}
Thus, conservation on the decomposed mesh $\meshDecomp$ given an underlying
finite-volume discretization on mesh $\mesh$ can be expressed as
\begin{equation} \label{eq:decompFOM} 
	\meshMapping\frac{d\state}{dt} = \velocityDecomp(\state,t;\params)
\end{equation} 
or equivalently
\begin{equation} \label{eq:decompFOMRes} 
	\meshMapping
	\resGalArgs{\frac{d\state}{dt}}{\timeVar}{\state}{\params} = \zero.
\end{equation} 
Applying a linear multistep scheme to discretize \eqref{eq:decompFOM} in time
yields
 \begin{equation} \label{eq:resLinMultiSolveDecomp}
	 \resDecompArgs{\timestepit}{\state^\timestepit}{\params}=\zero.
 \end{equation}
Note that the decomposed ODE \eqref{eq:decompFOM} and decomposed O$\Delta$E \eqref{eq:resLinMultiSolveDecomp} are underdetermined,
as they comprise $\ndofDecomp$ equations in $\ndof(\geq\ndofDecomp)$ unknowns.

We now demonstrate that conservation that is enforced over a decomposed mesh
automatically leads to conservation over a coarser mesh that embeds the
decomposed mesh.
\begin{theorem}[Conservation over coarser
	decompositions]\label{thm:conservationCoarse}
Define a decomposed mesh $\meshDecompDecomp$
as a decomposition of
 a non-overlapping decomposed mesh $\meshDecomp$ 
 satisfying $\measure{\subdomainArg{i}\cap\subdomainArg{j}} = 0$ for $i\neq j$
such that 
$\subdomainDecompArg{i} =
\cup_{j\in\subdomainSet\subseteq\nat{\nSubdomains}}\subdomainArg{j} $,
$i\innat{\nSubdomainsDecomp}$,
 with $\meshDecompDecomp \defeq
\{\subdomainDecompArg{i}\}_{i=1}^{\nSubdomainsDecomp}$
 and
$\nSubdomainsDecomp\leq\nSubdomains(\leq\nControlVol)$.
Then, satisfaction of time-continuous conservation on $\meshDecomp$ (i.e.,
Eq.~\eqref{eq:decompFOM}) implies
satisfaction of  time-continuous conservation on $\meshDecompDecomp$, i.e.,
\begin{equation} \label{eq:decompDecompFOM} 
	\meshMappingDecomp\frac{d\state}{dt} = \velocityDecompDecomp(\state,t;\params)
\end{equation} 
and 
satisfaction of time-discrete conservation on $\meshDecomp$ 
(i.e., Eq.~\eqref{eq:resLinMultiSolveDecomp}) implies satisfaction of
time-discrete conservation on $\meshDecompDecomp$, i.e.,
 \begin{equation} \label{eq:resLinMultiSolveDecompDecomp}
	 \resDecompDecompArgs{\timestepit}{\state^\timestepit}{\params}=0,
 \end{equation}
where
$\meshMappingDecomp\defeq\volumesDecompDecomp^{-1}\aggregationDecomp\volumes\in\RRplus{\ndofDecompDecomp\times\ndof}$,
the elements of
$\volumesDecompDecomp\in\RR{\ndofDecompDecomp\times\ndofDecompDecomp}$ are
$\volumesDecompDecompEntry{\consMapDecompDecomp{i}{j}}{\consMapDecompDecomp{\ell}{k}}=
\kronecker{i}{\ell}\kronecker{j}{k}\subdomainDecompArg{k}$, and the elements of 
$\aggregationDecomp$ are
$\aggregationDecompEntry{\consMapDecompDecomp{i}{j}}{\consMap{\ell}{k}} =
\kronecker{i}{\ell}\indicator{\controlVolArg{k}\subseteq\subdomainDecompArg{j}}$.
\end{theorem}
\begin{proof}
The conditions 
$\subdomainDecompArg{i} =
\cup_{j\in\subdomainSet\subseteq\nat{\nSubdomains}}\subdomainArg{j} $,
$i\innat{\nSubdomainsDecomp}$;
$\subdomainArg{i}= \cup_{j\in\controlVolSet\subseteq \nat{\nControlVol}}
\controlVolArg{j}$,
$i\innat{\nSubdomains}$; and
$\measure{\subdomainArg{i}\cap\subdomainArg{j}} = 0$ for $i\neq j$
 imply
that
the aggregation operator 
characterizing the mesh $\meshDecompDecomp$ can be
applied in two stages, i.e.,
\begin{equation*}
	\aggregationDecompEntry{\consMapDecompDecomp{i}{j}}{\consMap{\ell}{k}} =
\kronecker{i}{\ell}\indicator{\controlVolArg{k}\subseteq\subdomainDecompArg{j}}
=\kronecker{i}{\ell}\sum_{m=1}^{\nSubdomains}\indicator{\subdomainArg{m}\subseteq\subdomainDecompArg{j}}\indicator{\controlVolArg{k}\subseteq\subdomainArg{m}}
=\kronecker{i}{\ell}\kronecker{\ell}{h}
\sum_{m=1}^{\nSubdomains}\sum_{\ell=1}^{\nConservation}\aggregationDecompCoarseToCoarserEntry{\consMapDecompDecomp{i}{j}}{\consMapDecomp{\ell}{m}}
\aggregationEntry{\consMapDecomp{\ell}{m}}{\consMap{h}{k}}
\end{equation*}
and thus
$\aggregationDecomp = \aggregationDecompCoarseToCoarser\aggregation$, where
the elements of $\aggregationDecompCoarseToCoarser$ are
$\aggregationDecompCoarseToCoarserEntry{\consMapDecompDecomp{i}{j}}{\consMapDecomp{\ell}{k}}=\kronecker{i}{\ell}\indicator{\subdomainArg{k}\subseteq\subdomainDecompArg{j}}$.
Substituting $\aggregationDecomp = \aggregationDecompCoarseToCoarser\aggregation$ in the definition of
$\meshMappingDecomp$ yields
\begin{equation}
\meshMappingDecomp=\volumesDecompDecomp^{-1}\aggregationDecompCoarseToCoarser\aggregation\volumes
=
\volumesDecompDecomp^{-1}\aggregationDecompCoarseToCoarser\volumesDecomp\meshMapping.
\end{equation}
Thus, Eqs.~\eqref{eq:decompDecompFOM} and \eqref{eq:resLinMultiSolveDecompDecomp}
can be rewritten as
\begin{align}
	&\volumesDecompDecomp^{-1}\aggregationDecompCoarseToCoarser\volumesDecomp\meshMapping\frac{d\state}{dt} =
	\volumesDecompDecomp^{-1}\aggregationDecompCoarseToCoarser\volumesDecomp\meshMapping\velocity(\state,t;\params)\\
	&\volumesDecompDecomp^{-1}\aggregationDecompCoarseToCoarser\volumesDecomp\meshMapping\resArgs{\timestepit}{\state^\timestepit}{\params}=0,
 \end{align}
 which are clearly satisfied if Eqs.~\eqref{eq:decompFOM} and
 \eqref{eq:resLinMultiSolveDecomp} are satisfied, respectively.
\end{proof}
\begin{corollary}[Full-order model conservation]
The full-order model satisfies time-continuous and time-discrete conservation
over any decomposed mesh.
\end{corollary}
\begin{proof}
	This corresponds to a particular case of Theorem \ref{thm:conservationCoarse} with
	$\meshDecomp = \mesh$, as any decomposed mesh $\meshDecompDecomp$ must satisfy 
$\subdomainDecompArg{i} =
\cup_{j\in\controlVolSet\subseteq\nat{\nConservation}}\controlVolArg{j} $, $i\innat{\nSubdomainsDecomp}$ and
the original mesh is non-overlapping, i.e., $\measure{\controlVolArg{i}\cap
\controlVolArg{j}}=0$ for $i\neq j$.
\end{proof}
\begin{corollary}[Global conservation]\label{cor:global}
If the decomposed mesh $\meshDecomp$ satisfies $
\cup_{i=1}^{\nSubdomains}\completesubdomainArg{i}=\spaceDomain $
and is non-overlapping, i.e., $\measure{\subdomainArg{i}\cap\subdomainArg{j}} = 0$ for $i\neq j$, then 
satisfaction of time-continuous conservation on $\meshDecomp$ implies
satisfaction of  time-continuous (global) conservation on 
$\meshDecompGlobal \defeq
\{\domain\}$, and
satisfaction of time-discrete conservation on $\meshDecomp$ implies
satisfaction of  time-discrete (global) conservation on $\meshDecompGlobal$.
\end{corollary}
\begin{proof}
	This corresponds to a particular case of Theorem \ref{thm:conservationCoarse}
	with $\meshDecompDecomp=\meshDecompGlobal$, as the required condition
$\subdomainDecompArg{i} =
\cup_{j\in\subdomainSet\subseteq\nat{\nSubdomains}}\subdomainArg{j} $, $i\innat{\nSubdomainsDecomp}$
is satisfied for $\nSubdomainsDecomp=1$, $\subdomainDecompArg{1}=\domain$, and 
$\subdomainSet = \nat{\nSubdomains}$
under the stated assumptions.
\end{proof}

We now derive the proposed conservative Galerkin and conservative LSPG
projection techniques, which equip their associated optimization problems with
equality constraints that enforce conservation over the decomposed mesh
$\meshDecomp$.

\subsection{Conservative Galerkin projection}\label{sec:consGal}
To enable a Galerkin-like projection scheme that enforces conservation, we
equip the unconstrained optimization problems
\eqref{eq:galOptimality}--\eqref{eq:galOptimality2}---which are defined at
the time-continuous level---with equality constraints corresponding to
(time-continuous) conservation \eqref{eq:decompFOMRes} over the decomposed mesh
$\meshDecomp$. The resulting conservative Galerkin solution
$\frac{d\stateApprox}{d\timeVar}\left(\stateInit(\params) +
	\podstate\stateRed,t;\params\right)$
satisfies
\begin{align}\label{eq:galOptimalityStructPres}
	\begin{split}
		\underset{\velocityOptDummy\in\range{\podstate}}{\text{minimize}}&\ 
		\norm{
	\resGalArgs{\velocityOptDummy}{\timeVar}{\stateInit + \podstate\stateRed}{\params}
}_2\\
	\text{subject to}&\
	\meshMapping\resGalArgs{\velocityOptDummy}{\timeVar}{\stateInit +
	\podstate\stateRed}{\params}=\zero.
	\end{split}
\end{align}
Equivalently, the conservative Galerkin generalized coordinates
$\frac{d\stateRed}{d\timeVar}\left(\stateInit(\params) +
	\podstate\stateRed,t;\params\right) $
satisfy
\begin{align}\label{eq:galOptimalityStructPres2}
	\begin{split}
		\underset{\velocityOptDummyRed\in\RR{\nstateRed}}{\text{minimize}}&\ 
		\norm{
	\resGalArgs{\podstate\velocityOptDummyRed}{\timeVar}{\stateInit + \podstate\stateRed}{\params}
}_2\\
	\text{subject to}&\
	\meshMapping\resGalArgs{\podstate\velocityOptDummyRed}{\timeVar}{\stateInit +
	\podstate\stateRed}{\params}=\zero.
	\end{split}
\end{align}

We now provide a finite-volume interpretation of the conservative Galerkin model, define
the feasible set, and provide an algebraic description of the 
solution.
\begin{remark}[Conservative Galerkin ROM ODE: 
	interpretation]\label{rem:GalODEstructPres}
	From Remark \ref{rem:GalODE}, the conservative Galerkin ROM ODE \eqref{eq:galOptimalityStructPres} can be interpreted as minimizing the sum of
	squared
(normalized) \emph{rates of violation of conservation} across
all variables
$\conservedQuantity{i}$, $i\innat{\nConservation}$ and control volumes
$\controlVolArg{j}$, $j\innat{\nControlVol}$ at time instance
$\timeVar$ subject to 
the enforcement of conservation of all variables
$\conservedQuantity{i}$, $i\innat{\nConservation}$ over subdomains
$\subdomainArg{j}$, $j\innat{\nSubdomains}$ at time instance
$\timeVar$
under one approximation: the flux and source terms are
approximated using the finite-volume discretization (i.e.,
$\fluxVecArg{i}\leftarrow\fluxApproxVecArg{i}$, and 
$\sourceEntry{i}\leftarrow\sourceApproxEntry{i}$). 
\end{remark}

\begin{definition}[Feasibility of conservative Galerkin projection]
Problem \eqref{eq:galOptimalityStructPres2} is feasible if the
Galerkin feasible set $\feasibleSetGalArgs{\stateInit(\params) +
\podstate\stateRed}{t}{\params}$, defined as
\begin{equation}
\feasibleSetGalArgs{\stateDummy}{\timeDummy}{\paramsDummy}
\defeq\{\velocityOptDummyTwo\in\RR{\nstateRed}\, |\,
			\meshMapping\resGalArgs{\podstate\velocityOptDummyTwo}{\timeDummy}{\stateDummy}{\paramsDummy}=\zero\},
			\end{equation}
			is non-empty.
\end{definition}

\begin{theorem}
	If Problem \eqref{eq:galOptimalityStructPres2} is feasible, then the solution
	is unique and satisfies the time-dependent saddle-point problem
\begin{equation} \label{eq:galODEStructPresSaddle}
	\begin{bmatrix}
		\identity&\podstate^T\meshMapping^T\\
		\meshMapping\podstate & \zero
	\end{bmatrix}
	\begin{bmatrix}
		\frac{d\stateRed}{dt} \\
		\lagrangeMultipliersGalSolve
	\end{bmatrix}
	=
	\begin{bmatrix}
\podstate^T\velocityArgs{\stateInit(\params) +
	\podstate\stateRed}{\timeVar}{\params}\\
\velocityDecomp(\stateInit(\params) +
	\podstate\stateRed,t;\params)
	\end{bmatrix}, \quad \stateRed(0) = \zero,
\end{equation} 
which can be expressed equivalently as
	\begin{equation}	\label{eq:galODEStructPres}
\frac{d\stateRed}{dt} = 
\podstate^T\velocityArgs{\stateInit(\params) +
	\podstate\stateRed}{\timeVar}{\params} + 
\deltaVelocityGalArgs{\stateInit(\params) +
	\podstate\stateRed}{\timeVar}{\params}
, \quad \stateRed(0) = \zero
	\end{equation}	
	with
	\begin{equation} \label{eq:deltaVelocityGalArgsDef}
\deltaVelocityGalArgs{\stateDummy}{\timeDummy}{\paramsDummy}
	\defeq
	(\meshMapping\podstate)^+
[\velocityDecompArgs{\stateDummy}{\timeDummy}{\paramsDummy}
-\meshMapping\podstate\podstate^T\velocityArgs{\stateDummy}{\timeDummy}{\paramsDummy}]
\end{equation} 
and Lagrange multipliers
\begin{equation} \label{eq:galConsSolAlgOrthoProjLagrange}
\lagrangeMultipliersGalSolve = -(\meshMapping\podstate)
^{+T}(\meshMapping\podstate)^+[\velocityDecompArgs{\stateDummy}{\timeDummy}{\paramsDummy}
-\meshMapping\podstate\podstate^T\velocityArgs{\stateDummy}{\timeDummy}{\paramsDummy}].
\end{equation} 
\end{theorem}
\begin{proof}
The Lagrangian associated with problem
\eqref{eq:galOptimalityStructPres2} can be written as
\begin{equation}
	\lagrangianGalArgs{\velocityOptDummyRed}{\lagrangeMultipliersDummy}{t}{\params} \defeq 
	\frac{1}{2}\norm{\podstate\velocityOptDummyRed - 
	\velocityArgs{\stateInit(\params) +
\podstate\stateRed}{\timeVar}{\params}}_2^2 + \lagrangeMultipliersDummy^T
	\left[\meshMapping\podstate\velocityOptDummyRed - \velocityDecomp(\stateInit(\params) +
	\podstate\stateRed,t;\params)\right].
\end{equation}
We note that problem
\eqref{eq:galOptimalityStructPres2}
corresponds to a convex linear least-squares problem with linear
equality constraints; thus,
$(\frac{d\stateRed}{dt}, \lagrangeMultipliersGalSolve) $ is a unique solution if
and only if it
 satisfies the stationarity conditions
 $$
 \reviewerB{
\frac{\partial\lagrangianGal}{\partial\velocityOptDummyRed}
\left(\frac{d\stateRed}{dt},
	\lagrangeMultipliersGalSolve,\timeVar;\params\right)
= \zero,\quad
\frac{\partial\lagrangianGal}{\partial\lagrangeMultipliersDummy}\left(\frac{d\stateRed}{dt},
\lagrangeMultipliersGalSolve,\timeVar;\params\right)=\zero.}$$
\noindent Noting that $\podstate^T\podstate=\identity$, these conditions are equivalent to
Eq.~\eqref{eq:galODEStructPresSaddle}.
The proof of Eqs.~\eqref{eq:galODEStructPres}--\eqref{eq:galConsSolAlgOrthoProjLagrange} follows the null-space method for solving 
	optimization problems with linear equality constraints.
	Feasibility implies that the feasible set 
	$\feasibleSetGalArgs{\stateInit(\params) +
\podstate\stateRed}{t}{\params}$
	is non-empty, which in turn implies that
	 the second block of
Eqs.~\eqref{eq:galODEStructPresSaddle} is consistent
and
	$\velocityDecompArgs{\stateDummy}{\timeDummy}{\paramsDummy}\in\range{\meshMapping\podstate}$
	or equivalently
	\begin{equation} \label{eq:velocityDecompRed}
\velocityDecompArgs{\stateDummy}{\timeDummy}{\paramsDummy} =
\leftSing\velocityDecompRedArgs{\stateDummy}{\timeDummy}{\paramsDummy} ,
\end{equation} 
with $\velocityDecompRed:
\RR{\ndof}\times\timeDomain\times \paramDomain\rightarrow
\RR{\rank{\meshMapping\podstate}}$,
where 
\begin{equation}\label{eq:meshMappingpodstateSVD}
\meshMapping\podstate = \leftSing\Sing\rightSing^T
\end{equation}
is the singular value decomposition with
$\leftSing\in\orthomat{\rank{\meshMapping\podstate}}{\ndofDecomp}$ and
$\leftSing^T\leftSing=\identity$,
$\Sing\equiv\diag{\sVal{1},\ldots,\sVal{\rank{\meshMapping\podstate}}}$ and 
$\sVal{1}\geq\cdots\geq\sVal{\rank{\meshMapping\podstate}} > 0$, and
$\rightSing\in\orthomat{\rank{\meshMapping\podstate}}{\nstateRed}$ with
$\rightSing^T\rightSing=\identity$. Because
$\range{\rightSing}\oplus\range{\nullBasis}=\RR{\nstateRed}$
with
$\nullBasis\in\orthomat{\left(\nstateRed-\rank{\meshMapping\podstate}\right)}{\nstateRed}$
an orthogonal basis for the null space of $\meshMapping\podstate$, we can
decompose the unknown vector
$\velocityOptDummyRed\in\RR{\nstateRed}$ appearing in optimization problem
\eqref{eq:galOptimalityStructPres2} as
\begin{equation}\label{eq:velocityOptDummyRedDecomp}
\velocityOptDummyRed = \rightSing\velocityOptDummyRedOne +
\nullBasis\velocityOptDummyRedTwo
\end{equation}
with $\velocityOptDummyRedOne\in\RR{\rank{\meshMapping\podstate}}$ and 
$\velocityOptDummyRedTwo\in\RR{\nstateRed-\rank{\meshMapping\podstate}}$.
Substituting Eqs.~\eqref{eq:velocityDecompRed}, \eqref{eq:meshMappingpodstateSVD},
and \eqref{eq:velocityOptDummyRedDecomp} into the constraints of Problem
\eqref{eq:galOptimalityStructPres2} and noting that
$\rightSing^T\nullBasis=\zero$ yields
\begin{align}\label{eq:velocityOptDummyRedOnefirst}
\begin{split} 
	\velocityOptDummyRedOne =\Sing^{-1}\reviewerB{\leftSing^T}
	\velocityDecompArgs{\stateInit(\params) +
	\podstate\stateRed}{t}{\params}.
\end{split} 
\end{align} 
Pre-multiplying \eqref{eq:velocityOptDummyRedOnefirst} by $\rightSing$,
substituting $\pinv{\meshMapping\podstate}=\rightSing\Sing^{-1}\leftSing^T$,
and using \eqref{eq:velocityDecompRed} yields 
\begin{equation} \label{eq:velocityOptDummyRedOne}
\rightSing\velocityOptDummyRedOne = \pinv{\meshMapping\podstate}
\velocityDecompArgs{\stateInit(\params) +
	\podstate\stateRed}{t}{\params}.
	\end{equation} 
Thus, decomposing the solution as
\begin{equation}\label{eq:galSolDecomp} 
\frac{d\stateRed}{d\timeVar} = 
	\rightSing\left[\frac{d\stateRed}{d\timeVar}\right]_1 +
\nullBasis\left[\frac{d\stateRed}{d\timeVar}\right]_2 ,
	  \end{equation} 
		we have 
		\begin{equation} \label{eq:galConsSolAlgOne}
\rightSing\left[\frac{d\stateRed}{d\timeVar}\right]_1\ =\pinv{\meshMapping\podstate}
\velocityDecompArgs{\stateInit(\params) +
	\podstate\stateRed}{t}{\params}.
	  \end{equation} 
	Now, substituting Eqs.~\eqref{eq:velocityOptDummyRedDecomp} with
$\velocityOptDummyRedOne$	defined in
	\eqref{eq:velocityOptDummyRedOne} into Problem
	\eqref{eq:galOptimalityStructPres2} yields an unconstrained optimization
	problem  in $\velocityOptDummyRedTwo$ only, i.e.,
	$\left[\frac{d\stateRed}{d\timeVar}\right]_2$ is the solution to
\begin{align}\label{eq:galOptimalityStructPres2uncon}
	\begin{split}
		\underset{\velocityOptDummyRedTwo\in\RR{\nstateRed-\rank{\meshMapping\podstate}}}{\text{minimize}}&\
		\norm{\podstate[\pinv{\meshMapping\podstate}
\velocityDecompArgs{\stateInit(\params) +
	\podstate\stateRed}{t}{\params}+\nullBasis\velocityOptDummyRedTwo]- 
	\velocityArgs{\stateInit(\params) +
\podstate\stateRed}{\timeVar}{\params}}_2,
	\end{split}
\end{align}
which---using orthogonality of $\podstate\nullBasis$---is simply
\begin{equation} \label{eq:galConsSolAlgTwoFirst}
	\left[\frac{d\stateRed}{d\timeVar}\right]_2 =
	\nullBasis^T\podstate^T[\velocityArgs{\stateInit(\params) +
		\podstate\stateRed}{\timeVar}{\params}-\podstate\pinv{\meshMapping\podstate}
	\velocityDecompArgs{\stateInit(\params) +
	\podstate\stateRed}{t}{\params}].
\end{equation} 
Applying Eqs.\eqref{eq:galConsSolAlgOne}, and
\eqref{eq:galConsSolAlgTwoFirst} to Eq.~\eqref{eq:galSolDecomp} yields
\begin{equation}\label{eq:galConsSolAlgFirst}
\frac{d\stateRed}{d\timeVar} = 
	\pinv{\meshMapping\podstate}
\velocityDecompArgs{\stateInit(\params) +
	\podstate\stateRed}{t}{\params} +
\nullBasis\nullBasis^T[\podstate^T\velocityArgs{\stateInit(\params) +
		\podstate\stateRed}{\timeVar}{\params}-\pinv{\meshMapping\podstate}
	\velocityDecompArgs{\stateInit(\params) +
	\podstate\stateRed}{t}{\params}] ,
\end{equation}

Applying $\nullBasis^T\pinv{\meshMapping\podstate}=\zero$ to 
Eq.~\eqref{eq:galConsSolAlgFirst} and observing that 
$
\nullBasis\nullBasis^T\velocityOptDummyRed = 
(\identity
-\rightSing\rightSing^T)\velocityOptDummyRed
=(\identity
-(\meshMapping\podstate)^{+}\meshMapping\podstate)\velocityOptDummyRed
$ yields Eq.~\eqref{eq:galODEStructPres}.
Finally, Eq.~\eqref{eq:galConsSolAlgOrthoProjLagrange} arises from
substituting \eqref{eq:galODEStructPres} into \eqref{eq:galODEStructPresSaddle}.
\end{proof}

\noindent 
 Comparing Eqs.~\eqref{eq:galODE} and \eqref{eq:galODEStructPres} reveals that
 equipping the Galerkin-ROM optimization problem with equality constraints
 associated with conservation has the effect of modifying the velocity vector
 through the addition of the term $\deltaVelocityGal$ defined in
 Eq.~\eqref{eq:deltaVelocityGalArgsDef}.  Note that
 Eq.~\eqref{eq:galODEStructPres} corresponds to an initial-value problem that
 can be integrated in time, e.g., using a linear multistep method.

We now show that the conservative Galerkin velocity can be expressed
as the orthogonal projection of the standard Galerkin velocity onto the
feasible set.
\begin{corollary}
If Problem \eqref{eq:galOptimalityStructPres} is feasible, then the solution
corresponds to the orthogonal projection of the standard Galerkin velocity
\eqref{eq:galODE} onto the feasible set, i.e.,
\begin{equation} \label{eq:galConsSolAlgOrthoProj}
	\frac{d\stateRed}{d\timeVar}\left(\stateInit(\params) +
	\podstate\stateRed,t;\params\right) =
	\underset{\velocityOptDummy\in\feasibleSetGalArgs{\stateInit(\params) +
\podstate\stateRed}{t}{\params}
}{\arg\min}\norm{\velocityOptDummy-
\podstate^T\velocityArgs{\stateInit(\params) +
	\podstate\stateRed}{\timeVar}{\params}
}_2.
\end{equation} 
\end{corollary}
\begin{proof}
	We first identify the feasible set from Eqs.~\eqref{eq:galSolDecomp} and \eqref{eq:galConsSolAlgOne} as
	\begin{equation}
\feasibleSetGalArgs{\stateInit(\params) +
\podstate\stateRed}{t}{\params}=
	\pinv{\meshMapping\podstate}
\velocityDecompArgs{\stateInit(\params) +
\podstate\stateRed}{t}{\params}+\range{\nullBasis}.
\end{equation}
Noting that the orthogonal projection of a vector $\stateDummy$ onto an affine
subspace $\subspaceRef + \range{\subspaceOrthoBasis}$ with
$\subspaceOrthoBasis$ an orthogonal matrix with more rows than columns is
simply $\subspaceRef +
\subspaceOrthoBasis\subspaceOrthoBasis^T(\stateDummy-\subspaceRef)$, we identify
$\frac{d\stateRed}{d\timeVar} $ as the orthogonal projection of the standard
Galerkin velocity $\podstate^T\velocityArgs{\stateInit(\params) +
		\podstate\stateRed}{\timeVar}{\params}$ onto the affine subspace
		corresponding to the feasible set
		$	
\feasibleSetGalArgs{\stateInit(\params) +
\podstate\stateRed}{t}{\params}
$; Eq.~\eqref{eq:galConsSolAlgOrthoProj}
derives from this result and the optimality property of orthogonal projectors.
\end{proof}

	Of course, numerically solving the conservative Galerkin ROM ODE, requires introducing a
	time integrator. Applying a linear multistep scheme to solve
	Eq.~\eqref{eq:galODEStructPresSaddle} characterizing the
	conservative Galerkin ROM ODE yields at time instance $n$
	yields the conservative Galerkin ROM O$\Delta$E
	\begin{equation}\label{eq:consGalTimeDisc} 
	\begin{alignedat} {2}
	&\sum_{j=0}^k\alpha_j\stateRedTime{\timestepit-j} +
	\sum_{j=0}^k\alpha_j\podstate^T\meshMapping^T\lagrangeMultipliersGalTime{\timestepit-j}
	&&= 
\dt\sum_{j=0}^k\beta_j\podstate^T\velocityArgs{\stateInit(\params) +
\podstate\stateRedTime{\timestepit-j}}{\timeVar}{\params}\\
&\sum_{j=0}^k\alpha_j\meshMapping\podstate\stateRedTime{\timestepit-j} &&=
\dt\sum_{j=0}^k\beta_j\velocityDecompArgs{\stateInit(\params) +
\podstate\stateRedTime{\timestepit-j}}{\timeVar}{\params}.
\end{alignedat} 
\end{equation} 
We now demonstrate that conservative Galerkin projection and time discretization are
commutative.
\begin{theorem}[Commutativity of conservative Galerkin projection and time
	discretization]\label{sec:commutative}
 \ \\ Conservative Galerkin projection is equivalent to 
computing an
approximate solution
$(\stateApprox(\timeVar;\params),\lagrangeMultipliersGal(\timeVar;\params))\in
\stateInit(\params) + \range{\podstate} \times \RR{\ndofDecomp}$
 via Galerkin
projection applied to the system 
\begin{equation}\label{eq:fullJointODE}
	\begin{bmatrix}
		\identity & \meshMapping^T\\
		\meshMapping & \zero
	\end{bmatrix}
	\begin{bmatrix}
		\frac{d\state}{dt}\\
\lagrangeMultipliersGalSolve
	\end{bmatrix}
	=
	\begin{bmatrix}
	\velocityArgs{\state}{\timeVar}{\params}\\
\velocityDecomp(\state,t;\params)
	\end{bmatrix}\reviewerB{.}
\end{equation}
Further, performing conservative Galerkin projection on
Eq.~\eqref{eq:fullJointODE} and subsequently applying time discretization
yields the same model as first applying time discretization on
Eq.~\eqref{eq:fullJointODE} and subsequently performing conservative Galerkin
projection.
\end{theorem}
\begin{proof}
	The first part of the theorem can be derived by noticing that substituting
	Eq.~\eqref{eq:stateApprox} in \eqref{eq:fullJointODE} and premultiplying
	by $\begin{bmatrix}
		 \podstate & \zero\\
		 \zero & \identity
	\end{bmatrix}$ yields the conservative Galerkin saddle-point
	system \eqref{eq:galODEStructPresSaddle}. Then, applying a linear multistep scheme to solve
	Eq.~\eqref{eq:galODEStructPresSaddle} 
	yields the conservative Galerkin ROM O$\Delta$E
\eqref{eq:consGalTimeDisc} above.
Now, applying a linear multistep scheme to integrate \eqref{eq:fullJointODE} in time yields
\begin{equation} \label{eq:consGalTimeDiscFull}
	\begin{alignedat} {2}
	&\sum_{j=0}^k\alpha_j\stateArg{\timestepit-j} +
	\sum_{j=0}^k\alpha_j\meshMapping^T\lagrangeMultipliersGalTime{\timestepit-j}
	&&= 
\dt\sum_{j=0}^k\beta_j\velocityArgs{
\stateArg{\timestepit-j}}{\timeVar}{\params}\\
&\sum_{j=0}^k\alpha_j\meshMapping\stateArg{\timestepit-j} &&=
\dt\sum_{j=0}^k\beta_j\velocityDecompArgs{
\stateArg{\timestepit-j}}{\timeVar}{\params}.
\end{alignedat} 
\end{equation} 
Because applying conservative Galerkin projection to Eq.~\eqref{eq:consGalTimeDiscFull} yields
Eq.~\eqref{eq:consGalTimeDisc}, we conclude that conservative Galerkin projection and time
discretization are commutative.
\end{proof}

\subsection{Conservative LSPG projection}\label{sec:consLSPG}
Analogously to the procedure employed to derive the conservative Galerkin ROM,
we now equip the
unconstrained optimization problem
\eqref{eq:LSPGODeltaE}--\eqref{eq:LSPGODeltaEred}---which is defined at
the time-discrete level---with equality constraints corresponding to
(time-discrete) conservation \eqref{eq:resLinMultiSolveDecomp} over the decomposed mesh
$\meshDecomp$. The resulting conservative LSPG solution
$\stateApproxArg{n}$ satisfies
\begin{align} \label{eq:LSPGODeltaEStructPres}
	\begin{split}
\underset{\stateOptDummy\in\stateInit(\params) +
	\range{\podstate}}{\text{minimize}}&\
	\norm{\weightingMatrixArgs{\stateOptDummy}{\params}\resArgs{n}{\stateOptDummy}{\params}}_2\\
	\text{subject to}&\ \resDecompArgs{n}{\stateOptDummy}{\params}=0.
	\end{split}
\end{align} 
Equivalently, the conservative LSPG generalized coordinates $\stateRedTime{n}$
satisfy
\begin{align} 
	\label{eq:LSPGODeltaEStructPresred}
	\begin{split}
		\underset{\stateOptDummyRed\in\RR{\nstateRed}
	}{\text{minimize}}&\
			\norm{\weightingMatrixArgs{\stateInit(\params) +
	\podstate \stateRedOptDummy}{\params}\resArgs{n}{\stateInit(\params) +
\podstate \stateRedOptDummy}{\params}}_2\\
	\text{subject to}&\ \resDecompArgs{n}{\stateInit(\params) +
	\podstate \stateRedOptDummy}{\params}=0.
	\end{split}
\end{align} 
We now provide a finite-volume interpretation of the proposed model, define
the feasible set, and provide an algebraic description of the 
solution.
\begin{remark}[Conservative LSPG ROM O$\Delta$E: interpretation for Adams
	methods]\label{rem:lspgODeltaEstructPres}
	From Remark \ref{rem:lspgODeltaE}, the conservative LSPG ROM
O$\Delta$E
\eqref{eq:LSPGODeltaEStructPres} can be interpreted as minimizing the 
sum of squared
(normalized) \emph{violation of conservation} across all variables
$\conservedQuantity{i}$, $i\innat{\nConservation}$ and control volumes
$\controlVolArg{j}$, $j\innat{\nConservation}$ over time
interval $[\timeArg{n-1},\timeArg{n}]$ 
subject to the enforcement of conservation 
of all variables
$\conservedQuantity{i}$, $i\innat{\nConservation}$ over subdomains
$\subdomainArg{j}$, $j\innat{\nSubdomains}$ and time
interval $[\timeArg{n-1},\timeArg{n}]$
under two approximations: (1) the flux and source terms are
approximated using the finite-volume discretization (i.e.,
$\fluxVecArg{i}\leftarrow\fluxApproxVecArg{i}$, and 
$\sourceEntry{i}\leftarrow\sourceApproxEntry{i}$), and (2) a polynomial
interpolation is used to approximate the integrand for time integration.
\end{remark}

\begin{definition}[Feasibility of conservative LSPG projection]
Problem \eqref{eq:LSPGODeltaEStructPresred} is feasible if the
LSPG feasible set $\feasibleSetLSPGArgs{n}{\params}$, defined as
\begin{equation}
\feasibleSetLSPGArgs{n}{\paramsDummy}
\defeq
\{\velocityOptDummyTwo\in\RR{\nstateRed}\,|\,\resDecompArgs{n}{\stateInit(\params) +
	\podstate \velocityOptDummyTwo}{\params}=0 \},
			\end{equation}
			is non-empty.
\end{definition}

\begin{proposition}
	If Problem \eqref{eq:LSPGODeltaEStructPresred} is feasible, then a solution
	exists 
and satisfies the nonlinear saddle-point problem
\begin{align}\label{eq:LSPGKKT}
	\begin{split}
		\testBasisArgs{n}{\stateRedTime{n}}{\params}^T\left[\weightingMatrixT\weightingMatrix\resArgs{n}{\stateInit(\params) +
\podstate \stateRedTime{n}}{\params}
+\meshMapping^T\lagrangeMultipliersLSPGSolveTime{n}\right]
	&=\zero\\
\resDecompArgs{n}{\stateInit(\params) +
	\podstate \stateRedTime{n}}{\params}
	&=\zero,
	\end{split}
\end{align}
where $\lagrangeMultipliersLSPGSolveTime{n}\in\RR{\ndofDecomp}$ denote
Lagrange multipliers.
\end{proposition}
\begin{proof}
Defining the Lagrangian associated with problem  \eqref{eq:LSPGODeltaEStructPres} as 
\begin{equation}
	\lagrangianLSPGArgs{n}{\stateOptDummyRed}{\lagrangeMultipliersDummy}{\params}\defeq
	\frac{1}{2}\norm{\weightingMatrixArgs{\stateInit(\params) +
\podstate\stateOptDummyRed}{\params}\resArgs{n}{\stateInit(\params) +
\podstate\stateOptDummyRed}{\params}}_2^2 + \lagrangeMultipliersDummy^T\resDecompArgs{n}{\stateInit(\params) + \podstate\stateOptDummyRed}{\params},
\end{equation}
the solution $(\stateRedTime{n},\lagrangeMultipliersLSPGSolveTime{n})$ satisfies the
first-order necessary optimality conditions associated with problem
\eqref{eq:LSPGODeltaEStructPres}, i.e.,
$
\partial\lagrangianLSPGTime{n}/\partial\stateOptDummyRed(\stateRedTime{n},\lagrangeMultipliersLSPGSolveTime{n};\params) = \zero$ and 
${\partial\lagrangianLSPGTime{n}}/{\partial\lagrangeMultipliersDummy}(\stateRedTime{n},\lagrangeMultipliersLSPGSolveTime{n};\params) = \zero
$, which---using the definition of the test basis in Eq.~\eqref{eq:testBasisDef}---are equivalent to Eqs.~\eqref{eq:LSPGKKT}.
\end{proof}

Any appropriate optimization algorithm could be applied to  solve minimization
problem \eqref{eq:LSPGODeltaEStructPres} characterizing the
conservative LSPG ROM at each time instance. In this work, we
propose solving problem \eqref{eq:LSPGODeltaEStructPres} using the sequential
quadratic programming (SQP) method with the
Gauss--Newton Hessian approximation. This amounts to applying Newton's
method (with globalization) to the first-order necessary optimality conditions
\eqref{eq:LSPGKKT} and neglecting the term involving differentiation of the
test basis $\testBasisArgs{n}{\stateRedTime{n}}{\params}$. After choosing an
initial guess $\stateRedTime{n(0)}$, this approach leads to the following 
iterations for $k=0,\ldots, K$ 
\begin{align} \label{eq:LSPGODeltaEStructPresSaddle}
\begin{split}
	&\begin{bmatrix}
		\testBasisArgs{n}{\stateRedTime{n(k)}}{\params}^T\weightingMatrixT\weightingMatrix
		\testBasisArgs{n}{\stateRedTime{n(k)}}{\params}
		&\testBasisArgs{n}{\stateRedTime{n(k)}}{\params}^T\meshMapping^T\\
		\meshMapping\testBasisArgs{n}{\stateRedTime{n(k)}}{\params} & \zero
	\end{bmatrix}
	\begin{bmatrix}
		\delta\stateRedTime{n(k)} \\
		\delta\lagrangeMultipliersLSPGSolveTime{n(k)}
	\end{bmatrix}\\
	&\qquad=
	-
	\begin{bmatrix}
		\testBasisArgs{n}{\stateRedTime{n(k)}}{\params}^T\left(\weightingMatrixT\weightingMatrix
			\resArgs{n}{\stateInit(\params) +
		\podstate \stateRedTime{n(k)}}{\params}
	+\meshMapping^T\lagrangeMultipliersLSPGSolveTime{n(k)}\right)\\
\resDecompArgs{n}{\stateInit(\params) +
	\podstate \stateRedTime{n(k)}}{\params}
	\end{bmatrix}.
\end{split} 
\end{align}
\begin{equation}
\begin{bmatrix}
	\stateRedTime{n(k+1)}\\
	\lagrangeMultipliersLSPGSolveTime{n(k+1)}
\end{bmatrix}
= 
\begin{bmatrix}
	\stateRedTime{n(k)}\\
	\lagrangeMultipliersLSPGSolveTime{n(k)}
\end{bmatrix}
 + 
\linesearchParam{n(k)}
\begin{bmatrix}
	\delta\stateRedTime{n(k)}\\
	\delta\lagrangeMultipliersLSPGSolveTime{n(k)}
\end{bmatrix},
\end{equation} 
where
$\linesearchParam{n(k)}\in\RR{}$ is the step length that can be chosen, e.g., to
satisfy the strong Wolfe conditions to ensure global convergence to a local
solution of \eqref{eq:LSPGODeltaEStructPresred}.
		 
\subsection{Handling infeasibility}\label{sec:infeasible}

Of course, the optimization problems characterizing conservative
Galerkin projection (i.e., problems
\eqref{eq:galOptimalityStructPres}--\eqref{eq:galOptimalityStructPres2}) and
conservative LSPG projection  (i.e., problems
\eqref{eq:LSPGODeltaEStructPres}--\eqref{eq:LSPGODeltaEStructPresred}) may not
be feasible for arbitrary decomposed meshes $\meshDecomp$ and reduced basis
matrices $\podstate$. For example, if the decomposed mesh corresponds to the
original mesh (i.e., $\meshDecomp=\mesh$) and the reduced basis is
low-dimensional (i.e., $\nstateRed\ll\ndof$), then the constraints in these
problems correspond to \textit{exactly} satisfying the full-order-model
equations over a low-dimensional subspace; it is likely impossible to do so.

In practice, infeasibility of a given model can be detected by identifying that
the feasible set is empty. In the case of conservative Galerkin
projection, this occurs at a given time instance
$\timeVar^\timestepit$ and parameter instance $\params$ if $\feasibleSetGalArgs{\stateInit(\params) +
\podstate\stateRedTime{n}}{\timeVar^n}{\params}=\emptyset$, which implies that 
$\velocityDecompArgs{\stateInit(\params) +
\podstate\stateRedTime{n}}{\timeVar^n}{\params}\not\in\range{\meshMapping\podstate}$.
Similarly, in the case of conservation-preserving LSPG projection,
infeasibility is detected if a given time instance
$\timeVar^\timestepit$ and parameter instance $\params$ yield $
\feasibleSetLSPGArgs{n}{\params}
=\emptyset$, which implies that 
no value of $\velocityOptDummyTwo$ can set $\resDecompArgs{n}{\stateInit(\params) +
	\podstate \velocityOptDummyTwo}{\params}$ to zero.  
We now describe two approaches for handling the case where infeasibility is
detected.
\vspace{0.15cm}

\begin{enumerate}[leftmargin=*]
	\item\label{infeas:coarsen} \textbf{Coarsen the decomposed mesh}.  First, the number of
constraints can be reduced by coarsening the decomposed mesh, i.e., replace
$\meshDecomp$ by another decomposed mesh characterized by fewer subdomains
$\nsubdomains$. As this reduces the number of constraints, the likelihood of
feasibility increases, although feasibility remains not guaranteed. This
procedure can be repeated until the decomposed mesh leads to a nonempty
feasible set or a decomposed mesh characterized by one subdomain ($\nsubdomains=1$) is infeasible. 

If a decomposed mesh leading to feasibility
is constructed via coarsening, the conservative reduced-order model
can be redefined  using the new decomposed mesh and the reduced-order-model
simulation can be either (1) reinitialized and restarted from $t = 0$, or (2)
resumed from the time instance $t^n$ where infeasibility was detected. The
first approach facilitates analysis, as the reduced-order-model trajectory
association with a fixed decomposed mesh, while the latter precludes the need
to re-simulate any part of the time interval. Further, if the new decomposed
mesh is a decomposition of the previous
decomposed mesh, and the previous decomposed mesh is non-overlapping, then conservation
over the new decomposed mesh holds over the first part of the time interval
(see Theorem \ref{thm:conservationCoarse}).
We note that this approach is not guaranteed to ensure feasibility, as it is
possible for infeasibility to exist even in the case of  $\nsubdomains=1$.
\vspace{0.15cm}

\item \label{infeas:penalty}\textbf{Penalty formulation}.
Alternatively, infeasibility can be addressed by including the constraints in
the objective function via penalization. 
In the case of conservative Galerkin projection, problem \eqref{eq:galOptimalityStructPres2}
is reformulated as
\begin{align} \label{eq:GalerkinPenalty}
	\begin{split}
		\underset{\velocityOptDummyRed\in\RR{\nstateRed}}{\text{minimize}}&\ 
		\norm{
	\resGalArgs{\podstate\velocityOptDummyRed}{\timeVar}{\stateInit + \podstate\stateRed}{\params}
}_2^2+\penaltyParam
\norm{\meshMapping\resGalArgs{\podstate\velocityOptDummyRed}{\timeVar}{\stateInit +
\podstate\stateRed}{\params}}_2^2,
	\end{split}
\end{align} 
while the conservative LSPG projection problem \eqref{eq:LSPGODeltaEStructPresred} is reformulated as
\begin{align} \label{eq:LSPGABPenalty}
	\begin{split}
		\underset{\stateOptDummyRed\in\RR{\nstateRed}
	}{\text{minimize}}&\
			\norm{\weightingMatrixArgs{\stateInit(\params) +
	\podstate \stateRedOptDummy}{\params}\resArgs{n}{\stateInit(\params) +
\podstate \stateRedOptDummy}{\params}}_2^2+\penaltyParam
\norm{\resDecompArgs{n}{\stateInit(\params) +
\podstate \stateRedOptDummy}{\params}}_2^2,
	\end{split}
\end{align} 
where $\penaltyParam\in\RRplus{}$ is a penalty parameter. This
approach does not enforce conservation over any subdomain of the problem.
\end{enumerate}

\subsection{Hyper-reduction}\label{sec:hyperreduction2}

To enable hyper-reduction for the proposed conservative reduced-order models,
in addition to approximating the nonlinear objective functions 
that appear in optimization
problems \eqref{eq:galOptimalityStructPres2} and
\eqref{eq:LSPGODeltaEStructPresred}
as previously described in Section
\ref{sec:hyperreduction}, we must also approximate the nonlinear constraints
$\meshMapping\resGalArgs{\podstate\velocityOptDummyRed}{\timeVar}{\stateInit +
	\podstate\stateRed}{\params}=\zero$ and $\resDecompArgs{n}{\stateInit(\params) +
	\podstate \stateRedOptDummy}{\params}=\zero$. 
	 To accomplish this, we propose
	applying 
hyper-reduction to 
the nonlinear residuals that appears in the constraints, i.e., the
constraints become
\begin{equation}
\meshMapping\resGalHyperDecompArgs{\podstate\velocityOptDummyRed}{\timeVar}{\stateInit +
\podstate\stateRed}{\params}=\zero\quad \text{and}\quad \meshMapping\resHyperDecompArgs{n}{\stateInit(\params) +
	\podstate \stateRedOptDummy}{\params}=\zero
\end{equation}	
for conservative Galerkin and LSPG projection, respectively. Here,
approximations
$\resGalHyperDecomp(\approx\resGal)$ and 
$\resHyperDecompTime{\timestepit}(\approx\resTime{\timestepit})$ can be
constructed using any of the approaches described in Section
\ref{sec:hyperreduction}; we note that, in general,
different approximations can be employed for the objective and constraints
such that  $\resGalHyperDecomp\neq\resGalHyper$ and
$\resHyperDecompTime{\timestepit}\neq\resHyperTime{\timestepit}$.

In addition to the two forms of hyper-reduction introduced in Section
\ref{sec:hyperreduction}, we also propose a third type that leverages the
underlying finite-volume discretization of the governing equations:
\begin{enumerate}[leftmargin=*]
	\item[3]\label{hyper:fluxSource} \textbf{Flux and source hyper-reduction}. This approach
	respects the underlying decomposition of the velocity
	vector. It adopts the same residual approximation
	\eqref{eq:resHyperCont}--\eqref{eq:resHyperDisc} as velocity hyper-reduction
	(approach
	\ref{hyper:velocity} in Section \ref{sec:hyperreduction}), but employs separate approximations for each term
	comprising the velocity, i.e.,
\begin{equation} 
\velocityHyper=
\velocitySourceHyper
 + 
\velocityFluxHyper,\quad
\velocityFluxHyper =
	\alignment\faceFluxHyperVec
\end{equation} 
where
\begin{equation}
	\velocitySourceHyper=
	\podsource\pinv{\sampleMatsource\podsource}\sampleMatsource\velocitySource,
\quad
	\faceFluxHyperVec=\podflux\pinv{\sampleMatflux\podflux}\sampleMatflux\faceFluxVec
\end{equation}
in the case of gappy POD, or 
\begin{equation}
	\velocitySourceHyper =
	\sampleMatsource^T\sampleMatsource\velocitySource,
\quad
	\faceFluxHyperVec=\sampleMatflux^T\sampleMatflux\faceFluxVec
\end{equation}
in the case of collocation.
Here, $\sampleMatsource\in\{0,1\}^{\nsamplessource\times
	\ndof}$ and 
	$\sampleMatflux\in\{0,1\}^{\nsamplesflux\times
\ndofFace}$
	denote sampling matrices comprising selected rows of the
	identity matrix, while
	$\podsource\in\RRstar{\nstate\times\nsourceRed}$ and 
	$\podflux\in\RRstar{\ndofFace\times\nfluxRed}$
	denote reduced-basis
	matrices constructed for the source and flux, respectively.
 \end{enumerate}

One can consider a hierarchy of models that employ objective functions
and constraints, each of which may or may not employ one of the three proposed
hyper-reduction techniques. For this purpose,
we define the Tier-1 and Tier-2 Galerkin and LSPG objective functions as
\begin{gather}
	\objGalOneArgs{\velocityOptDummyRed}{t}{\params}\defeq	
	\norm{
	\resGalArgs{\podstate\velocityOptDummyRed}{\timeVar}{\stateInit + \podstate\stateRed}{\params}
}_2^2,\quad
\objGalTwoArgs{\velocityOptDummyRed}{t}{\params}\defeq	\norm{
	\resGalHyperArgs{\podstate\velocityOptDummyRed}{\timeVar}{\stateInit + \podstate\stateRed}{\params}
}_2^2,\\
\objLSPGOneArgs{\stateOptDummyRed}{\params}\defeq\norm{\weightingMatrixArgs{\stateInit(\params) +
	\podstate \stateRedOptDummy}{\params}\resArgs{n}{\stateInit(\params) +
\podstate \stateRedOptDummy}{\params}}_2^2,\quad
	\objLSPGTwoArgs{\stateOptDummyRed}{\params}\defeq\norm{\resHyperArgs{n}{\stateInit(\params) +
	\podstate \stateRedOptDummy}{\params}}_2^2,
\end{gather}
and the Tier-0 (unconstrained), Tier-1, and Tier-2 Galerkin and LSPG constraints as 
\begin{gather}
	\constraintGalGenArgs{\velocityOptDummyRed}{t}{\params}{0}\defeq \zero,\ \ \constraintGalOneArgs{\velocityOptDummyRed}{t}{\params}\defeq	\meshMapping
	\resGalArgs{\podstate\velocityOptDummyRed}{\timeVar}{\stateInit + \podstate\stateRed}{\params},\ \
	\constraintGalTwoArgs{\velocityOptDummyRed}{t}{\params}\defeq	\meshMapping
	\resGalHyperDecompArgs{\podstate\velocityOptDummyRed}{\timeVar}{\stateInit +
	\podstate\stateRed}{\params}\\
	\constraintLSPGGenArgs{\stateOptDummyRed}{\params}{0}\defeq\zero,\ \
\constraintLSPGOneArgs{\stateOptDummyRed}{\params}\defeq\resDecompArgs{n}{\stateInit(\params) +
	\podstate \stateRedOptDummy}{\params},\ \
\constraintLSPGTwoArgs{\stateOptDummyRed}{\params}\defeq\meshMapping\resHyperDecompArgs{n}{\stateInit(\params) +
	\podstate \stateRedOptDummy}{\params},
\end{gather}
Then, we say the Tier A-B Galerkin ROM solution $\frac{d\stateRed}{d\timeVar}$
is the solution to
\begin{align} \label{eq:GalerkinAB}
	\begin{split}
		\underset{\velocityOptDummyRed\in\RR{\nstateRed} }{\text{minimize}}\
		\objGalGenArgs{\velocityOptDummyRed}{t}{\params}{A} \ \text{subject to}\
		\constraintGalGenArgs{\velocityOptDummyRed}{t}{\params}{B} =\zero
	\end{split}
\end{align} 
and the Tier A-B LSPG ROM solution $\stateRedTime{n}$ is the solution to
\begin{align} \label{eq:LSPGAB}
	\begin{split}
		\underset{\stateOptDummyRed\in\RR{\nstateRed} }{\text{minimize}}\
		\objLSPGGenArgs{\stateOptDummyRed}{\params}{A} \ \text{subject to}\
		\constraintLSPGGenArgs{\stateOptDummyRed}{\params}{B} =\zero.
	\end{split}
\end{align} 
Note that Tier $i$-0 models correspond to the (original) unconstrained models, 
Tier $i$-1 models enforce conservation over subdomains, and 
Tier $i$-2 models enforce \textit{approximate} conservation over subdomains.
The penalty-method variants of the Tier A-B Galerkin and LSPG ROMs are,
respectively,
\begin{gather} \label{eq:GalerkinABPenalty}
		\underset{\velocityOptDummyRed\in\RR{\nstateRed} }{\text{minimize}}\
		\objGalGenArgs{\velocityOptDummyRed}{t}{\params}{A} + \penaltyParam\norm{
		\constraintGalGenArgs{\velocityOptDummyRed}{t}{\params}{B}}_2^2\\
\label{eq:LSPGABPenalty}
		\underset{\stateOptDummyRed\in\RR{\nstateRed} }{\text{minimize}}\
		\objLSPGGenArgs{\stateOptDummyRed}{\params}{A} +\penaltyParam\norm{
		\constraintLSPGGenArgs{\stateOptDummyRed}{\params}{B}}_2^2.
\end{gather} 
\begin{remark}[Computational cost of evaluating the objective function and
constraints]
	We note that the computational cost incurred by evaluating the constraints
	is often significantly lower than the cost of evaluating the objective
	function. For example, for a linear or zero source term, the only
	nonlinear contribution to the constraints arises from the face
	flux along the boundary of the subdomains comprising the decomposed mesh.
	For a small number of subdomains (e.g., global conservation with $\nSubdomains =1$), this requires computing only a small number of the
	elements of the face-flux vector $\faceFluxVec$, even without
	hyper-reduction. Thus, applying hyper-reduction to the objective function is
	generally more important for computational-cost reduction than applying
	hyper-reduction to the constraints, i.e., Tier 2--1 ROMs may be preferable
	to Tier 2--2 ROMs, as their cost is often similar and the former
	strictly enforces conservation.
\end{remark}

\subsection{Snapshot-based training}\label{sec:training}
Here, we propose to construct the reduced-basis matrices 
$\podstate$, $\podres$,
$\podvelocity$, $\podflux$, and $\podsource$
\reviewerB{during the offline stage}
using proper orthogonal decomposition (POD). In particular, given a set of training parameter instances
$\paramDomainTrain \defeq \{\paramTrain{1}, \ldots,
\paramTrain{\ntrain}\}\subset\paramDomain$,
we execute training simulations from which we compute `data tensors'
\begin{align}
	&\tensor_{ijk}\defeq
\stateEntry{i}(t^j;\paramTrain{k})-\stateInitEntry{i}(\paramTrain{k}),\quad
i\innat{\nspacedof},\ j\innat{\ntimedof},\ k\innat{\ntrain}\\
&\tensorres_{ijk\ell}(\stateDummy)\defeq
\resEntryTime{i}{j}(\stateDummy^{j(\ell)};\paramTrain{k}),\quad
i\innat{\nspacedof},\ j\innat{\ntimedof},\ k\innat{\ntrain},\
\ell\innat{\kmax{j}{\stateDummy}{\paramTrain{k}}}\\
&\tensorvelocity_{ijk}(\stateDummy)\defeq
\velocityEntry{i}(\stateDummy^{j},\timeArg{j};\paramTrain{k}),\quad
i\innat{\nspacedof},\ j\innat{\ntimedof},\ k\innat{\ntrain}\\
&\tensorflux_{ijk}(\stateDummy)\defeq\faceFluxEntryArgs{i}{\stateDummy^j}{\timeArg{j}}{\paramTrain{k}},\quad
i\innat{\ndofFace},\ j\innat{\ntimedof},\ k\innat{\ntrain}\\
&\tensorsource_{ijk}(\stateDummy)\defeq
\velocityFluxEntryArgs{i}{\stateDummy^j}{\timeArg{j}}{\paramTrain{k}},\quad
i\innat{\nspacedof},\ j\innat{\ntimedof},\ k\innat{\ntrain}.
\end{align}
Here, a superscript $j(\ell)$ denotes the value of a variable at the $\ell$th
Newton(-like) iteration during the solution of its nonlinear O$\Delta$E at
time instance $\timeArg{j}$ and $\kmax{j}{\stateDummy}{\params}$ denotes the
maximum number of Newton(-like) iterations taken during the simulation of
solution $\stateDummy$ at time instance $\timeArg{j}$ and parameter instance
$\params$. 

Note that constructing the state tensor $\tensor$ requires solving
the full-order model \eqref{eq:fom} at training instances
$\params\in\paramDomainTrain$, while constructing the other tensors requires
computing the solution $\stateDummy$ at these parameter instances;
$\stateDummy$ can correspond to the full-order model state (i.e., $\stateDummy =
\state$) or 
Tier A-B reduced-order model states (i.e., $\stateDummy = \stateApprox$) for
$\mathrm{A}\in\{1,2\}$ and $\mathrm{B}\in\{0,1\}$.
Clearly, the least computationally expensive approach is
to employ either $\stateDummy=\state$---as the training full-order-model
simulations are already required to construct the state tensor $\tensor$---or
$\stateDummy=\stateApprox$ corresponding to the Tier 2-B model for
$\text{B}\in\{0,1\}$, as the hyper-reduced objective function reduces the
simulation cost significantly.

The reduced-basis matrix associated with each data tensor can be computed as
the dominant left singular vectors of its mode-1 unfolding; for example, the
state basis $\podstate\equiv[\podstateVec{1}\ \cdots\
\podstateVec{\nstateRed}]$ is computed as 
 \begin{gather} \label{eq:spatialSVD}
	 \tensorUnfold{1} \defeq
	 \bmat{\snapshots(\paramTrain{1}) &
\ldots & \snapshots(\paramTrain{\ntrain})}
	 = \leftsingmatstate \singvalmatstate
	\rightsingmatstate^T\in\RR{\ndof\times\ntimedof\ntrain}\\
	\podstateVec{i} = \leftsingvecstatei{i},\quad i\innat\nstateRed,
 \end{gather} 
 where $\snapshots(\param)\defeq[\stateArg{1}(\params)\ \cdots\
 \stateArg{\ntimedof}(\params)]$ is often referred to as the
 `snapshot matrix'.
		 
Further, we propose to construct the sampling matrices $\sampleMatres$,
$\sampleMatvelocity$, $\sampleMatflux$, and $\sampleMatsource$ using the
sample-mesh greedy method presented in Ref.~\cite{carlbergJCP}, which allows
for oversampling to enable least-squares regression via gappy POD and also
constructs a `sample mesh' wherein all residual elements associated with a
given control volume are sampled. However, rather than constructing each of
these sampling matrices independently, we propose to construct $\sampleMatres$
according to the greedy method executed with basis $\podres$ and subsequently
set $\sampleMatvelocity = \sampleMatsource=\sampleMatres$. Further, we
construct $\sampleMatflux$ to select the faces associated with the control
volumes sampled by $\sampleMatres$; this corresponds to selecting the
sampling matrix $\sampleMatflux$ with the maximum number of rows such that
$\sampleMatres\alignment =
\sampleMatres\alignment\sampleMatflux^T\sampleMatflux$.

\section{Analysis}\label{sec:analysis}

This section performs analysis of the proposed conservative Galerkin and
conservative LSPG techniques. For simplicity, we focus on the models without
hyper-reduction; the hyper-reduced variants of the results can be derived in a
similar manner by making the obvious substitutions.

\subsection{Feasibility conditions}\label{sec:feasibility}
We first derive sufficient conditions under which the optimization problems
characterizing conservative Galerkin and conservative LSPG projection are
feasible.
\begin{proposition}[Sufficient conditions for feasibility of conservative
	Galerkin projection]\label{prop:sufficientFeasGal}
Problem \eqref{eq:galOptimalityStructPres2} is feasible if 
$\rank{\meshMapping\podstate}=\ndofDecomp$, which in turn requires
$\nstateRed\geq \ndofDecomp$, i.e., the number of reduced basis vectors 
exceeds the number of constraints.
\end{proposition}
\begin{proof}
	If 
$\rank{\meshMapping\podstate}=\ndofDecomp$, then
$\range{\meshMapping\podstate} = \RR{\ndofDecomp}$
and thus
$\velocityDecompArgs{\stateDummy}{\timeDummy}{\paramsDummy}\in\range{\meshMapping\podstate}$
regardless of its arguments.
\end{proof}

\begin{proposition}[Sufficient conditions for feasibility of conservative
	LSPG projection]\label{prop:suffLSPG}
Problem \eqref{eq:LSPGODeltaEStructPresred} is feasible if 
(1) an explicit scheme is employed and
$\rank{\meshMapping\podstate}=\ndofDecomp$, 
(2) the limit $\dt\rightarrow 0$ is taken,
or (3) the velocity $\velocity$ is
linear in its first argument and 
$\rank{\meshMapping[\alpha_0\identity - \dt\beta_0
\partial\velocity/\partial\stateDummy(\cdot,\timeArg{\timestepit};\params)]\podstate}=\ndofDecomp
$.
\end{proposition}
\begin{proof}
	\textit{Case 1.}
	If an explicit scheme is employed, then $\beta_0=0$ and the feasible set
	becomes
	\begin{equation}\label{eq:feasibleSetLSPGArgsExplicit}
\feasibleSetLSPGArgs{n}{\paramsDummy}
=
\{\velocityOptDummyTwo\in\RR{\nstateRed}\,|\,
	\alpha_0
 \meshMapping\podstate\unknown = 
	-\sum_{j=1}^k\alpha_j\meshMapping\podstate\stateRed^{\timestepit -j}(\params)+
\dt \sum_{j=1}^k\beta_j \velocityDecompArgs{\stateInit(\params) +
\podstate \stateRed^{\timestepit -j}}{t^{\timestepit -j}}{\params}\}.
			\end{equation}
If $\rank{\meshMapping\podstate}=\ndofDecomp$, then
$\range{\meshMapping\podstate} = \RR{\ndofDecomp}$
and right-hand-side of the constraints in \eqref{eq:feasibleSetLSPGArgsExplicit}
must lie in
$\range{\meshMapping\podstate}$.\\
	\noindent \textit{Case 2.}
	If the limit $\dt\rightarrow 0$ is taken, then the feasible set
	becomes
	\begin{equation}\label{eq:feasibleSetLSPGArgsLimit}
\feasibleSetLSPGArgs{n}{\paramsDummy}
=
\{\velocityOptDummyTwo\in\RR{\nstateRed}\,|\,
	\alpha_0
 \meshMapping\podstate\unknown = 
	-\sum_{j=1}^k\alpha_j\meshMapping\podstate\stateRed^{\timestepit -j}(\params)
\}
			\end{equation}
and 
the right-hand-side of the constraints in \eqref{eq:feasibleSetLSPGArgsLimit}
will
lie in
$\range{\meshMapping\podstate}$ regardless of its rank.\\
	\noindent \textit{Case 3.} 
	If the velocity is linear in the state, the feasible set becomes
	\reviewerB{
	\begin{align}\label{eq:feasibleSetLSPGArgsLinear}
			\begin{split}
\feasibleSetLSPGArgs{n}{\paramsDummy}
=
\{\velocityOptDummyTwo\in\RR{\nstateRed}\,|\,
	\meshMapping[&\alpha_0\podstate - \dt\beta_0
	\partial\velocity/\partial\stateDummy(\cdot,\timeArg{\timestepit};\params)\podstate]\unknown
	= \\
	&-\sum_{j=1}^k\alpha_j\meshMapping\podstate\stateRed^{\timestepit -j}(\params)+
\dt \sum_{j=1}^k\beta_j \velocityDecompArgs{\stateInit(\params) +
\podstate \stateRed^{\timestepit -j}}{t^{\timestepit -j}}{\params}\}
			\end{split}
			\end{align}
		}
\noindent and (as above) $\rank{\meshMapping[\alpha_0\identity - \dt\beta_0
\partial\velocity/\partial\stateDummy(\cdot,\timeArg{\timestepit};\params)]\podstate}=\ndofDecomp$
ensures
the right-hand-side of the constraints in \eqref{eq:feasibleSetLSPGArgsExplicit}
will lie in
$\range{\meshMapping\podstate}$.
\end{proof}

\subsection{Equivalence conditions}\label{sec:equivalence}
We now derive conditions under which conservative Galerkin and conservative
LSPG projection are equivalent.
\begin{theorem}[Equivalence]\label{thm:equivalent}
	The discrete-time conservative Galerkin ROM solution is
	equivalent to the conservative LSPG solution if either (1) an explicit scheme is employed
	or (2) the limit $\dt\rightarrow 0$ is taken. Further, under
	these conditions, the
	Lagrange multipliers are related as
\begin{equation} \label{eq:lagrangeLSPGlagrangeGal}
	\lagrangeMultipliersLSPGSolveTime{n}  =
	\sum_{j=0}^k\alpha_j\lagrangeMultipliersGalTime{\timestepit-j}.
\end{equation} 
\end{theorem}
\begin{proof}
	We first note that the discrete-time conservative Galerkin ROM solution
	$\stateRedGalTime{n}$ satisfies
	Eqs.~\eqref{eq:consGalTimeDisc}, which can be re-written as
	\begin{equation}\label{eq:consGalTimeDiscRewrite} 
	\begin{alignedat} {2}
		&\podstate^T\left[
		\resArgs{\timestepit}{\stateInit(\params) +\podstate\stateRedGalTime{\timestepit} }{\params} +
	\sum_{j=0}^k\alpha_j\meshMapping^T\lagrangeMultipliersGalTime{\timestepit-j}\right]
	&&= 
\zero\\
&
\resDecompArgs{\timestepit}{\stateInit(\params) +\podstate\stateRedGalTime{\timestepit} }{\params}
 &&=
\zero.
\end{alignedat} 
\end{equation} 
Comparing Eqs.~\eqref{eq:LSPGKKT} and \eqref{eq:consGalTimeDiscRewrite} reveals that the
discrete-time conservative LSPG and 
conservative Galerkin solutions are equivalent if (1) $
\weightingMatrixT\weightingMatrix\testBasisArgs{n}{\stateRedTime{n}}{\params}=
\constant\podstate
$  and (2) 
 $
\testBasisArgs{n}{\stateRedTime{n}}{\params}^T\meshMapping^T\lagrangeMultipliersLSPGSolveTime{n}=
\constant\sum_{j=0}^k\alpha_j\podstate^T\meshMapping^T\lagrangeMultipliersGalTime{\timestepit-j}
$ for any constant $\constant\in\RR{}$. As was shown in Ref.~\citep{carlbergGalDiscOpt}, the first condition holds 
for $\constant = \alpha_0$ if 
 either the
scheme 
is explicit (i.e., $\beta_0=0$) or the limit $\dt\rightarrow 0$ is taken.
The same conditions apply to the second condition above.  To see this, 
note that 
\begin{equation} \label{eq:MeshMappingRev}
\meshMapping\testBasisArgs{n}{\stateRedTime{n}}{\params} = 
\alpha_0 \meshMapping\podstate - \dt \beta_0\meshMapping
\frac{\partial\velocity}{\partial\stateDummy}(\unknown,\timeArg{\timestepit};\params)\podstate\reviewerB{.}
\end{equation} 
If either $\beta_0 = 0$ or the limit $\dt\rightarrow 0$ is taken, the second term
vanishes such that we have
\begin{equation} 
\testBasisArgs{n}{\stateRedTime{n}}{\params}^T\meshMapping^T \lagrangeMultipliersLSPGSolveTime{n}
= 
\alpha_0\podstate^T\meshMapping^T \lagrangeMultipliersLSPGSolveTime{n} .
\end{equation} 
This expression is equivalent to
$\constant\sum_{j=0}^k\alpha_j\podstate^T\meshMapping^T\lagrangeMultipliersGalTime{\timestepit-j}$
with $\constant = \alpha_0$
if Eq.~\eqref{eq:lagrangeLSPGlagrangeGal} holds.
\end{proof}

\subsection{Error analysis}\label{sec:error}

We now derive several \textit{a posteriori} error bounds for (components of)
the solution computed by the proposed conservative model-reduction methods.
We employ some of the same techniques used for error analysis in
Ref.~\cite{carlbergGalDiscOpt}. For notational simplicity, we drop dependence
of the operators on the parameters $\params$.

We begin by writing the discrete equations characterizing the full-order,
Galerkin, and LSPG models as 
\begin{align}
	\label{eq:FOMerror}\alpha^n_0\stateFOM{n} &= \beta_0^n\dt\f{\stateInit +
\stateFOM{n}}{\timeArg{n}} + \resFOMPrev\\
	\label{eq:Gal}\alpha^n_0\stateGal{n} &= \beta_0^n\dt\podstate^T\f{\stateInit +
\podstate\stateGal{n}}{\timeArg{n}} + \resGalPrev-
\sum_{j=0}^k\alpha^n_j\podstate^T\meshMapping^T\lagrangeMultipliersGalTime{\timestepit-j}
,
\\
\label{eq:LSPG}
\begin{split}
	\alpha^n_0\stateLSPG{n} &= \beta_0^n\dt(\testBasisAAShortnT\podstate)^{-1}\testBasisAAShortnT\f{\stateInit +
\podstate\stateLSPG{n}}{\timeArg{n}} + \resLSPGPrev \\
&- (\testBasisAAShortnT\podstate)^{-1}\testBasisAAShortnT\meshMapping^T\lagrangeMultipliersLSPGSolveTime{n},
\end{split}
\end{align}
respectively,
where $\testBasisAAShortn\defeq\testBasisArgsTwo{n}{\stateLSPG{n}}$,
as well as
\begin{align}
\label{eq:galDecomp}	\alpha^n_0\meshMapping\podstate\stateGal{n} &= \beta_0^n\dt\meshMapping\f{\stateInit +
\podstate\stateGal{n}}{\timeArg{n}} + \meshMapping\resPrev{\podstate\stateGal}\\
\label{eq:lspgDecomp}
	\alpha^n_0\meshMapping\podstate\stateLSPG{n} &= \beta_0^n\dt\meshMapping\f{\stateInit +
\podstate\stateLSPG{n}}{\timeArg{n}} +
\meshMapping\resPrev{\podstate\stateLSPG}
\end{align}
with  $\stateFOM{0} = \zero$ and $\stateGal{0}=\stateLSPG{0} = \zero$. Here, we have defined
 \begin{align} \label{eq:resPrevRev}
	 \resPrev{\stateArg}&\defeq\sum_{\ell=1}^k\left(\beta_\ell^n\dt\f{\stateInit+\stateArg{n-\ell}}{\timeArg{n-\ell}}-\alpha^n_\ell\stateArg{n-\ell}\right)\reviewerB{.}
 \end{align} 
We also assume Lipschitz continuity of
$\velocity$ in its first argument:
\begin{itemize}
\item[${\bf A_1}$] There exists a constant $\lipschitzConstant >0$ such that for $\bds x, \bds y \in \RR{\ndof}$
    \begin{equation*}
        \normtwo{\f {\state}{t} - \f{\bds y}{t}} \le \lipschitzConstant  \normtwo{\state - \bds y},\quad \forall t\in\timeDomain.
    \end{equation*}
\end{itemize}
To simplify notation, we define the Galerkin and LSPG operators as
\begin{equation}
\GalProj\defeq\podstate\podstate^T,\quad \LSPGProj\defeq\podstate(\testBasisAAShortnT\podstate)^{-1}\testBasisAAShortnT,
\end{equation}
respectively, and the Galerkin and LSPG state-space errors at time instance $n$ as
\begin{equation}
\errorGal{n}\defeq\stateFOM{n} - \podstate\stateGal{n},\quad
\errorLSPG{n}\defeq\stateFOM{n} - \podstate\stateLSPG{n},
\end{equation}
respectively. Because the time instance of the first and second arguments of
$\velocity$ always match for linear multistep schemes, we omit the second
argument (time) from $\velocity$ in the remainder of this section. All norms in this section
correspond to the Euclidean norm, i.e., $\norm{\cdot} = \norm{\cdot }_2$.

We proceed by deriving \textit{a posteriori} error bounds for the proposed
conservative techniques. We remark that derivation of state-space error bounds for the
proposed constrained ROMs is complicated by the presence of Lagrange
multipliers in the discrete equations
\eqref{eq:Gal}--\eqref{eq:LSPG}.  Thus, we derive bounds that relate to
the null-space and row-space of the associated constraint matrices, which
enables elimination of these
Lagrange multipliers from the analysis. To
accomplish this, we
make use of three 
decompositions of $\RR{\ndof}$.  The first is $\RR{\ndof} =
\range{\rightSingFull} \oplus\range{\nullBasisFull}$, where 
$\rightSingFull\in\RR{\ndof\times \rank{\meshMapping\podstate}}$ and 
$\nullBasisFull\in\RR{\ndof\times \ndof-\rank{\meshMapping\podstate}}$ are orthogonal matrices satisfying
\begin{gather}
\meshMapping\podstate = \leftSing\Sing\rightSing^T,\quad \rightSingFull = \podstate\rightSing,\quad\ \nullBasisFull^T\rightSingFull = \zero.
\end{gather}
The second is
$\RR{\ndof} = \range{\rightSingPFull} \oplus\range{\nullBasisPFull}$, where 
$\rightSingPFull\in\RR{\ndof\times \rank{\meshMapping \testBasisAAShortn}}$ and 
$\nullBasisPFull\in\RR{\ndof\times \ndof-\rank{\meshMapping \testBasisAAShortn}}$ are orthogonal matrices satisfying
\begin{gather}
	\meshMapping \testBasisAAShortn(\podstate^T\testBasisAAShortn)^{-1}
	 = \leftSingP\SingP\rightSingPT,\quad \rightSingPFull = \podstate\rightSingP,\quad\ \nullBasisPFullT\rightSingPFull = \zero.
\end{gather}
Finally, we consider $\RR{\ndof} = \range{\rightSingCFull}
\oplus\range{\nullBasisCFull}$, where 
$\rightSingCFull\in\RR{\ndof\times \rank{\meshMapping}}$ and 
$\nullBasisCFull\in\RR{\ndof\times \ndof-\rank{\meshMapping}}$ are orthogonal matrices satisfying
\begin{gather}
\meshMapping = \leftSingC\SingC\rightSingCFull^T,\quad \nullBasisCFull^T\rightSingCFull = \zero.
\end{gather}
Note that $\pinvmeshMapping\meshMapping = \rightSingCFull\rightSingCFull^T$.

\begin{lemma}[Local \textit{a posteriori} error bounds: null-space
	error]\label{thm:boundNull}
	If ${\bf A_1}$ holds and $\dt < \abs{\alpha_0^n}/(\abs{\beta_0^n}\lipschitzConstant)$, then 
\begin{align}
\label{eq:localFinalGal}
&\norm{\nullBasisFull^T\errorGal{n}} \leq
\sum_{\ell=0}^k\epsilonCoeff{\ell}{n}\left(\lipschitzConstant\norm{\rightSingFull^T\errorGal{n-\ell}}
	+ \norm{(\identity -
\GalProj)\fSimple{\stateInit+\podstate\stateGal{n-\ell}}{\timeArg{n-\ell}}}\right)+\sum_{\ell=1}^k\gammaCoeff{\ell}{n}\norm{\nullBasisFull^T\errorGal{n-\ell}}
\\
\label{eq:localFinalLSPG}
&\norm{\nullBasisPFullT\errorLSPG{n}} \leq 
\sum_{\ell=0}^k\epsilonCoeff{\ell}{n}\left(\lipschitzConstant\norm{\rightSingPFullT\errorLSPG{n-\ell}}
	+ \norm{(\identity -
\LSPGProj)\fSimple{\stateInit+\podstate\stateLSPG{n-\ell}}{\timeArg{n-\ell}}}\right)+\sum_{\ell=1}^k\gammaCoeff{\ell}{n}\norm{\nullBasisPFullT\errorLSPG{n-\ell}},
\end{align}
where $\epsilonCoeff{\ell}{m}\defeq\abs{\beta_\ell^m}\dt/\hCoeff{m}$, 
$\gammaCoeff{\ell}{m}\defeq(\abs{\alpha_\ell^m} + \abs{\beta_\ell^m}\lipschitzConstant\dt)/\hCoeff{m}$, and
$\hCoeff{m}\defeq\abs{\alpha_0^m} - \abs{\beta_0^m}\lipschitzConstant\dt$.
\end{lemma}
\begin{proof}
Subtracting the premultiplication of Eq.~\eqref{eq:LSPG} by $\nullBasisPFullT\podstate$
from the premultiplication of Eq.~\eqref{eq:FOMerror} by $\nullBasisPFullT$ 
yields
\begin{align}\label{eq:localErrOneP}
\begin{split}
\alpha^n_0\nullBasisPFullT\errorLSPG{n} = \beta_0^n\dt\nullBasisPFullT\left(\fSimple{\stateInit + \stateFOM{n}}{\timeArg{n}}
-
\LSPGProj
\fSimple{\stateInit +
\podstate\stateLSPG{n}}{\timeArg{n}}
\right) + \nullBasisPFullT\resLSPGDiff,
\end{split}
\end{align}
where $\resLSPGDiff\defeq\resFOMPrev-\LSPGProj\resPrev{\podstate\stateLSPG}$. Adding and subtracting 
$\beta_0^n\dt\nullBasisPFullT\fSimple{\stateInit +
\podstate\stateLSPG{n}}{\timeArg{n}}$ from Eq.~\eqref{eq:localErrOneP} yields
\begin{align}\label{eq:localErrTwoP}
	\begin{split}
		\alpha^n_0\nullBasisPFullT\errorLSPG{n} =& \beta_0^n\dt\nullBasisPFullT[\fSimple{\stateInit + \stateFOM{n}}{\timeArg{n}}
-\fSimple{\stateInit +
\podstate\stateLSPG{n}}{\timeArg{n}}\\
&+\fSimple{\stateInit +
\podstate\stateLSPG{n}}{\timeArg{n}}-\LSPGProj\fSimple{\stateInit +
\podstate\stateLSPG{n}}{\timeArg{n}}
] + \nullBasisPFullT\resLSPGDiff.
\end{split}
\end{align}
Applying the triangle inequality yields
\begin{equation}
	\abs{\alpha^n_0}\norm{\nullBasisPFullT\errorLSPG{n}} \leq
	\abs{\beta_0^n}\dt\left(\norm{\fSimple{\stateInit + \stateFOM{n}}{\timeArg{n}}
-\fSimple{\stateInit +
\podstate\stateLSPG{n}}{\timeArg{n}}}+\norm{(\identity-\LSPGProj)\fSimple{\stateInit +
\podstate\stateLSPG{n}}{\timeArg{n}}}
\right) + \norm{\nullBasisPFullT\resLSPGDiff}.
\end{equation}
Now, using Lipschitz continuity and $\bds y =
\rightSingPFull\rightSingPFullT\bds y + \nullBasisPFull\nullBasisPFullT\bds y$ for all $\bds y\in\RR{\ndof}$, we have
\begin{equation}
	\abs{\alpha^n_0}\norm{\nullBasisPFullT\errorLSPG{n}} \leq
	\abs{\beta_0^n}\dt\left(\lipschitzConstant\norm{\nullBasisPFullT\errorLSPG{n}} +
		\lipschitzConstant\norm{\rightSingPFullT\errorLSPG{n}} + \norm{(\identity-\LSPGProj)\fSimple{\stateInit +
		\podstate\stateLSPG{n}}{\timeArg{n}}}
	\right) + \norm{\nullBasisPFullT\resLSPGDiff}.
\end{equation}
Using $\dt<\abs{\alpha^n_0}/(\abs{\beta_0^n}\lipschitzConstant)$, we have 
\begin{equation}\label{eq:localOne}
	\norm{\nullBasisPFullT\errorLSPG{n}} \leq
	\frac{\abs{\beta_0^n}\dt}{\hCoeff{n}}\left(\lipschitzConstant\norm{\rightSingPFullT\errorLSPG{n}}
		+ \norm{(\identity-\LSPGProj)\fSimple{\stateInit +
		\podstate\stateLSPG{n}}{\timeArg{n}}}
	\right) + \frac{1}{\hCoeff{n}}\norm{\nullBasisPFullT\resLSPGDiff}.
\end{equation}
Next, we estimate $\norm{\nullBasisPFullT\resLSPGDiff}$. First, we have
 \begin{align} 
 \begin{split} 
	 \nullBasisPFullT\resLSPGDiff=&\sum_{\ell=1}^k\beta_\ell^n\dt\nullBasisPFullT\left(\fSimple{\stateInit+\stateFOM{n-\ell}}{\timeArg{n-\ell}}-\LSPGProj\fSimple{\stateInit+\podstate\stateLSPG{n-\ell}}{\timeArg{n-\ell}}\right)\\
																&-\sum_{\ell=1}^k\alpha^n_\ell\nullBasisPFullT\left(\stateFOM{n-\ell}-\podstate\stateLSPG{n-\ell}\right).
 \end{split} 
 \end{align} 
Following similar steps to those above, adding and subtracting 
$\sum_{\ell=1}^k\beta_\ell^n\dt\nullBasisPFullT\fSimple{\stateInit +
\podstate\stateLSPG{n-\ell}}{\timeArg{n}}$ and applying the triangle inequality yields
 \begin{equation} \label{eq:finaldR}
	 \norm{\nullBasisPFullT\resLSPGDiff}\leq\sum_{\ell=1}^k\abs{\beta_\ell^n}\dt\left(\lipschitzConstant\norm{\rightSingPFullT\errorLSPG{n-\ell}}
		 + \norm{(\identity -
	 \LSPGProj)\fSimple{\stateInit+\podstate\stateLSPG{n-\ell}}{\timeArg{n-\ell}}}\right)+\sum_{\ell=1}^k(\abs{\beta_\ell^n}\lipschitzConstant\dt+\abs{\alpha^n_\ell})\norm{\nullBasisPFullT\errorLSPG{n-\ell}}.
 \end{equation} 
Combining inequalities \eqref{eq:localOne} and \eqref{eq:finaldR} yields the final
result \eqref{eq:localFinalLSPG}. The Galerkin counterpart
\eqref{eq:localFinalGal} can be derived by following the same steps with the
Galerkin operators.
\end{proof}

Lemma \ref{thm:boundNull} shows that the component of the error in the null
space of the constraints behaves very similarly to the full-space error in the
case of standard, unconstrained ROMs as reported in \cite[Theorem 6.1]{carlbergGalDiscOpt}; the only difference is the addition of
the terms arising from the row-space errors, which is 
$\sum_{\ell=0}^k\epsilonCoeff{\ell}{n}\lipschitzConstant\norm{\rightSingFull^T\errorGal{n-\ell}}$
and
$
\sum_{\ell=0}^k\epsilonCoeff{\ell}{n}\lipschitzConstant\norm{\rightSingPFullT\errorLSPG{n-\ell}}
$
for
conservative Galerkin and conservative LSPG projection, respectively.

\begin{lemma}[Local \textit{a posteriori} error bounds: row-space
	error]\label{thm:boundRow}
	If ${\bf A_1}$ holds and $\dt < \abs{\alpha_0^n}/(\abs{\beta_0^n}\lipschitzConstant)$, then 
\begin{align}\label{eq:localFinalGalInplane}
 \begin{split} 
	 \norm{\rightSingFull^T\errorGal{n}} \leq&
	 \sum_{\ell=0}^k\epsilonCoeff{\ell}{n}\left(\lipschitzConstant\norm{\nullBasisFull^T\errorGal{n-\ell}}
		 +\normedQuantityG\norm{(\identity-\GalProj)\fSimple{\stateInit +
	 \podstate\stateGal{n-\ell}}{\timeArg{n-\ell}}}\right)
	 +\sum_{\ell=1}^k\gammaCoeff{\ell}{n}\norm{\rightSingFull^T\errorGal{n-\ell}}
	  \end{split} \\
		\label{eq:localFinalLSPGInplane}
 \begin{split} 
	 \norm{\rightSingPFullT\errorLSPG{n}} \leq&
	 \sum_{\ell=0}^k\epsilonCoeff{\ell}{n}\left(\lipschitzConstant\norm{\nullBasisPFullT\errorLSPG{n-\ell}}
		 +\normedQuantityP\norm{(\identity-\LSPGProjT)\fSimple{\stateInit +
	 \podstate\stateLSPG{n-\ell}}{\timeArg{n-\ell}}}\right)
	 +\sum_{\ell=1}^k\gammaCoeff{\ell}{n}\norm{\rightSingPFullT\errorLSPG{n-\ell}}\\
&
	 +\frac{\normedQuantityP}{\hCoeff{n}}\norm{\ProjDiff}\sum_{\ell=0}^k\abs{\alpha_\ell^n}\norm{\stateLSPG{n-\ell}}\reviewerB{,}
	  \end{split}
\end{align}
 where
 $\normedQuantityG\defeq\norm{\Sing^{-1}\leftSing^T\meshMapping}$, $\normedQuantityP\defeq
 \norm{\SingPInv\leftSingPT\meshMapping}$,
 and
 $\ProjDiff\defeq\testBasisAAShortn(\podstate^T\testBasisAAShortn)^{-1}-\podstate$.
\end{lemma}
\begin{proof}
Noting that $\rightSingPFullT\podstate\stateLSPG{n} = \rightSingPT\stateLSPG{n}$, we have from 
adding and subtracting
$\alpha_0^n\meshMapping[\testBasisAAShortn(\podstate^T\testBasisAAShortn)^{-1}]\stateLSPG{n}$
to Eq.~\eqref{eq:lspgDecomp}
and
pre-multiplying by $\SingPInv\leftSingPT$
 \begin{equation} \label{eq:secondLSPG}
\alpha_0^n\rightSingPFullT\podstate\stateLSPG{n} =
\beta_0^n\dt\SingPInv\leftSingPT\meshMapping\fSimple{\stateInit +
\stateLSPG{n}}{\timeArg{n}} + 
\alpha_0^n\SingPInv\leftSingPT\meshMapping\ProjDiff\stateLSPG{n}+
\SingPInv\leftSingPT\meshMapping\resPrev{\podstate\stateLSPG}.
 \end{equation} 
Premultiplying
Eq.~\eqref{eq:FOMerror} by
$\rightSingPFullT=\SingPInv\leftSingPT\meshMapping\LSPGProjT$
yields
 \begin{equation} \label{eq:secondFOMP}
\alpha^n_0\rightSingPFullT\stateFOM{n} = \beta_0^n\dt
\SingPInv\leftSingPT\meshMapping\LSPGProjT\fSimple{\stateInit +
\stateFOM{n}}{\timeArg{n}} +
\SingPInv\leftSingPT\meshMapping\LSPGProjT\resPrev{\podstate\stateFOM}.
 \end{equation} 
 Subtracting \eqref{eq:secondLSPG} from \eqref{eq:secondFOMP} yields
 \begin{align} \label{eq:rowspaceLater}
 \begin{split} 
	 \alpha^n_0\rightSingPFullT\errorLSPG{n} =& \beta_0^n\dt
	 \SingPInv\leftSingPT\meshMapping\left[\LSPGProjT\fSimple{\stateInit +
\stateFOM{n}}{\timeArg{n}} - \fSimple{\stateInit +
\stateLSPG{n}}{\timeArg{n}}
\right]
- \alpha_0^n\SingPInv\leftSingPT\meshMapping\ProjDiff\stateLSPG{n}  \\
&+
\SingPInv\leftSingPT\meshMapping\resLSPGDiffP,
	  \end{split} 
	  \end{align} 
		where 
$\resLSPGDiffP\defeq\LSPGProjT\resPrev{\stateFOM}-
\resPrev{\podstate\stateLSPG}$.
		Adding and subtracting 
	$\LSPGProjT\fSimple{\stateInit +
\stateLSPG{n}}{\timeArg{n}}$ to the bracketed quantity, applying the triangle inequality, and using
Lipschitz continuity yields
 \begin{align} 
 \begin{split} 
	 \abs{\alpha^n_0}\norm{\rightSingPFullT\errorLSPG{n}} \leq& 
	 \abs{\beta_0^n}\dt\lipschitzConstant
	 \norm{\SingPInv\leftSingPT\meshMapping\LSPGProjT}(\norm{\rightSingPFullT\errorLSPG{n}}+\norm{\nullBasisPFullT\errorLSPG{n}})\\
	&+\abs{\beta_0^n}\dt
	 \norm{\SingPInv\leftSingPT\meshMapping}\norm{(\identity-\LSPGProj)\fSimple{\stateInit +
	 \podstate\stateLSPG{n}}{\timeArg{n}}}
\\
&+
\abs{\alpha_0^n}\norm{\SingPInv\leftSingPT\meshMapping\ProjDiff}\norm{\stateLSPG{n}}  
+
\norm{\SingPInv\leftSingPT\meshMapping\resLSPGDiffP}.
	  \end{split} 
	  \end{align} 
		Noting that $\norm{\SingPInv\leftSingPT\meshMapping\LSPGProjT} =
		\norm{\podstate\rightSingP}= 1$
		and $\dt < \abs{\alpha^n_0}/(\abs{\beta_0^n}
	 \lipschitzConstant)$
	 yields
 \begin{align}\label{eq:localOneInplane}
 \begin{split} 
	 \norm{\rightSingPFullT\errorLSPG{n}} \leq&
	 \epsilonCoeff{0}{n}\lipschitzConstant\norm{\nullBasisPFullT\errorLSPG{n}}
	 +\epsilonCoeffP{0}{n}\norm{(\identity-\LSPGProj)\fSimple{\stateInit +
	 \podstate\stateLSPG{n}}{\timeArg{n}}}
\\
&+
\frac{\abs{\alpha_0^n}\norm{\SingPInv\leftSingPT\meshMapping\ProjDiff}}{\hCoeff{n}}\norm{\stateLSPG{n}}  
+
\frac{1}{\hCoeff{n}}\norm{\SingPInv\leftSingPT\meshMapping\resLSPGDiffP}.
	  \end{split} 
	  \end{align} 
		Next, we estimate $\norm{\SingPInv\leftSingPT\meshMapping\resLSPGDiffP}$.
		First, we have from
adding and subtracting 
$\sum_{\ell=1}^k\testBasisAAShortn(\podstate^T\testBasisAAShortn)^{-1}\stateLSPG{n-\ell}$
\reviewerB{that}
\begin{align}\label{eq:tmptmp}
		\begin{split}
			\SingPInv\leftSingPT\meshMapping\resLSPGDiffP = &
\sum_{\ell=1}^k\beta_\ell^n\dt\SingPInv\leftSingPT\meshMapping\left(\LSPGProjT\fSimple{\stateInit+\stateFOM{n-\ell}}{\timeArg{n-\ell}}-\fSimple{\stateInit+\podstate\stateLSPG{n-\ell}}{\timeArg{n-\ell}}\right)\\
& -
\sum_{\ell=1}^k\alpha_\ell^n\SingPInv\leftSingPT\meshMapping\left(\LSPGProjT\stateFOM{n-\ell}-
\testBasisAAShortn(\podstate^T\testBasisAAShortn)^{-1}\stateLSPG{n-\ell}
+\ProjDiff\stateLSPG{n-\ell}\right).
		\end{split}
		\end{align}
Following similar steps to those above, adding and subtracting
$\sum_{\ell=1}^k\beta_\ell^n\dt\SingPInv\leftSingPT\meshMapping
\LSPGProjT\fSimple{\stateInit+\podstate\stateLSPG{n-\ell}}{\timeArg{n-\ell}}
$  and applying the triangle inequality yields
\begin{align}\label{eq:finaldRInplane}
		\begin{split}
			\norm{\SingPInv\leftSingPT\meshMapping\resLSPGDiffP}\leq &
			\sum_{\ell=1}^k\abs{\beta_\ell^n}\dt\lipschitzConstant(\norm{\rightSingPFullT\errorLSPG{n-\ell}}+\norm{\nullBasisPFullT\errorLSPG{n-\ell}}) \\
			&+\sum_{\ell=1}^k\abs{\beta_\ell^n}\dt
			\norm{\SingPInv\leftSingPT\meshMapping}\norm{(\identity-\LSPGProj)\fSimple{\stateInit +
			\podstate\stateLSPG{n-\ell}}{\timeArg{n-\ell}}}\\
& +
\sum_{\ell=1}^k\abs{\alpha_\ell^n}
\norm{
\rightSingPFullT\errorLSPG{n-\ell}
}
+
\sum_{\ell=1}^k\abs{\alpha_\ell^n}\norm{\SingPInv\leftSingPT\meshMapping\ProjDiff}\norm{\stateLSPG{n-\ell}}.
		\end{split}
		\end{align}
Combining inequalities \eqref{eq:localOneInplane} and \eqref{eq:finaldRInplane} yields the final
result \eqref{eq:localFinalLSPGInplane}. The Galerkin counterpart
\eqref{eq:localFinalGalInplane} can be derived by following the same steps with the
Galerkin operators; the main modification is that adding and subtracting 
$\sum_{\ell=0}^k\beta_\ell^n\dt\Sing^{-1}\leftSing^T\meshMapping
\GalProj\fSimple{\stateInit+\podstate\stateGal{n-\ell}}{\timeArg{n-\ell}}
$ is not needed in the Galerkin case.
\end{proof}

Lemma \ref{thm:boundRow} shows that---as was the case with null-space error
bounds in Lemma \ref{thm:boundNull}---the row-space error bounds are affected by the error incurred in the
null space through the terms
$
\sum_{\ell=0}^k\epsilonCoeff{\ell}{n}\lipschitzConstant\norm{\nullBasisFull^T\errorGal{n-\ell}}
$
and
$
\sum_{\ell=0}^k\epsilonCoeff{\ell}{n}\lipschitzConstant\norm{\nullBasisPFullT\errorLSPG{n-\ell}}
$
for conservative Galerkin and conservative LSPG projection, respectively.
Further, these bounds are quite similar to the null-space error bounds with
two exceptions. First, the projection-error term is multiplied by a constant,
which is $\normedQuantityG$ in the case of conservative Galerkin projection
and is $\normedQuantityP$ in the case of conservative LSPG projection; this
constant arises from the fact that these bounds are derived from the discrete
equations associated with subdomain conservation
\eqref{eq:galDecomp}--\eqref{eq:lspgDecomp}. We also note that the
conservative LSPG row-space error bound employs the transpose of the typical
LSPG projector, i.e., $\LSPGProjT$, which is the oblique projection onto
$\range{\testBasisAAShortn}$ orthogonal to $\range{\podstate}$; this 
arises from appearance of the \textit{transpose} of the constraints in
the discrete equations \eqref{eq:LSPG}. This also leads to the appearance of
the term proportional to $\norm{\ProjDiff}$ in the conservative LSPG error bound.

\begin{theorem}[Local \textit{a posteriori} error bounds]\label{thm:localaposteriori}
	If ${\bf A_1}$ holds and $\dt < \abs{\alpha_0^n}/(\abs{\beta_0^n}\lipschitzConstant)$, then
 \begin{align} \label{eq:localFinalGalCombine}
 \begin{split} 
	 \norm{\errorGal{n}} \leq& 
	 \sum_{\ell=0}^k(1 + \normedQuantityG)\epsilonCoeff{\ell}{n}
	 \norm{(\identity - \GalProj)\fSimple{\stateInit +
	 \podstate\stateGal{n-\ell}}{\timeArg{n-\ell}}}+
	 \sum_{\ell=1}^k\gammaCoeff{\ell}{n}\norm{\errorGal{n-\ell}}
	  \end{split} \\
 \begin{split} 
	 \norm{\errorLSPG{n}} \leq& 
	 \sum_{\ell=0}^k\epsilonCoeff{\ell}{n}
	 \norm{(\identity - \LSPGProj)\fSimple{\stateInit +
	 \podstate\stateLSPG{n-\ell}}{\timeArg{n-\ell}}}
		 +\normedQuantityP\sum_{\ell=0}^k\epsilonCoeff{\ell}{n}
		 \norm{(\identity - \LSPGProjT\LSPGProj)\fSimple{\stateInit +
		\podstate \stateLSPG{n-\ell}}{\timeArg{n-\ell}}}
 \\
 &+
 \frac{\normedQuantityP}{\hCoeff{n}}\norm{\ProjDiff}\sum_{\ell=0}^k\abs{\alpha_\ell^n}\norm{\stateLSPG{n-\ell}}  
 +\sum_{\ell=1}^k\gammaCoeff{\ell}{n}\norm{\errorLSPG{n-\ell}}.
	  \end{split} 
	\end{align}
\end{theorem}
\begin{proof}
	Adding the premultiplication of Eq.~\eqref{eq:localErrOneP} by
	$\nullBasisFull$ to the premultiplication of Eq.~\eqref{eq:rowspaceLater}
	by $\rightSingPFull$ and noting  that 
	$\rightSingPFullT=\SingPInv\leftSingPT\meshMapping\LSPGProjT$
	yields
 \begin{align} \label{eq:rowspaceLater2}
 \begin{split} 
	 \alpha^n_0\errorLSPG{n} =& 
	 \beta_0^n\dt \left[
		 \fSimple{\stateInit
		 + \stateFOM{n}}{\timeArg{n}} -
		 \compositeProj\fSimple{\stateInit +\podstate \stateLSPG{n}}{\timeArg{n}}
 \right]
- \alpha_0^n\rightSingPFull\SingPInv\leftSingPT\meshMapping\ProjDiff\stateLSPG{n}  \\
&+ \nullBasisPFull\nullBasisPFullT\resLSPGDiff+
\rightSingPFull\SingPInv\leftSingPT\meshMapping\resLSPGDiffP\reviewerB{.}
	  \end{split} 
	  \end{align} 
	Adding and subtracting  $\fSimple{\stateInit
		 +\podstate \stateLSPG{n}}{\timeArg{n}} + \rightSingPFull\rightSingPFullT\LSPGProj
		 \fSimple{\stateInit
		 +\podstate \stateLSPG{n}}{\timeArg{n}}$ from the bracketed quantity, using
		 $\rightSingPFullT\rightSingPFullT\LSPGProj =
		 \rightSingPFull\SingPInv\leftSingPT\meshMapping\LSPGProjT\LSPGProj$, applying the
		 triangle inequality, and using Lipschitz continuity yields
 \begin{align} \label{eq:rowspaceLaterInequality}
 \begin{split} 
	 \abs{\alpha^n_0}\norm{\errorLSPG{n}} \leq& 
	 \abs{\beta_0^n}\dt \left[
		 \lipschitzConstant\norm{\errorLSPG{n}}
		  +
			\norm{(\identity - \LSPGProj)\fSimple{\stateInit +
		\podstate	\stateLSPG{n}}{\timeArg{n}}}
		 +
		 \norm{\SingPInv\leftSingPT\meshMapping}
		 \norm{(\identity - \LSPGProjT\LSPGProj)\fSimple{\stateInit +
		\podstate \stateLSPG{n}}{\timeArg{n}}}
 \right]\\
 &+
 \abs{\alpha_0^n}\norm{\SingPInv\leftSingPT\meshMapping\ProjDiff}\norm{\stateLSPG{n}}  
 + \norm{\nullBasisPFull\nullBasisPFullT\resLSPGDiff+
 \rightSingPFull\SingPInv\leftSingPT\meshMapping\resLSPGDiffP}.
	  \end{split} 
	\end{align}
	Now, using $\dt < \abs{\alpha_0^n}/(\abs{\beta_0^n}\lipschitzConstant)$ yields
 \begin{align} \label{eq:firstCombine}
 \begin{split} 
	 \norm{\errorLSPG{n}} \leq& 
	 \epsilonCoeff{0}{n}
	 \norm{(\identity - \LSPGProj)\fSimple{\stateInit +
	\podstate\stateLSPG{n}}{\timeArg{n}}}
		 +\epsilonCoeff{0}{n}
		 \normedQuantityP
		 \norm{(\identity - \LSPGProjT\LSPGProj)\fSimple{\stateInit +
		\podstate\stateLSPG{n}}{\timeArg{n}}}
 \\
 &+
 \frac{\normedQuantityP}{\hCoeff{n}}\norm{\ProjDiff}\abs{\alpha_0^n}\norm{\stateLSPG{n}}  
 +\frac{1}{\hCoeff{n}} \norm{\nullBasisPFull\nullBasisPFullT\resLSPGDiff+
 \rightSingPFull\SingPInv\leftSingPT\meshMapping\resLSPGDiffP}.
	  \end{split} 
	\end{align}
	Next, we estimate $\norm{\nullBasisPFull\nullBasisPFullT\resLSPGDiff+
	\rightSingPFull\SingPInv\leftSingPT\meshMapping\resLSPGDiffP}$.
\begin{align}
	\nullBasisPFull\nullBasisPFullT\resLSPGDiff &=
\nullBasisPFull\nullBasisPFullT\resFOMPrev-\nullBasisPFull\nullBasisPFullT\LSPGProj\resPrev{\podstate\stateLSPG}\\
\begin{split}
\rightSingPFull\SingPInv\leftSingPT\meshMapping\resLSPGDiffP
& = 
\rightSingPFull\SingPInv\leftSingPT\meshMapping
\LSPGProjT\resPrev{\stateFOM}-\rightSingPFull\SingPInv\leftSingPT\meshMapping
\resPrev{\podstate\stateLSPG}\\
& = 
\rightSingPFull\rightSingPFullT\resPrev{\stateFOM}-\rightSingPFull\SingPInv\leftSingPT\meshMapping
\resPrev{\podstate\stateLSPG}.
\end{split}
\end{align}
Thus,
\begin{align}
	\begin{split}
		\nullBasisPFull\nullBasisPFullT\resLSPGDiff &+
	\rightSingPFull\SingPInv\leftSingPT\meshMapping\resLSPGDiffP=\\
	&\sum_{\ell=1}^k\beta_\ell^n\dt\left[\fSimple{\stateInit+\stateFOM{n-\ell}}{\timeArg{n-\ell}}
	-\compositeProj\fSimple{\stateInit+\podstate\stateLSPG{n-\ell}}{\timeArg{n-\ell}}\right]
\\
&-\sum_{\ell=1}^k\alpha^n_\ell\left[\stateFOM{n-\ell}-\compositeProj\podstate\stateLSPG{n-\ell}\right].
	\end{split}
\end{align}
Adding and subtracting
$\sum_{\ell=1}^k\beta_\ell^n\dt(\fSimple{\stateInit
		 +\podstate\stateLSPG{n-\ell}}{\timeArg{n-\ell}} + \rightSingPFull\rightSingPFullT\LSPGProj
		 \fSimple{\stateInit
		 +\podstate\stateLSPG{n-\ell}}{\timeArg{n-\ell}}) +
	\sum_{\ell=1}^k
	\alpha_\ell^n\rightSingPFull\rightSingPFull\LSPGProj\podstate\stateLSPG{n-\ell}
		 $, using $\LSPGProj \podstate = \podstate$, and applying the triangle
		 inequality yields
		 \begin{align}\label{eq:residualCombine}
	\begin{split}
		\norm{\nullBasisPFull\nullBasisPFullT\resLSPGDiff &+
	\rightSingPFull\SingPInv\leftSingPT\meshMapping\resLSPGDiffP}\leq\\
	&\sum_{\ell=1}^k\abs{\beta_\ell^n}\dt\left[\lipschitzConstant\norm{\errorLSPG{n-\ell}}
	+
	\norm{(\identity-\LSPGProj)\fSimple{\stateInit+\podstate\stateLSPG{n-\ell}}{\timeArg{n-\ell}}}
+\normedQuantityP\norm{(\identity-\LSPGProjT\LSPGProj)\fSimple{\stateInit+\podstate\stateLSPG{n-\ell}}{\timeArg{n-\ell}}}\right]
\\
&+\sum_{\ell=1}^k\abs{\alpha^n_\ell}\left[\norm{\errorLSPG{n-\ell}}
+\normedQuantityP\norm{\ProjDiff}\norm{\stateLSPG{n-\ell}}\right].
	\end{split}
\end{align}
Combining inequalities \eqref{eq:firstCombine} and \eqref{eq:residualCombine}
yields the final result. The Galerkin result is derived similarly, except we
note that $\GalProj^T\GalProj = \identity$ and the $\ProjDiff$ term associated
with Galerkin projection is zero.
\end{proof}

Comparing Theorem \ref{thm:localaposteriori} with \cite[Theorem
6.1]{carlbergGalDiscOpt} shows that the state-space error bounds for the
conservative Galerkin and conservative LSPG models are in general larger than
the bounds for their unconstrained counterparts; the conservative Galerkin
bound has the addition of the constant  $\normedQuantityG$, while the second
and third terms in the conservative LSPG bound are added. This is to be
expected, as the proposed models do not strictly minimize their associated
residuals; they do so only subject to the satisfaction of nonlinear equality
constraints. Thus, general state-space error bounds that are related to the
full-space residual alone will lead to larger bounds. Instead, if we consider
the components of the error associated with the constraints themselves, we can
derive more favorable bounds.

\begin{lemma}[Local \textit{a posteriori} error bounds in conserved
	quantities]\label{thm:boundConserved}
	The error in the conserved quantities can be bounded as
\begin{align}\label{eq:firstThmResultGalSimple}
	\norm{\meshMapping\errorGal{n}} &\leq  
\sum_{\ell=0}^k
\frac{\abs{\beta_\ell^n}\dt}{\abs{\alpha_0^n}}
\norm{\meshMapping\fSimple{\stateInit +
\stateFOM{n-\ell}}{\timeArg{n-\ell}}-
\meshMapping\fSimple{\stateInit +
\podstate\stateGal{n-\ell}}{\timeArg{n-\ell}}
}+
\sum_{\ell=1}^k
\frac{\abs{\alpha_\ell^n}}{\abs{\alpha_0^n}}
\norm{\meshMapping\errorGal{n-\ell}}\\
\norm{\meshMapping\errorLSPG{n}} &\leq  
\sum_{\ell=0}^k
\frac{\abs{\beta_\ell^n}\dt}{\abs{\alpha_0^n}}
\norm{\meshMapping\fSimple{\stateInit +
\stateFOM{n-\ell}}{\timeArg{n-\ell}}-
\meshMapping\fSimple{\stateInit +
\podstate\stateLSPG{n-\ell}}{\timeArg{n-\ell}}
}+
\sum_{\ell=1}^k
\frac{\abs{\alpha_\ell^n}}{\abs{\alpha_0^n}}
\norm{\meshMapping\errorLSPG{n-\ell}}.
\end{align}
\end{lemma}
\begin{proof}
	The result can be obtained trivially by 
subtracting Eq.~\eqref{eq:galDecomp} 
from the premultiplication of Eq.~\eqref{eq:FOMerror} by $\meshMapping$ and
applying the triangle inequality.
\end{proof}

Lemma \ref{thm:boundConserved}---while very simple---highlights an important
attribute of the proposed methods. In particular, the new contribution to the
error at time instance $\timeArg{n}$ is due to a term comprising a scalar
multiple of
$ \norm{\meshMapping\fSimple{\stateInit +
\stateFOM{n}}{\timeArg{n}}- \meshMapping\fSimple{\stateInit +
\podstate\stateGal{n}}{\timeArg{n}} }$ and
$\norm{\meshMapping\fSimple{\stateInit +
\stateFOM{n}}{\timeArg{n}}- \meshMapping\fSimple{\stateInit +
\podstate\stateLSPG{n}}{\timeArg{n}} }$ for conservative Galerkin and
conservative LSPG, respectively. In the absence of source terms, this error
associates with the \textit{error in the flux along the faces $\subdomainFaceSet$ of the
decomposed mesh $\meshDecomp$}. In many cases, this error will be quite small.
For example, in the case of global conservation characterized by $\meshDecomp
= \meshDecompGlobal$, $\nSubdomains = 1$, $\subdomainArg{1} = \spaceDomain$
and $\subdomainInterfaceArg{1} =
\Interface$, the error in the globally conserved quantities arises entirely
from the error in the flux computed along the boundary of the domain.

\begin{theorem}[Local \textit{a posteriori} error bounds in conserved
	quantities]\label{thm:boundConserved}
	If ${\bf A_1}$ holds and $\dt < \abs{\alpha_0^n}/(\abs{\beta_0^n}\lipschitzConstant\cond{\meshMapping})$, then 
\begin{align}\label{eq:firstThmResultGal}
	\norm{\meshMapping\errorGal{n}} &\leq  
\sum_{\ell=0}^k
\epsilonCoeffC{\ell}{n}\lipschitzConstant
\norm{\nullBasisCFull^T\errorGal{n-\ell}}+
\sum_{\ell=1}^k
\gammaCoeffC{\ell}{n}
\norm{\meshMapping\errorGal{n-\ell}}\\
\norm{\meshMapping\errorLSPG{n}} &\leq  
\sum_{\ell=0}^k
\epsilonCoeffC{\ell}{n}\lipschitzConstant
\norm{\nullBasisCFull^T\errorLSPG{n-\ell}}+
\sum_{\ell=1}^k
\gammaCoeffC{\ell}{n}
\norm{\meshMapping\errorLSPG{n-\ell}},
\end{align}
where
$\epsilonCoeffC{\ell}{m}\defeq\abs{\beta_\ell^m}\norm{\meshMapping}\dt/\hCoeffC{m}$, 
$\gammaCoeffC{\ell}{m}\defeq(\abs{\alpha_\ell^m} + \abs{\beta_\ell^m}\lipschitzConstant\cond{\meshMapping}\dt)/\hCoeffC{m}$, 
$\hCoeffC{m}\defeq\abs{\alpha_0^m} - \abs{\beta_0^m}\lipschitzConstant\cond{\meshMapping}\dt$.
\end{theorem}
\begin{proof}
Subtracting Eq.~\eqref{eq:galDecomp} 
from the premultiplication of Eq.~\eqref{eq:FOMerror} by $\meshMapping$ yields
\begin{equation}
\alpha^n_0\meshMapping\errorGal{n} = \beta_0^n\dt\meshMapping\left(
\fSimple{\stateInit + \stateFOM{n}}{\timeArg{n}} - \fSimple{\stateInit +
\podstate\stateGal{n}}{\timeArg{n}}\right) + \meshMapping\resGalDiffC,
\end{equation}
where $\resGalDiffC\defeq\resFOMPrev-\resGalPrevC$.
Applying the triangle inequality yields
\begin{equation}
	\abs{\alpha^n_0}\norm{\meshMapping\errorGal{n}} \leq
	\abs{\beta_0^n}\dt\norm{\meshMapping}\norm{
\fSimple{\stateInit + \stateFOM{n}}{\timeArg{n}} - \fSimple{\stateInit +
\podstate\stateGal{n}}{\timeArg{n}}} + \norm{\meshMapping\resGalDiffC}.
\end{equation}
Now, using Lipschitz continuity and $\bds y = \rightSingCFull\rightSingCFull^T\bds y + \nullBasisCFull\nullBasisCFull^T\bds y$ for all $\bds y\in\RR{\ndof}$ with $\pinvmeshMapping\meshMapping = \rightSingCFull\rightSingCFull^T$, we have
\begin{equation}
	\abs{\alpha^n_0}\norm{\meshMapping\errorGal{n}} \leq
	\abs{\beta_0^n}\dt\norm{\meshMapping}\lipschitzConstant\left(\norm{\pinvmeshMapping}\norm{
	\meshMapping\errorGal{n}}  + \norm{\nullBasisCFull^T\errorGal{n}}\right) +
	\norm{\meshMapping\resGalDiffC}.
\end{equation}
Now, using $\dt <
\abs{\alpha^n_0}/(\abs{\beta_0^n}\lipschitzConstant\cond{\meshMapping})$ and
$\cond{\meshMapping} =
\sigma_\mathrm{max}(\meshMapping)/\sigma_\mathrm{min}(\meshMapping)=\norm{\meshMapping}\norm{\pinvmeshMapping}$, we have
\begin{equation}
\label{eq:localOneC}
	\norm{\meshMapping\errorGal{n}} \leq  
	\frac{\abs{\beta_0^n}\dt\norm{\meshMapping}\lipschitzConstant}{\hCoeffC{n}}
	\norm{\nullBasisCFull^T\errorGal{n}}+
	\frac{1}{\hCoeffC{n}}\norm{\meshMapping\resGalDiffC}.
\end{equation}
	Next, we estimate $\norm{\meshMapping\resGalDiffC}$. First, we have
	 \begin{equation} 
	\meshMapping\resGalDiffC=\sum_{\ell=1}^k\beta_\ell^n\dt\meshMapping\left(\fSimple{\stateInit+\stateFOM{n-\ell}}{\timeArg{n-\ell}}-\fSimple{\stateInit+\podstate\stateGal{n-\ell}}{\timeArg{n-\ell}}\right)-\sum_{\ell=1}^k\alpha^n_\ell\meshMapping\left(\stateFOM{n-\ell}-\podstate\stateGal{n-\ell}\right).
	 \end{equation} 
	Applying the triangle inequality and following the above steps yields
	 \begin{equation} \label{eq:finaldRC}
		 \norm{\meshMapping\resGalDiffC}\leq\sum_{\ell=1}^k\abs{\beta_\ell^n}\lipschitzConstant\dt\left(
			 \cond{\meshMapping}\norm{\meshMapping\errorGal{n-\ell}}
		 +\norm{\meshMapping}\norm{\nullBasisCFull^T\errorGal{n-\ell}}\right) 
		 +\sum_{\ell=1}^k\alpha^n_\ell\norm{\meshMapping\errorGal{n-\ell}}.
	 \end{equation} 
	Combining inequalities \eqref{eq:localOneC} and \eqref{eq:finaldRC} produces the final result.
\end{proof}

Theorem \ref{thm:boundConserved} shows that---at a given time instance
$\timeArg{n}$---the only new contribution to the error bound arises from the
term $\epsilonCoeffC{0}{n}\lipschitzConstant
\norm{\nullBasisCFull^T\errorGal{n}}$, which associates with error incurred in
the null space to the constraint matrix $\meshMapping$. Thus, even though the
methods explicitly enforce conservation over subdomains, the actual values of
those conserved variables may deviate from their full-order-model
counterparts. This can be interpreted as a closure problem: the errors in the
state component not restricted by the constraints can lead to errors in the
state component restricted by the constraints, i.e., the conserved variables.


	 \section{Numerical experiments}\label{sec:experiments}

This section compares the performance of several reduced-order models on a
parameterization of the quasi-1D Euler equations applied to supersonic flow in
a converging--diverging nozzle.

\subsection{Problem description: quasi-1D Euler equation}
We consider a parameterized quasi-1D Euler equation associated with
modeling inviscid compressible flow in a one-dimensional converging--diverging nozzle with a
continuously varying cross-sectional area \cite[Chapter 13]{maccormackNote};
Figure \ref{fig:converging-diverging_nozzle} depicts the problem geometry. 
\begin{figure}[htbp] 
  \centering 
  \includegraphics[width=0.5\textwidth]{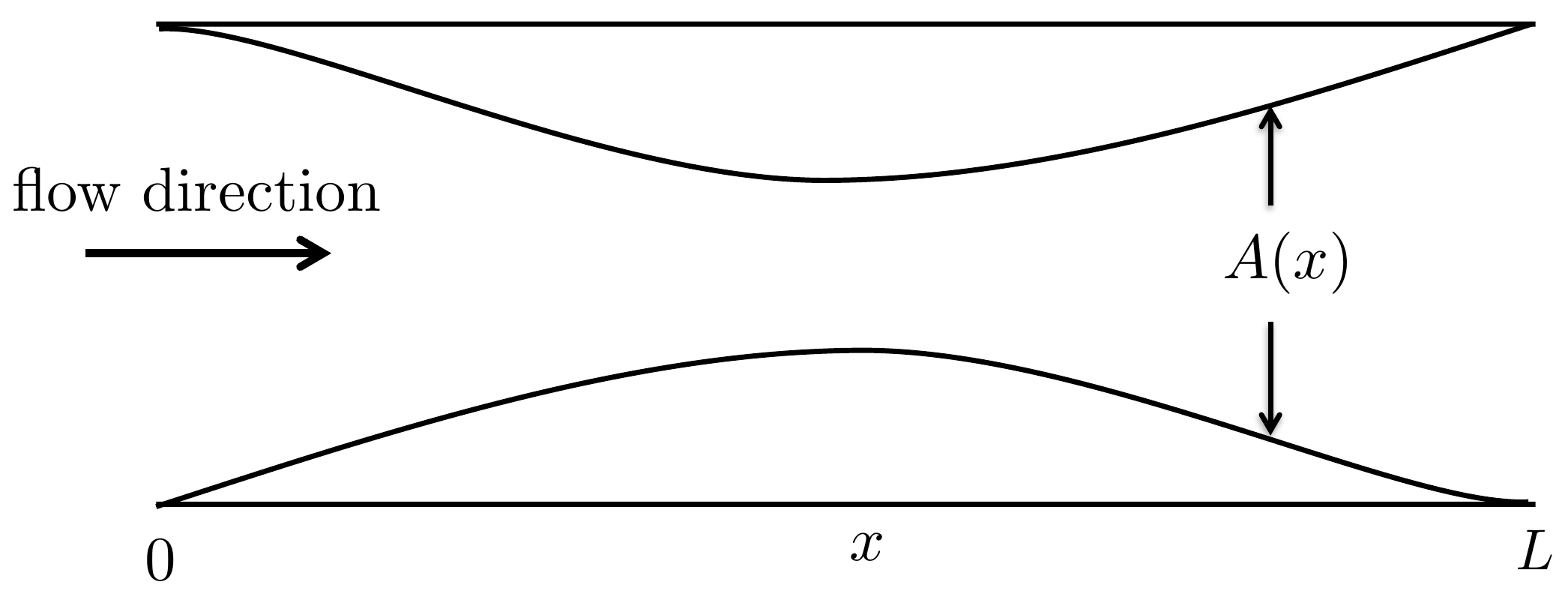} 
	\caption{\textit{Quasi-1D Euler.} \reviewerA{Problem geometry for the
		converging--diverging nozzle.}} 
  \label{fig:converging-diverging_nozzle} 
\end{figure}
\reviewerA{
	In integral form, the governing equations are:
	\begin{align*}\label{eq:PDEexperiment}
	&\frac{d}{dt}\int_{\generalSubdomain}{\crossareaSymb(\spaceVarScalar)\densitySymb}(\spaceVarScalar,\timeVar
		;\params)\dVolScalar&& +
\int_{\generalInterface}\crossareaSymb(\spaceVarScalar)
\densitySymb(\spaceVarScalar,\timeVar
;\params) \uvelocitySymb(\spaceVarScalar,\timeVar
;\params)\sign{\normalVecScalar(\spaceVarScalar)}\dInterfaceScalar =
0
\\
&\frac{d}{dt}\int_{\generalSubdomain}\crossareaSymb(\spaceVarScalar)\densitySymb(\spaceVarScalar,\timeVar
;\params) \uvelocitySymb(\spaceVarScalar,\timeVar
;\params)\dVolScalar&& +
\int_{\generalInterface}
\crossareaSymb(\spaceVarScalar)\left(\densitySymb(\spaceVarScalar,\timeVar
;\params) \uvelocitySymb(\spaceVarScalar,\timeVar
;\params)^2 + \pressureSymb(\spaceVarScalar,\timeVar
;\params)\right)
\sign{\normalVecScalar(\spaceVarScalar)}\dInterfaceScalar \\
& &&=
\int_{\generalSubdomain}
\pressureSymb(\spaceVarScalar,\timeVar
;\params)\frac{\partial \crossareaSymb}{\partial x}(\spaceVarScalar,\timeVar
;\params)
\dVolScalar\\
&\frac{d}{dt}\int_{\generalSubdomain}\crossareaSymb(\spaceVarScalar){\energydensity}(\spaceVarScalar,\timeVar
;\params)\dVolScalar&& +
\int_{\generalInterface}\crossareaSymb(\spaceVarScalar)
(\energydensity(\spaceVarScalar,\timeVar
;\params)+\pressureSymb(\spaceVarScalar,\timeVar
;\params))\uvelocitySymb(\spaceVarScalar,\timeVar
;\params)
\sign{\normalVecScalar(\spaceVarScalar)}\dInterfaceScalar =
0,
\end{align*}
$\forall\generalSubdomain\subseteq
\spaceDomain = [0,\domainLength]$.
}
Thus, the governing system of nonlinear partial differential equations is consistent
with the conservation-law formulation in Eq.~\eqref{eq:PDE}
with $\nspaceDim=1$ spatial dimension, $\nConservation = 3$ conserved variables corresponding to 
density $\conservedQuantity{1} = \reviewerA{\crossareaSymb}\densitySymb$, momentum $\conservedQuantity{2} = \reviewerA{\crossareaSymb}\densitySymb \uvelocitySymb$, and 
energy density $\conservedQuantity{3} = \reviewerA{\crossareaSymb}\energydensity$. The flux corresponds to 
$\fluxScalarArg{1}=\reviewerA{\crossareaSymb}\densitySymb \uvelocitySymb$,
$\fluxScalarArg{2}=\reviewerA{\crossareaSymb}\reviewerA{(}\densitySymb
\uvelocitySymb^2 + \pressureSymb\reviewerA{)}$, and
$\fluxScalarArg{3}=\reviewerA{\crossareaSymb}(\energydensity+\pressureSymb)\uvelocitySymb$, and the source corresponds to
$\sourceEntry{1}=\sourceEntry{3}=0$ and
$\sourceEntry{2}=\pressureSymb\frac{\partial \crossareaSymb}{\partial x}$. 
In addition, we have
      $\pressureSymb = (\specificheat-1)\densitySymb\energypermass$,
$\energypermass = \frac{\energydensity}{\densitySymb} -
\frac{\uvelocitySymb^2}{2}$,
and assume a perfect gas (i.e., $\pressureSymb = \densitySymb \specificgasconstant \temperatureSymb$).
Here,
$\densitySymb$ denotes density,
$\uvelocitySymb$ denotes velocity,
$\pressureSymb$ denotes pressure, 
$\energypermass$ denotes potential energy per unit mass,
$\energydensity$ denotes total energy density,
$\specificheat$ denotes the specific heat ratio, and
$\crossareaSymb$ denotes the converging--diverging nozzle cross-sectional area.
We employ a
specific heat ratio of $\specificheat=1.3$ and
a specific gas constant of $\specificgasconstant=355.4$
$\text{m}^2/\text{s}^2/\text{K}$.
The spatial domain is $\spaceDomain = [0,\domainLength]$ with $\domainLength = 0.25$ m.
The cross-sectional area $A(x)$ is determined by a cubic spline interpolation over the points 
\begin{align}
\begin{split}
(x,A(x)) \in \{& (0,0.035), (0.0208,0.0275), (0.0417,0.0206), 
          (0.0625, 0.0145), (0.0833,0.0097),\\
					 &(0.104,0.0066), (0.125,0.0055),   (0.146,0.0067), (0.1667,0.0107),
          (0.188,0.0178),\\
&  (0.208,0.0283), (0.229,0.0427),
          (0.25,0.0612) \}.
\end{split}
\end{align}
The final time is $\totaltime = 0.29$ s, and we 
employ the backward Euler scheme 
with a uniform time step of $\dt = 0.01$ s
for time discretization. We declare
convergence of the
Newton(-like) solver at each time instance when the
$\ell^2$-norm of the residual reaches $1\times 10^{-5}$ of its value with the
initial guess, which is provided by the solution at the previous time
instance.

The initial flow field is created in several steps.
First, the following isentropic relations are used to generate a zero pressure-gradient flow field at the inlet ($x=0$ m) and the outlet ($x=0.25$ m):
 \begin{align}
   \begin{split}
     \machSymb(x) &= \frac{\machSymb_m\crossareaSymb_m}{\crossareaSymb(x)}\left ( \frac{1+\frac{\specificheat-1}{2}\machSymb(x)^2}{1+\frac{\gamma-1}{2}\machSymb_m^2} \right
  )^{\frac{\specificheat+1}{2(\specificheat-1)}},\quad x\in\{0,0.25\}\ \text{m},
\label{eq:firstIsen}
   \end{split}
 \end{align} 
 where a subscript $m$ indicates the flow quantity at $x=0.125$ m, and
 $\machSymb$ denotes the
 Mach number. The initial Mach number at the middle of the domain is employed as the
problem parameter (i.e., $\paramSymb = \machSymb_m$ with $\nparam = 1$), from which the initial distribution of the Mach number is defined according to a
cubic-spline
interpolation with points $\{\machSymb(0), \paramSymb, \machSymb(0.25)\}$.
 Then, we use the following relations to obtain the rest of the initial flow field for $x\in\spaceDomain$:
 \begin{gather}
  \pressureSymb(x) = \pressureSymb_t\left (1+\frac{\specificheat-1}{2}\machSymb(x)^2\right )^{\frac{-\specificheat}{\specificheat-1}},\quad
   \temperatureSymb(x) = \temperatureSymb_t\left
	(1+\frac{\specificheat-1}{2}\machSymb(x)^2\right )^{-1},\\
  \label{eq:lastIsen}\densitySymb(x) = \frac{\pressureSymb(x)}{R\temperatureSymb(x)},\quad
  \speedofsoundSymb(x) =
	\sqrt{\specificheat\frac{\pressureSymb(x)}{\densitySymb(x)}} ,\quad
  \uvelocitySymb(x) = \machSymb(x) \speedofsoundSymb(x),
 \end{gather} 
where $\speedofsoundSymb$ denotes the speed of sound,
the total temperature is $\temperatureSymb_t = 2800$ K, and
the total pressure is $\pressureSymb_t = 2.068\times10^6$ $\text{N}/\text{m}^2$.

\subsection{Compared methods}
These experiments compare the following methods, which employ
the Tier A-B notation established in Section \ref{sec:hyperreduction2}:
\begin{itemize} 
	\item \textit{\FOM}. This model corresponds to the full-order model, i.e.,
		the solution satisfying Eq.~\eqref{eq:fom}.
	\item \textit{\Gal}. This model corresponds to the Tier 1-0 Galerkin ROM.
			\item \textit{\LSPG}. This model corresponds to the Tier 1-0 LSPG ROM.
	\item \textit{\LSPGcnstd}. This model corresponds to the Tier 1-1 LSPG ROM,
		which is conservative.
	\item \textit{\GNAT}. This model corresponds to the Tier 2-0 LSPG ROM,
		where the residual approximation $\resHyper$ is constructed using
		hyper-reduction method \ref{hyper:res} with gappy POD as described in Section
		\ref{sec:hyperreduction}. The snapshots used to construct $\podres$ are
		constructed during simulation of the \FOM\ method at training instances,
		i.e., the residual tensor $\tensorres(\state)$ is employed as described in
		Section \ref{sec:training}. 
		This corresponds to the GNAT method \cite{CarlbergGappy,carlbergJCP}. 
	\item \textit{\GNATcnstd}. This model corresponds to the Tier 2-1 LSPG ROM,
		which is conservative. While objective function is approximated in the same
		way as in the \GNAT\ method above, no hyper-reduction is applied to the
		constraints.
	\item \textit{\GNATcnstdArg{X}}. This model corresponds to the Tier 2-2
		LSPG ROM, which is \textit{approximately} conservative. The objective
		function is approximated in the same way as in the \GNAT\ method above.
		The residual approximation $\resHyperDecomp$, which appears in the
		constraints, is approximated using hyper-reduction method
		3 with gappy POD as described in Section
		\ref{sec:hyperreduction2}. The snapshots used to construct the required
		reduced-basis matrices $\podsource$ and
		$\podflux$ are
		constructed from data tensors $\tensorsource(\stateDummy)$ and
		$\tensorflux(\stateDummy)$, where $\stateDummy$ corresponds to the
		Method `X' state and X varies during the experiments. The sample
		matrices satisfy $\sampleMatsource=\sampleMatres$ and
		$\sampleMatres\alignment =
		\sampleMatres\alignment\sampleMatflux^T\sampleMatflux$ as described in
		Section \ref{sec:training} such that a single sample mesh can be employed
		for all approximations.
\end{itemize}

In all cases that employ constraints (i.e., Tier A-B ROMs with
$\mathrm{B}\in\{1,2\}$), the subdomains defining the decomposed mesh
$\meshDecomp$ are equally spaced, 
their union is equal to
the global domain (i.e., 
$
\cup_{i=1}^{\nSubdomains}\completesubdomainArg{i}=\spaceDomain $
), and are
non-overlapping (i.e.,
$\measure{\subdomainArg{i}\cap\subdomainArg{j}} = 0$ for $i\neq j$) such that
feasibility implies that all conservative ROMs are globally
conservative (Corollary \ref{cor:global}). If infeasibility is detected 
with $\ndofDecomp \leq \nstateRed$, then 
infeasibility-handling approach \ref{infeas:coarsen} in Section
\ref{sec:infeasible} is employed at that time instance; this amounts to
coarsening the decomposed mesh $\meshDecomp$ by reducing the number of
subdomains by one  and updating the operator
$\meshMapping$ accordingly; in all cases, global conservation was feasible. If instead
$\ndofDecomp > \nstateRed$, in which case feasibility cannot be guaranteed in
general (see
Proposition \ref{prop:sufficientFeasGal}), then 
infeasibility-handling approach \ref{infeas:penalty} in Section
\ref{sec:infeasible} is employed; this amounts to applying a penalty
formulation with a specified penalty parameter $\penaltyParam\in\RRplus{}$.
In all cases, a (Newton) step length of $\linesearchParam{n(k)}=1$ was
employed and led to convergence of the solution to the system of nonlinear
equations arising at each time instance. 
All ROMs employ a training set of $\paramDomainTrain
= \{1.7+0.1j\}_{j=0}^3$ such that $\ntrain = 4$. The online parameter instance
at which the ROMs are simulated is set to $\paramSymb_\star = 1.75$.

We assess the accuracy of any ROM solution $\stateApprox$ using
two metrics: the mean-squared and time-instantaneous state-space error, i.e.,
\begin{align}\label{eq:relErrorExpREv}
	\relativeError &\defeq \sqrt{\sum_{n=1}^\ntimedof
\norm{\stateArg{n}(\params)-\stateApproxArg{n}(\params)}_2^2}/\sqrt{
\sum_{n=1}^\ntimedof\norm{\stateArg{n}(\params)}_2^2 }\\
\relativeErrorInst{n} &\defeq 
\norm{\stateArg{n}(\params)-\stateApproxArg{n}(\params)}_2/
\norm{\stateArg{n}(\params)}_2,\quad n\innatSeq{\ntimedof}\reviewerB{,}
\end{align}
and the mean-squared and time-instantaneous error in the globally conserved variables
\begin{align}
	\relativeGlobalError &\defeq \sqrt{\sum_{n=1}^\ntimedof
\norm{\meshMappingGlobal\stateArg{n}(\params)-\meshMappingGlobal\stateApproxArg{n}(\params)}_2^2}/\sqrt{
	\sum_{n=1}^\ntimedof\norm{\meshMappingGlobal\stateArg{n}(\params)}_2^2
	}\\
	\relativeGlobalErrorInst{n}
	&\defeq\norm{\meshMappingGlobal\stateArg{n}(\params)-\meshMappingGlobal\stateApproxArg{n}(\params)}_2/\norm{\meshMappingGlobal\stateArg{n}(\params)}_2,\quad
	n\innatSeq{\ntimedof},
\end{align}
where $\meshMappingGlobal\in\RRplus{\nConservation\times\ndof}$ is the
operator $\meshMapping$ associated with the global decomposition
$\meshDecomp=\meshDecompGlobal \defeq
\{\domain\}$. We also assess the mean-squared and time-instantaneous violation in global conservation as
\begin{align}
	\violation &= \sqrt{\sum_{n=1}^\ntimedof
\norm{\meshMappingGlobal\resArgs{n}{\stateApproxArg{n}(\params)}{\params}}\reviewerB{_2}}\\
\violationInst{n} &= 
\norm{\meshMappingGlobal\resArgs{n}{\stateApproxArg{n}(\params)}{\params}}_2,\quad
n\innatSeq{\ntimedof}.
\end{align}

All timings are obtained by performing calculations on an Intel(R) Xeon(R) CPU
E5-2670 @ 2.60 GHz, 31.4 GB RAM using the \verb+MORTestbed+
\cite{zahr2010comparison} in Matlab. 

\subsection{\GNATcnstdArg{X}\ snapshot study}\label{sec:gnatSnapshotStudy}
This section assesses the effect of snapshot-collection method on the
performance of the \GNATcnstdArg{X}\ method; all subsequent experiments employ
the snapshot-collection method yielding the best performance.

We set the number of control volumes to $\nControlVol = 100$ such that $\ndof
= \nControlVol\nConservation = 300$, the
reduced-basis dimensions to $\nstateRed = 5$ (which corresponds to a relative
statistical energy of 99.78\%) and $\nresRed=\nfluxRed$, and employ a sample mesh of 20 control volumes,
which corresponds to $\nsamplesres = \nsamplesvelocity=\nsamplesflux = 60$. We
employ a penalty parameter of $\penaltyParam = 10^3$, which is used by
infeasibility-handling approach \ref{infeas:penalty} when $\nConstraints >
\nstateRed$.
In this setting we vary the number of constraints $\nConstraints$ and
flux-basis dimension $\nfluxRed$ and report the relative mean-squared
violation in global conservation over the time interval, i.e., the value of $\violation$ for the given
reduced-order model divided by the value of $\violation$ for the
(unconstrained) \GNAT\ model; note that this value is zero if the
constraint-approximation error is zero
and a feasible solution is computed at each time
instance.

Figure \ref{fig:nCnFstudy} reports the results for this experiment and 
elucidates several trends. First, Figure \ref{GNATcnstdresults} shows that the \GNATcnstd\
model---for which the constraints are enforced exactly---yields near-exact
satisfaction of the conservation laws for $\nConstraints < \nstateRed$; this
implies that a feasible solution was computed at every time instance of that
simulation. 

Second, we note that for $1<\nConstraints/\nstateRed < 2$, the
\GNATcnstd\ model yields approximate but accurate satisfaction of the conservation
laws, as the relative value of $\violation$ is less than $10^{-2}$ in these
cases. 

Third, Figures \eqref{GNATcnstdArgFOMresults}--\eqref{GNATcnstdArgGNATcnstdresults} show that the best results for the \GNATcnstdArg{X}\ method are
obtained for X=\LSPGcnstd\ (Figure \ref{GNATcnstdArgLSPGcnstdresults}) and
X=\GNATcnstd\ (Figure \ref{GNATcnstdArgGNATcnstdresults}); these techniques yield relative
values of $\violation$ for the \GNATcnstdArg{X}\ model less than
$10^{-2}$ for $\nConstraints/\nstateRed < 2$ in almost all cases. 
This result is sensible, as the training
simulations corresponding to (constrained) \LSPGcnstd\ and \GNATcnstd\ are
`closer' to the (constrained) \GNATcnstdArg{X} simulation relative to the (unconstrained)
\FOM,\ \LSPG, and \GNAT\ simulations.
However, in these
cases, the relative value of $\violation$ for $\nConstraints < \nstateRed$ is
small, but not close to machine zero as in the \GNATcnstd\ case because the
constraints are approximated. Thus, these methods---while having a cost
independent of $\nstate$ due to the introduction of hyper-reduction---are only
\textit{approximately} conservative.  

Fourth, we note that the \GNATcnstdArg{\LSPGcnstd}\ and 
\GNATcnstdArg{\LSPGcnstd}\ results are insensitive to the flux-basis dimension $\nfluxRed$ for 
$\nfluxRed$ sufficiently large ($\nfluxRed > 12$).

Finally, we note that while the \GNATcnstdArg{\LSPGcnstd}\ and
\GNATcnstdArg{\GNATcnstd}\ models yield similar accuracy, the latter method
incurs a lower training cost, as the former incurs training simulations with
the 
(Tier 1-1) \LSPGcnstd\ model, while the latter incurs training simulations
with the 
(Tier 2-1) \GNATcnstd\ model.  Thus, the
only \GNATcnstdArg{X}\ method we consider in subsequent experiments is the \GNATcnstdArg{\GNATcnstd}\
approach.
    \begin{figure}[h]
      \begin{center}
				\begin{subfigure}{0.46\textwidth}\includegraphics[width=\textwidth]{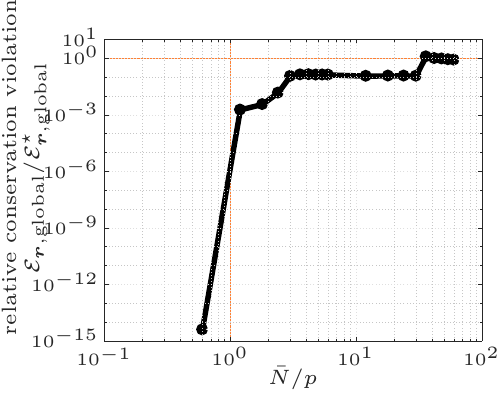}\caption{\GNATcnstd\
				results}\label{GNATcnstdresults}\end{subfigure}
				\begin{subfigure}{0.46\textwidth}\includegraphics[width=\textwidth]{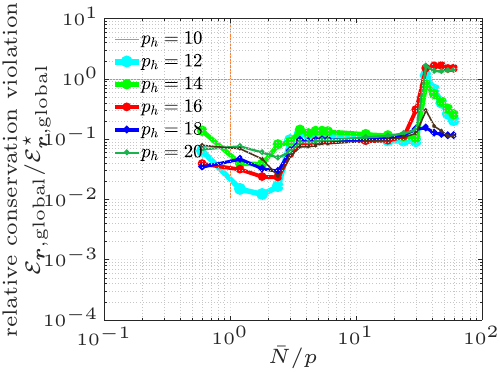}\caption{\GNATcnstdArg{\FOM}
				results}\label{GNATcnstdArgFOMresults}\end{subfigure}
				\begin{subfigure}{0.46\textwidth}\includegraphics[width=\textwidth]{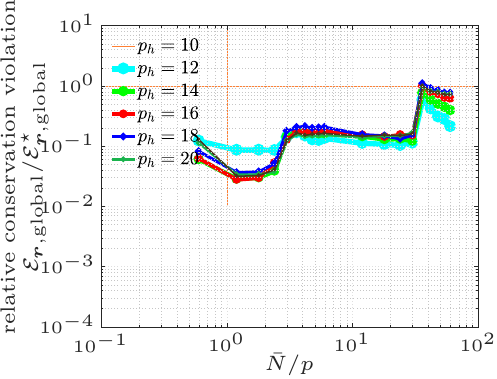}\caption{\GNATcnstdArg{\LSPG}
				results}\label{GNATcnstdArgLSPGresults}\end{subfigure}
				\begin{subfigure}{0.46\textwidth}\includegraphics[width=\textwidth]{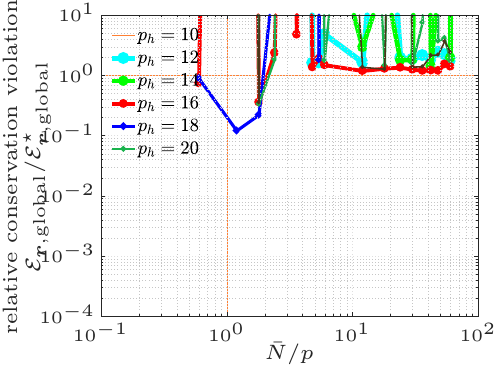}\caption{\GNATcnstdArg{\GNAT}
				results}\label{GNATcnstdArgGNATresults}\end{subfigure}
					\begin{subfigure}{0.46\textwidth}\includegraphics[width=\textwidth]{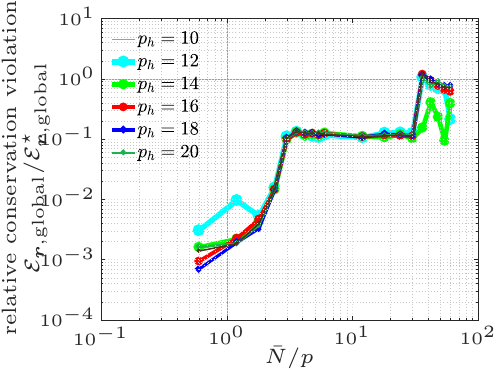}\caption{\GNATcnstdArg{\LSPGcnstd}
					results}\label{GNATcnstdArgLSPGcnstdresults}\end{subfigure}
					\begin{subfigure}{0.46\textwidth}\includegraphics[width=\textwidth]{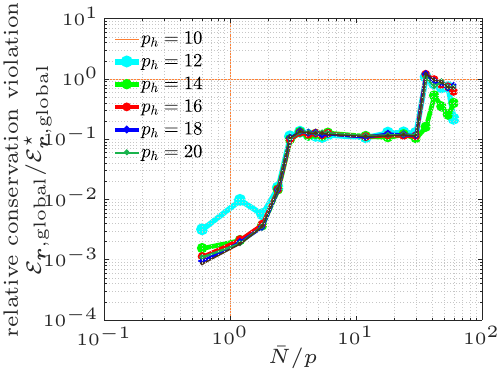}\caption{\GNATcnstdArg{\GNATcnstd}
					results}\label{GNATcnstdArgGNATcnstdresults}\end{subfigure}
      \end{center}
			\caption{\textit{One-dimensional Euler equation}. \GNATcnstd\
				snapshot study \reviewerA{described in Section
			\ref{sec:gnatSnapshotStudy}}. 
			\reviewerA{Figure \ref{GNATcnstdresults} corresponds to the
				(conservative) \GNATcnstd\
				model; other subfigures correspond to different
				(approximately conservative) \GNATcnstdArg{X} methods, which employ
				Method `X' snapshots to construct the required
		reduced-basis matrices $\podsource$ and
		$\podflux$. Within each subfigure, two parameters vary: the
		number of constraints $\nConstraints$ (the reduced-basis dimension
		is fixed to $\nstateRed=5)$, and the dimension of the reduced-basis matrices $\podsource$ and
		$\podflux$, which are enforced to have the same dimension such that
	$\nresRed=\nfluxRed$.} Here, $\violation^\star$ denotes the value of
	$\violation$ obtained for the unconstrained \GNAT\ model.
			\label{fig:nCnFstudy}}
    \end{figure}

		\subsection{Penalty-parameter study}\label{sec:penaltyParamStudy}

This section assesses the effect of the penalty parameter $\penaltyParam$ 
employed by
infeasibility-handling approach \ref{infeas:penalty} when $\nConstraints >
\nstateRed$
on the performance of the (constrained) ROMs \LSPGcnstd, \GNATcnstd, and
\GNATcnstdArg{\GNATcnstd}. All subsequent
experiments employ the penalty parameter yielding the best
performance.

We again set the number of control volumes to $\nControlVol = 100$ such that
$\ndof = \nControlVol\nConservation = 300$, the reduced-basis dimensions to
$\nstateRed = 5$ and $\nresRed=\nfluxRed=\nsourceRed=20$ and again employ a
sample mesh of 20 control volumes, which corresponds to $\nsamplesres =
\nsamplesvelocity=\nsamplesflux = 60$.  We vary the number of constraints
$\nConstraints$ and penalty parameter $\penaltyParam$ and report the
mean-squared state-space error $\relativeError$ and the  (absolute)
mean-squared violation in global conservation over the time interval
$\violation$. We note that a penalty value of $\penaltyParam=\infty$
corresponds to minimizing the norm of the constraints only (i.e., the
objective function is ignored).

Figure \ref{fig:penaltyStudy} reports the results for this experiment. First, we
note that values $\penaltyParam\in\{10,10^2,10^3\}$ yield similar performance,
which outperforms the other tested values. In particular, values of
$\penaltyParam\in\{1,\infty\}$ often yield unstable responses, while
$\penaltyParam=10$ almost always yields larger errors than employing
$\penaltyParam\in\{10,10^2,10^3\}$.

Second, we note that nearly all
cases outperform the unconstrained model, characterized by $\penaltyParam=0$;
this implies that employing the proposed constraints can improve accuracy,
even if the constraints are employed in a penalty formulation rather than as
strictly enforced constraints. 

Third, we observe that the two reported metrics are often correlated: larger
values of 
mean-squared violation in global conservation
$\violation$ typically implies larger values of the
relative mean-squared state-space error $\relativeError$. This lends credibility to the
proposed technique, which aims to reduce the violation in global conservation,
as it suggests that enforcing this constraint (or employing it as a penalty in
the objective function) can lead to more accurate ROMs.

Fourth, the plots indicate that accuracy typically degrades as constraints are
added to the problem, i.e., as the decomposed mesh becomes finer. In
particular, the case $\nConstraints=\ndof$ is equivalent to the unconstrained
case for \LSPGcnstd\ for any value of the penalty parameter $\penaltyParam$,
as $\meshMapping=\identity$ in this case and thus the objective function in
Problem \eqref{eq:LSPGABPenalty} is equal to a scalar multiple of the
objective in the (unconstrained) LSPG Problem
\eqref{eq:LSPGODeltaEStructPresred}. This is not true for the \GNATcnstd\ and
\GNATcnstdArg{\GNATcnstd} \ models, as different hyper-reduction approaches
are employed for the residual in the objective and constraints such that
$\resGalHyperDecomp\neq\resGalHyper$.

In subsequent experiments, we employ a penalty-parameter value of $\penaltyParam = 10^3$.

\begin{figure}[h]
  \begin{center}
	\begin{subfigure}{0.49\textwidth}\includegraphics[width=\textwidth]{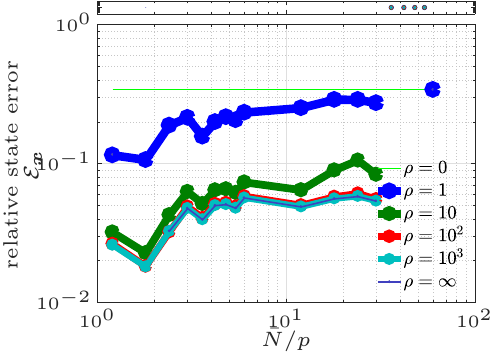}\caption{\LSPGcnstd:
	\reviewerA{$\relativeError$}}\label{}\end{subfigure}
	\begin{subfigure}{0.49\textwidth}\includegraphics[width=\textwidth]{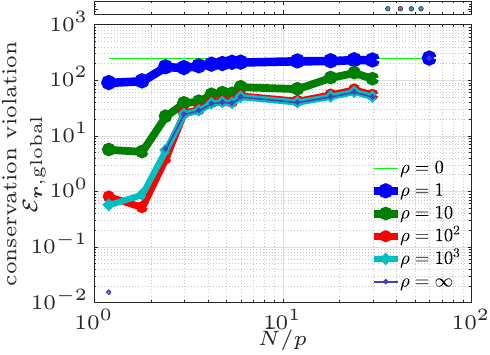}\caption{\LSPGcnstd:
	\reviewerA{$\violation$}}\label{}\end{subfigure}
	\begin{subfigure}{0.49\textwidth}\includegraphics[width=\textwidth]{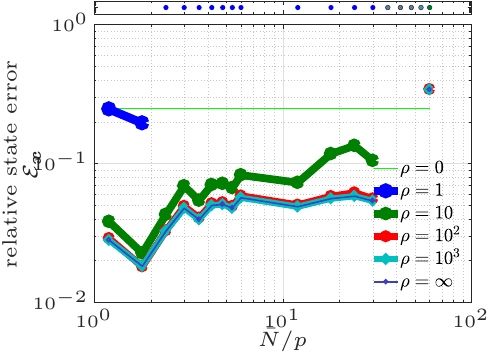}\caption{\GNATcnstd:
	\reviewerA{$\relativeError$}}\label{}\end{subfigure}
	\begin{subfigure}{0.49\textwidth}\includegraphics[width=\textwidth]{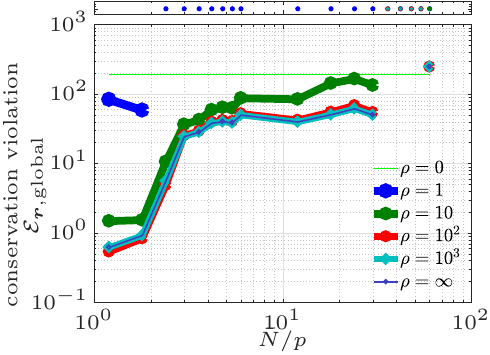}\caption{\GNATcnstd:
	\reviewerA{$\violation$}}\label{}\end{subfigure}
	\begin{subfigure}{0.49\textwidth}\includegraphics[width=\textwidth]{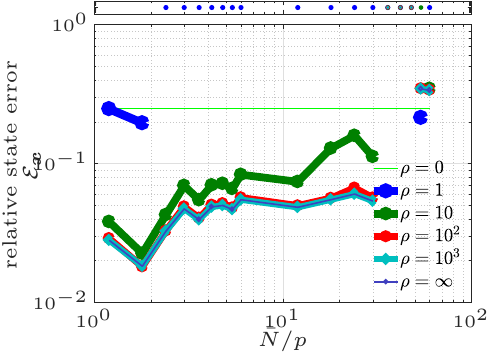}\caption{\GNATcnstdArg{\GNATcnstd}:
	\reviewerA{$\relativeError$}}\label{}\end{subfigure}
	\begin{subfigure}{0.49\textwidth}\includegraphics[width=\textwidth]{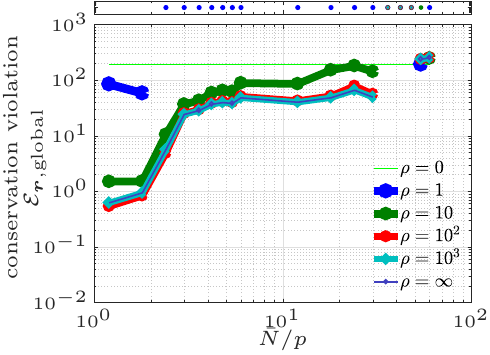}\caption{\GNATcnstdArg{\GNATcnstd}:
	\reviewerA{$\violation$}}\label{}\end{subfigure}
  \end{center}
  \caption{\textit{One-dimensional Euler equation}. \GNATcnstd\
				penalty-parameter study \reviewerA{described in Section
			\ref{sec:penaltyParamStudy}}. The top bar reports cases where the
				reduced-order-model simulation was unstable. 
			\reviewerA{Each row of subfigures corresponds to a different ROM method;
				each column reports a different error measure. Within each subfigure, two parameters vary: the
		number of constraints $\nConstraints$ (the reduced-basis dimension
		is fixed to $\nstateRed=5)$, and the penalty parameter $\penaltyParam$.
		Note that we only consider $\nConstraints/\nstateRed>1$, as infeasibility
does not occur for $\nConstraints/\nstateRed\leq1$.} }
		\label{fig:penaltyStudy}
		\end{figure}

		\subsection{State-basis-dimension study}\label{sec:stateBasisStudy}

This section assesses the effect of basis dimension $\nstateRed$ on the
proposed methods.
We again employ $\nControlVol = 100$ control volumes in the finite-volume
discretization, set
reduced-basis dimensions to  $\nresRed=\nfluxRed=\nsourceRed=20$, employ a
sample mesh with 20 control volumes, and set the penalty parameter to
$\penaltyParam=10^3$.  We vary both the state-basis dimension $\nstateRed$
and the number of constraints
$\nConstraints$ the relative
mean-squared state-space error $\relativeError$ and the  (absolute)
mean-squared violation in global conservation over the time interval
$\violation$. 

Figure \ref{fig:multidomain3} reports the results. First, and most
importantly, we note that Figures \ref{LSPGcnstdrelativestateerror},
\ref{GNATcnstdrelativestateerror}, and \ref{GNATcnstdGNATrelativestateerror}
show that the introduction of constraints yields the most
significant improvements for the smallest basis dimension $\nstateRed=5$. In
these cases,
the relative mean-squared state-space error $\relativeError$ is
reduced by over an order of magnitude for all ROMs, as the
unconstrained ROMs yield errors exceeding 30\%, while their constrained
counterparts employing $\nConstraints=3$ (i.e., global conservation with
$\meshDecomp=\meshDecompGlobal$) all yield errors less than 2\%. In contrast,
for $\nstateRed\geq 7$, the unconstrained ROMs are already quite accurate,
with errors already less then 2\%; incorporating constraints in these cases
does yield accuracy improvements in most cases, although these improvements
are less dramatic. Because the most significant improvements were obtained by
enforcing global conservation with $\nConstraints = 3$, subsequent experiments
employ ROMs that enforce global conservation by using a decomposed mesh of 
$\meshDecomp=\meshDecompGlobal$.

Second, Figures \ref{LSPGcnstdviolation} and \ref{GNATcnstdviolation} show
that the \LSPGcnstd\ and \GNATcnstd\ models produce near-exact satisfaction of
the conservation laws for $\nConstraints < \nstateRed$; this implies that a
feasible solution was computed at every time instance of the corresponding simulations. In
contrast, Figure \ref{GNATcnstdGNATviolation} shows that the
\GNATcnstdArg{\GNATcnstd}\ ROM is only approximately conservative.
Nonetheless, this approximate conservation does not adversely impact the
actual errors produced by the ROM, as the errors reported in Figures \ref{GNATcnstdrelativestateerror}
and \ref{GNATcnstdGNATrelativestateerror} are nearly identical in all cases.
So, while applying hyper-reduction to the constraints results in a loss of
numerically exact satisfaction of global conservation, the results are
extremely similar to the case where the constraints are applied exactly.

\begin{figure}[h]
\begin{center}
\begin{subfigure}{0.49\textwidth}\includegraphics[width=\textwidth]{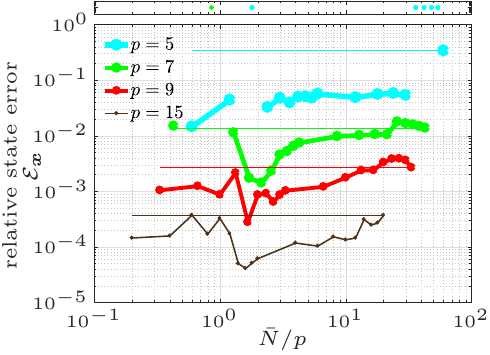}\caption{\LSPGcnstd:
 $\relativeError$}\label{LSPGcnstdrelativestateerror}\end{subfigure}
\begin{subfigure}{0.49\textwidth}\includegraphics[width=\textwidth]{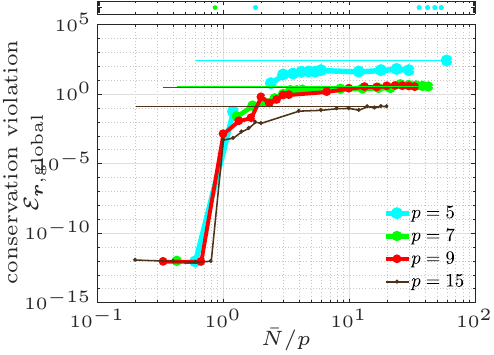}\caption{\LSPGcnstd:
 $\violation$}\label{LSPGcnstdviolation}\end{subfigure}
\begin{subfigure}{0.49\textwidth}\includegraphics[width=\textwidth]{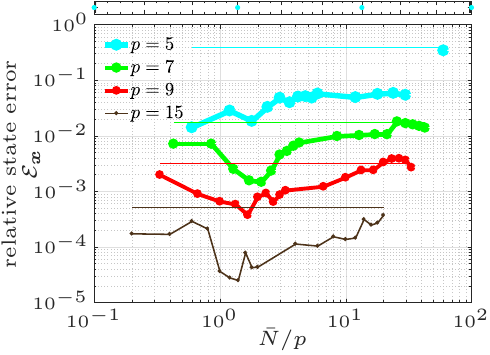}\caption{\GNATcnstd:
 $\relativeError$}\label{GNATcnstdrelativestateerror}\end{subfigure}
 \begin{subfigure}{0.49\textwidth}\includegraphics[width=\textwidth]{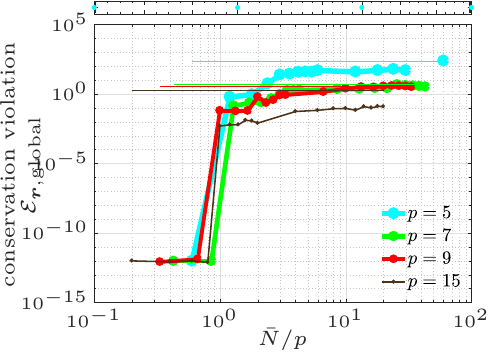}\caption{\GNATcnstd:
 $\violation$}\label{GNATcnstdviolation}\end{subfigure}
 \begin{subfigure}{0.49\textwidth}\includegraphics[width=\textwidth]{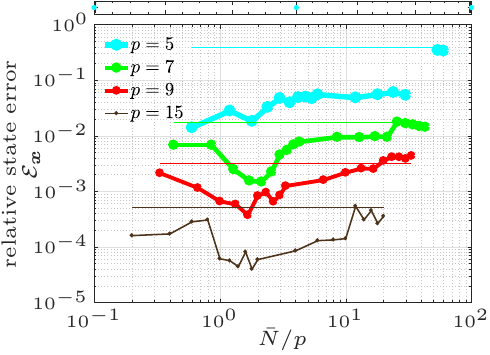}\caption{\GNATcnstdArg{\GNATcnstd}:
 $\relativeError$}\label{GNATcnstdGNATrelativestateerror}\end{subfigure}
 \begin{subfigure}{0.49\textwidth}\includegraphics[width=\textwidth]{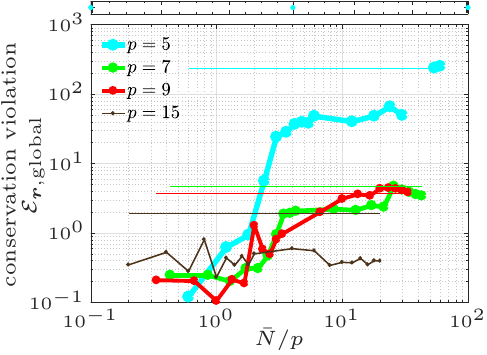}\caption{\GNATcnstdArg{\GNATcnstd}:
 $\violation$}\label{GNATcnstdGNATviolation}\end{subfigure}
\end{center}
\caption{\textit{One-dimensional Euler equation}. \GNATcnstd\
				state-basis-dimension study \reviewerA{described in Section
			\ref{sec:stateBasisStudy}}. The top bar reports cases where the
				reduced-order-model simulation was unstable. The colored horizontal
				lines correspond to the associated unconstrained ROM. \reviewerA{Each row of subfigures corresponds to a different ROM method;
				each column reports a different error measure. Within each subfigure, two parameters vary: the
		number of constraints $\nConstraints$ and the reduced-basis dimension
		$\nstateRed=5)$.}}
\label{fig:multidomain3}
\end{figure}

\subsection{Comparison across all methods}\label{sec:comparisonSteady}

This section assesses the relative performance of the methods over time; all
ROMs that employ constraints enforce global conservation, i.e., $\nConstraints
= 3$ and $\meshDecomp=\meshDecompGlobal$.

We consider two discretizations corresponding to $\nControlVol = 500$ and $\nControlVol=1000$
control volumes in the finite-volume discretization. We set
reduced-basis dimensions to  $\nstateRed = 5$ and
$\nresRed=\nfluxRed=\nsourceRed=20$ and employ a
sample mesh with 20 control volumes.  We report the time-instantaneous
state-space errors $\relativeErrorInst{n}$, $n\innat{\ntimedof}$, 
errors in the globally conserved variables
$\relativeGlobalErrorInst{n}$, $n\innat{\ntimedof}$, and violation in global
conservation $\violationInst{n}$, $n\innat{\ntimedof}$.

Figure \ref{fig:errorOverTime} reports the results. First, we note that the
errors $\relativeErrorInst{n}$ and $\relativeGlobalErrorInst{n}$ exhibit the
same trends in all cases; this suggests that enforcing global
conservation---which leads to lower errors in the globally conserved
quantities by construction---is an effective approach for also reducing the
error in the state itself. This also supports previous observations that
enforcing global conservation rather than employing a penalty approach leads
to smaller errors in most cases.

Second, we observe that the \FOM, \LSPGcnstd, and \GNATcnstd\ models all lead
to  global-conservation violations $\relativeGlobalErrorInst{n}$ near zero as
expected. In contrast, the \GNATcnstdArg{\GNATcnstd}\ approach only
approximately satisfies global conservation due the introduction of
hyper-reduction to the constraints; however, this has no noticeable effect on
its response, as the errors reported for \GNATcnstd\ and
\GNATcnstdArg{\GNATcnstd} are nearly identical in Figures
\ref{relError_N500}--\ref{relError_gCQ_N1000}.

Third, we notice that the conservative methods \LSPGcnstd\ and \GNATcnstd, as
well as the approximately conservative method \GNATcnstdArg{\GNATcnstd},
all yield significantly lower errors than the unconstrained methods Galerkin,
\LSPG, and \GNAT. Further, these unconstrained methods yield significant
violation in global conservation.

\reviewerA{Table \ref{tab:timings} reports the timings for these methods. We
	first note that the LSPG ROM does not have a valid timing for either
	problem, as the associated simulations yield negative pressures and thus do
	not successfully run for the entire time interval (see premature termination
	in Figure \ref{fig:errorOverTime}).  Second, while all other ROMs produce a
	speedup relative to the \FOM, methods that employ hyper-reduction for the
	objective function (\GNAT, \GNATcnstd) produce more significant speedups;
	further applying hyper-reduction to the constraint
	(\GNATcnstdArg{\GNATcnstd}) improves the speedup further.
}

\begin{figure}[h]
\begin{center}
\begin{subfigure}{0.49\textwidth}\includegraphics[width=\textwidth]{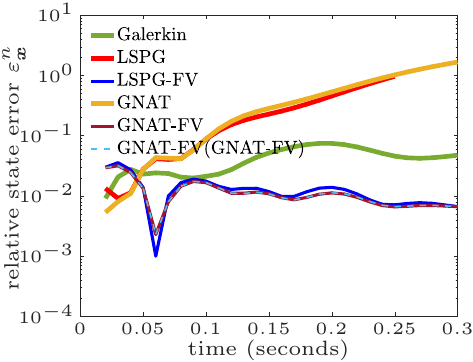}\caption{$\relativeErrorInst{n}$, $\nControlVol=500$}\label{relError_N500}\end{subfigure}
\begin{subfigure}{0.49\textwidth}\includegraphics[width=\textwidth]{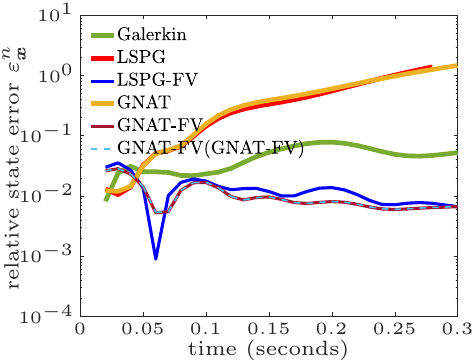}\caption{$\relativeErrorInst{n}$, $\nControlVol=1000$}\label{relError_N1000}\end{subfigure}
\begin{subfigure}{0.49\textwidth}\includegraphics[width=\textwidth]{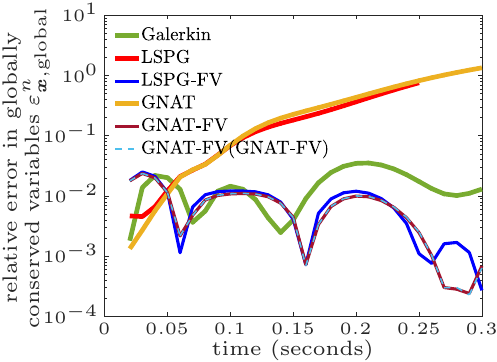}\caption{\reviewerA{$\relativeGlobalErrorInst{n}$,
$\nControlVol=500$}}\label{relError_gCQ_N500}\end{subfigure}
 \begin{subfigure}{0.49\textwidth}\includegraphics[width=\textwidth]{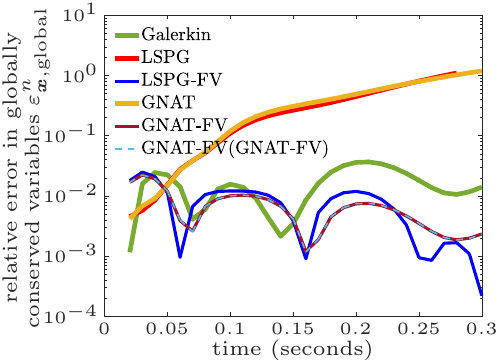}\caption{\reviewerA{$\relativeGlobalErrorInst{n}$,
 $\nControlVol=1000$}}\label{relError_gCQ_N1000}\end{subfigure}
 \begin{subfigure}{0.49\textwidth}\includegraphics[width=\textwidth]{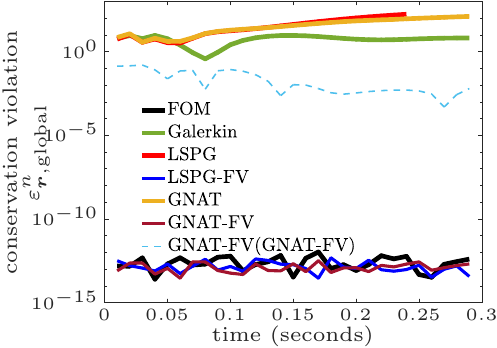}\caption{\reviewerA{$\violationInst{n}$,
 $\nControlVol=500$}}\label{GlobalConservation_N500}\end{subfigure}
\begin{subfigure}{0.49\textwidth}\includegraphics[width=\textwidth]{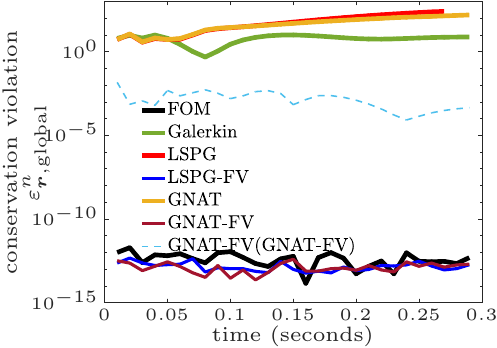}\caption{\reviewerA{$\violationInst{n}$,
$\nControlVol=1000$}}\label{GlobalConservation_N1000}\end{subfigure}
\end{center}
\caption{\textit{One-dimensional Euler equation}. Method comparison
	over time \reviewerA{as described in Section \ref{sec:comparisonSteady}},
	with conservative methods enforcing global conservation. 
	\reviewerA{Each curve depicts the time evolution of a given error measure
	for a given ROM.}
	\reviewerA{Each row
		of subfigures corresponds to a different error measure; the left
	and right subfigure columns} correspond to cases $\nControlVol=500$ and
	$\nControlVol=1000$, respectively.  Note that the missing data for the \LSPG\
method corresponds to time instances after a negative pressure was generated,
thus causing the simulation to end.}
\label{fig:errorOverTime}
\end{figure}

\begin{table}\small
		\reviewerA{
	\begin{center}
 \begin{tabular}{|c||c||c|c|c||c|c|c|} 
\hline
method & \FOM &\Gal &\LSPG &  \LSPGcnstd& \GNAT& \GNATcnstd& \GNATcnstdArg{\GNATcnstd} \\
\hline
\begin{tabular}{@{}c@{}} wall time (seconds)\\ for $\nControlVol = 500$
	\end{tabular}
&  40.9 & 29.1 & N/A & 31.9 &
13.1 & 11.1 & 8.1\\
\begin{tabular}{@{}c@{}} wall time (seconds)\\ for $\nControlVol = 1000$
	\end{tabular}
			&81.2 & 52.2 & N/A & 58.6 & 19.8 & 18.4 & 14.4\\
\hline 
\end{tabular} 
	\end{center}
	\caption{\reviewerA{\textit{One-dimensional Euler equation}. Timings for the
			ROM methods assessed in Section \ref{sec:comparisonSteady}. Here, all
ROMs that employ constraints enforce global conservation, i.e., $\nConstraints
= 3$ and $\meshDecomp=\meshDecompGlobal$. We set
reduced-basis dimensions to  $\nstateRed = 5$ and
$\nresRed=\nfluxRed=\nsourceRed=20$. ROM methods that use hyper-reduction employ a
sample mesh with 20 control volumes.}\label{tab:timings}}
}
\end{table}

To enable an objective comparison of the ROM methods, we compare their
performance across a wide variation of all method parameters.  We subject each
model to a parameter study wherein each model parameter is varied between
the limits specified in Table \ref{tab:steadyParamsPareto}. From these
results, we then construct a Pareto front for each method, which is
characterized by the method parameters that minimize the competing objectives
of error and wall time. 

\begin{table}
	\begin{center}
 \begin{tabular}{|c||c|c|} 
\hline
method & \LSPG, \LSPGcnstd& \GNAT, \GNATcnstd, \GNATcnstdArg{\GNATcnstd} \\
\hline
$\nstateRed$ &  $\{4,5,6\}$ & $\{4,5,6\}$  \\
$\nsamplesres=\nsamplessource=\nsamplesflux$  &    & $\{60,90\}$\\
$\nresRed$   &    & $\{10,20,30\}$   \\
$\nfluxRed$&    &  $\{10 + 5j\}_{j=0}^4$     \\
$\nsourceRed$  &    &  $\{10 + 5j\}_{j=0}^4$ \\
\hline 
\end{tabular} 
	\end{center}
\caption{\textit{One-dimensional Euler equation}. 
	Parameters varied for different ROMs to generate the Pareto
front\reviewerA{s reported in Figure \ref{fig:paretoFrontBurgers} as described
in Section \ref{sec:comparisonSteady}}.}\label{tab:steadyParamsPareto}
\end{table}

 Figure \ref{fig:paretoFrontBurgers} reports these Pareto fronts, where both
 the mean-squared state-space error $\relativeError$ and mean-squared
 violation in global conservation $\violation$ are considered as error
 measures, as well as an `overall' Pareto front that selects the
 Pareto-optimal methods across all parameter variations. \reviewerA{Note that
	 this figure reports the relative wall time with respect to that of the FOM simulation;
 relative wall times less than one imply the ROM yields a speedup.} Here, Figure
 \ref{paretoRelError_N500} shows that the \GNATcnstdArg{\GNATcnstd}\ method is
 always Pareto dominant for error measure $\relativeError$, as no other method is both less expensive and more
 accurate for any tested parameter combination. The method that performs
 second best is the proposed \GNATcnstd\ method, which exactly enforces constraints;
 note that it is only slightly more expensive than the
 \GNATcnstdArg{\GNATcnstd}\ method, as the benefit of performing
 hyper-reduction on the residual appearing in the constraints with $\nConstraints$ small is much less
 significant than the benefit of performing hyper-reduction on the residual
 appearing in the objective
 function when $\nConstraints$ is small. In particular, note that the
 Pareto-optimal parameter combinations for the \LSPGcnstd\ method yield
 similar accuracy to the Pareto-optimal \GNATcnstdArg{\GNATcnstd}\ points, but
 incur significantly larger wall times. Figure
 \ref{paretoGlobalConservation_N500} shows that \GNATcnstdArg{\GNATcnstd}\ is
 Pareto optimal for error measure $\violation$ for \reviewerA{relative} wall
 times less than \reviewerA{0.28}, but \reviewerA{\GNATcnstd, which enforces}
 constraints exactly\reviewerA{,} \reviewerA{is} Pareto
 optimal for larger \reviewerA{relative wall} times, yielding near-zero violations in global
 conservation. We emphasize that 
 both conservative variants of the \GNAT\ method (i.e., \GNATcnstd\ and
 \GNATcnstdArg{\GNATcnstd}) outperform the original \GNAT\ approach, and the
 conservative variant of the \LSPG\ method (i.e., \LSPGcnstd) outperforms the
 original \LSPG\ method; this demonstrates the benefit of the proposed method
 and the performance improvement gained by enforcing conservation.
 \reviewerA{In particular, note that the introduction of constraints does not
	 adversely affect ROM wall-time performance; in fact, \LSPGcnstd\ has
	 \textit{better} wall-time performance relative to the
	 the \LSPG\ method. This occurs because global conservation corresponds
	 to only $\nConstraints = 3$ constraints in this case, and because these
 constraints lead to improved accuracy and thus promote convergence, the
 associated simulations require fewer iterations to solve the optimization
 problem at each time instance.}
 \reviewerA{We also note that hyper-reduction is needed to realize significant
	 speedups: Pareto-optimal parameter combinations for ROMs employing
	 hyper-reduction lead to relative wall times less than 0.36, while
Pareto-optimal parameter combinations for ROMs without hyper-reduction 
yield relative wall times exceeding 0.6.}
\begin{figure}[h]
\begin{center}
\begin{subfigure}{0.49\textwidth}\includegraphics[width=\textwidth]{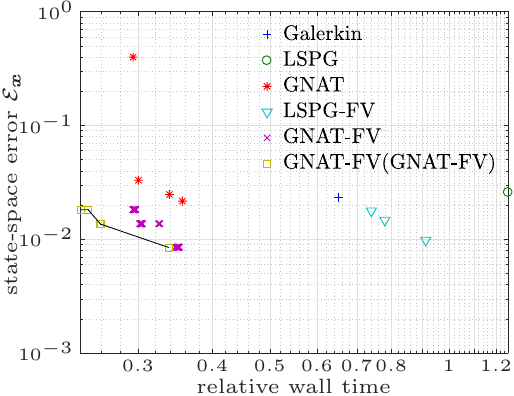}\caption{Pareto
front with error measure $\relativeError$}\label{paretoRelError_N500}\end{subfigure}
\begin{subfigure}{0.49\textwidth}\includegraphics[width=\textwidth]{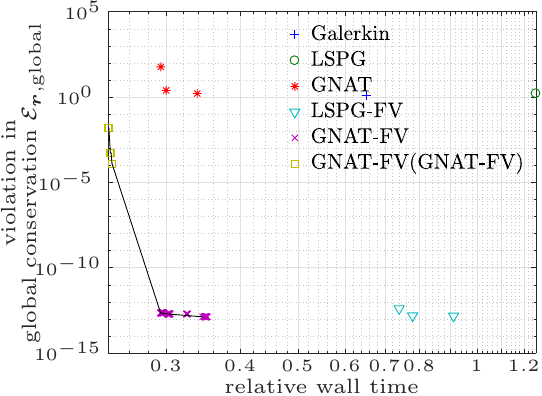}\caption{Pareto
front with error measure $\violation$}\label{paretoGlobalConservation_N500}\end{subfigure}
\end{center}
\caption{\textit{One-dimensional Euler equation}. 
	Pareto-optimal performance of various methods after
		varying model parameters reported in Table
		\ref{tab:steadyParamsPareto}
		for  $\nControlVol=500$ \reviewerA{as described in Section \ref{sec:comparisonSteady}}. \reviewerA{Wall times are reported relative to that
		of the FOM simulation.} \reviewerA{Note that the Pareto-optimal ROM
		methods in terms of minimizing error and wall time are the proposed
\GNATcnstd\ and \GNATcnstdArg{\GNATcnstd}\ methods.}}
	\label{fig:paretoFrontBurgers}
\end{figure}

	 \section{Conclusions}\label{sec:conclusions}
	 This work proposed two model-reduction methods for finite-volume models
	 that enforce conservation over subdomains: conservative Galerkin and
	 conservative LSPG projection. These methods associate with optimization
	 problems characterized by a minimum-residual objective function and
	 nonlinear equality constraints formulated at the time-continuous and
	 time-discrete levels, respectively. We equipped these methods with
	 techniques for handling infeasible constraints, and we also developed
	 hyper-reduction methods to ensure low-cost ROM simulations in the presence
	 of nonlinear flux or source terms.

	 We performed analysis that demonstrated commutativity of conservative
	 Galerkin projection and time discretization, developed sufficient conditions
	 for feasibility, demonstrated conditions under which conservative Galerkin
	 and conservative LSPG models are equivalent, and derived \textit{a
	 posteriori} error bounds.
	 Numerical experiments on a model problem highlighted the benefit of
	 conservative projection, and also demonstrated that enforcing global
	 conservation led to the most accurate results.

	 Future work involves implementing the proposed techniques in
	 a production-level computational fluid-dynamics code, demonstrating the
	 methods on truly large-scale finite-volume models, and investigating
	 combining the methodology with space--time projection approaches
	 \cite{volkwein2006algorithm,baumann2016space,choiCarlberg}, as these
	 techniques have demonstrated error bounds that grow slowly in time. 

	 \section*{Acknowledgments}
We thank Matthew Barone and Irina Tezaur for insightful conversations related
to structure-preserving model reduction in fluid dynamics. We also thank J.\
Nathan Kutz for his help in forging the collaboration. This work was
funded by Sandia's Laboratory Directed Research and Development (LDRD)
program under Project \#190968.
Sandia National Laboratories is a multimission laboratory
managed and operated by National Technology and Engineering Solutions of
Sandia, LLC., a wholly owned subsidiary of Honeywell International, Inc.,
for the U.S.\ Department of Energy's National Nuclear Security
Administration
under contract DE-NA-0003525.

\bibliography{references}

\begin{thebibliography}{10}

\bibitem{abgrall2015robust}
{\sc R.~Abgrall, D.~Amsallem, and R.~Crisonovan}, {\em Robust model reduction
  by $l^1$-norm minimization and approximation via dictionaries: Application to
  linear and nonlinear hyperbolic problems}, arXiv e-print,  (2016).

\bibitem{an2008optimizing}
{\sc S.~An, T.~Kim, and D.~James}, {\em Optimizing cubature for efficient
  integration of subspace deformations}, ACM Transactions on Graphics (TOG), 27
  (2008), p.~165.

\bibitem{HAntil_SField_RHNochetto_MTiglio_2013}
{\sc H.~Antil, S.~Field, F.~Herrmann, R.~Nochetto, and M.~Tiglio}, {\em
  Two-step greedy algorithm for reduced order quadratures}, Journal of
  Scientific Computing, 57 (2013), pp.~604--637.

\bibitem{HAntil_MHeinkenschloss_DCSorensen_2013a}
{\sc H.~Antil, M.~Heinkenschloss, and D.~C. Sorensen}, {\em Application of the
  discrete empirical interpolation method to reduced order modeling of
  nonlinear and parametric systems}, vol.~8 of Springer MS\&A series: Reduced
  Order Methods for modeling and computational r G. Rozza, Eds, Springer-Verlag
  Italia, Milano, 2013.

\bibitem{astrid2007mpe}
{\sc P.~Astrid, S.~Weiland, K.~Willcox, and T.~Backx}, {\em Missing point
  estimation in models described by proper orthogonal decomposition}, IEEE
  Transactions on Automatic Control, 53 (2008), pp.~2237--2251.

\bibitem{aubry1988dynamics}
{\sc N.~Aubry, P.~Holmes, J.~L. Lumley, and E.~Stone}, {\em The dynamics of
  coherent structures in the wall region of a turbulent boundary layer},
  Journal of Fluid Mechanics, 192 (1988), pp.~115--173.

\bibitem{balajewicz2012stabilization}
{\sc M.~Balajewicz and E.~Dowell}, {\em Stabilization of projection-based
  reduced order models of the navier--stokes equations}, Nonlinear Dynamics, 70
  (2012), pp.~1619--1632.

\bibitem{balajewicz2013low}
{\sc M.~Balajewicz, E.~Dowell, and B.~Noack}, {\em Low-dimensional modelling of
  high-reynolds-number shear flows incorporating constraints from the
  navier--stokes equation}, Journal of Fluid Mechanics, 729 (2013),
  pp.~285--308.

\bibitem{barone2009stable}
{\sc M.~F. Barone, I.~Kalashnikova, D.~J. Segalman, and H.~K. Thornquist}, {\em
  Stable {G}alerkin reduced order models for linearized compressible flow},
  Journal of Computational Physics, 228 (2009), pp.~1932--1946.

\bibitem{barrault2004eim}
{\sc M.~Barrault, Y.~Maday, N.~C. Nguyen, and A.~T. Patera}, {\em An `empirical
  interpolation' method: application to efficient reduced-basis discretization
  of partial differential equations}, Comptes Rendus Math\'ematique Acad\'emie
  des Sciences, 339 (2004), pp.~667--672.

\bibitem{baumann2016space}
{\sc M.~Baumann, P.~Benner, and J.~Heiland}, {\em Space-time {G}alerkin {POD}
  with application in optimal control of semi-linear parabolic partial
  differential equations}, arXiv preprint arXiv:1611.04050,  (2016).

\bibitem{Bergmann2009516}
{\sc M.~Bergmann, C.-H. Bruneau, and A.~Iollo}, {\em Enablers for robust {POD}
  models}, Journal of Computational Physics, 228 (2009), pp.~516 -- 538.

\bibitem{bos2004als}
{\sc R.~Bos, X.~Bombois, and P.~Van~den Hof}, {\em Accelerating large-scale
  non-linear models for monitoring and control using spatial and temporal
  correlations}, Proceedings of the American Control Conference, 4 (2004),
  pp.~3705--3710.

\bibitem{carlberg2014adaptive}
{\sc K.~Carlberg}, {\em Adaptive $h$-refinement for reduced-order models},
  International Journal for Numerical Methods in Engineering, 102 (2015),
  pp.~1192--1210.

\bibitem{carlbergGalDiscOpt}
{\sc K.~Carlberg, M.~Barone, and H.~Antil}, {\em {G}alerkin v.\ least-squares
  {P}etrov--{G}alerkin projection in nonlinear model reduction}, Journal of
  Computational Physics, 330 (2017), pp.~693--734.

\bibitem{CarlbergGappy}
{\sc K.~Carlberg, C.~Bou-Mosleh, and C.~Farhat}, {\em {Efficient non-linear
  model reduction via a least-squares {P}etrov--{G}alerkin projection and
  compressive tensor approximations}}, International Journal for Numerical
  Methods in Engineering, 86 (2011), pp.~155--181.

\bibitem{carlbergJCP}
{\sc K.~Carlberg, C.~Farhat, J.~Cortial, and D.~Amsallem}, {\em The {GNAT}
  method for nonlinear model reduction: effective implementation and
  application to computational fluid dynamics and turbulent flows}, Journal of
  Computational Physics, 242 (2013), pp.~623--647.

\bibitem{carlberg2012spd}
{\sc K.~Carlberg, R.~Tuminaro, and P.~Boggs}, {\em Preserving {L}agrangian
  structure in nonlinear model reduction with application to structural
  dynamics}, SIAM J. Sci. Comput., 37 (2015), pp.~B153---B184.

\bibitem{chaturantabut2010journal}
{\sc S.~Chaturantabut and D.~C. Sorensen}, {\em Nonlinear model reduction via
  discrete empirical interpolation}, SIAM Journal on Scientific Computing, 32
  (2010), pp.~2737--2764.

\bibitem{choiCarlberg}
{\sc Y.~Choi and K.~Carlberg}, {\em Space--time least-squares
  {P}etrov--{G}alerkin projection for nonlinear model reduction}, arXiv
  preprint arXiv:1703.04560,  (2017).

\bibitem{deane1991low}
{\sc A.~Deane, I.~Kevrekidis, G.~E. Karniadakis, and S.~Orszag}, {\em
  Low-dimensional models for complex geometry flows: Application to grooved
  channels and circular cylinders}, Physics of Fluids A: Fluid Dynamics, 3
  (1991), pp.~2337--2354.

\bibitem{drohmannEOI}
{\sc M.~Drohmann, B.~Haasdonk, and M.~Ohlberger}, {\em Reduced basis
  approximation for nonlinear parametrized evolution equations based on
  empirical operator interpolation}, SIAM Journal on Scientific Computing, 34
  (2012), pp.~A937--A969.

\bibitem{sirovichOrigGappy}
{\sc R.~Everson and L.~Sirovich}, {\em {K}arhunen--{L}o\`{e}ve procedure for
  gappy data}, Journal of the Optical Society of America A, 12 (1995),
  pp.~1657--1664.

\bibitem{Fang2013540}
{\sc F.~Fang, C.~Pain, I.~Navon, A.~Elsheikh, J.~Du, and D.~Xiao}, {\em
  Non-linear {P}etrov--{G}alerkin methods for reduced order hyperbolic
  equations and discontinuous finite element methods}, Journal of Computational
  Physics, 234 (2013), pp.~540 -- 559.

\bibitem{farhat2014dimensional}
{\sc C.~Farhat, P.~Avery, T.~Chapman, and J.~Cortial}, {\em Dimensional
  reduction of nonlinear finite element dynamic models with finite rotations
  and energy-based mesh sampling and weighting for computational efficiency},
  International Journal for Numerical Methods in Engineering, 98 (2014),
  pp.~625--662.

\bibitem{fick2017reduced}
{\sc L.~Fick, Y.~Maday, A.~T. Patera, and T.~Taddei}, {\em A reduced basis
  technique for long-time unsteady turbulent flows}, arXiv preprint
  arXiv:1710.03569,  (2017).

\bibitem{galbally2009non}
{\sc D.~Galbally, K.~Fidkowski, K.~Willcox, and O.~Ghattas}, {\em {Non-linear
  model reduction for uncertainty quantification in large-scale inverse
  problems}}, International Journal for Numerical Methods in Engineering, 81
  (2009), pp.~1581--1608.

\bibitem{galletti2004low}
{\sc B.~Galletti, C.~Bruneau, L.~Zannetti, and A.~Iollo}, {\em Low-order
  modelling of laminar flow regimes past a confined square cylinder}, Journal
  of Fluid Mechanics, 503 (2004), pp.~161--170.

\bibitem{gerbeau2014approximated}
{\sc J.-F. Gerbeau and D.~Lombardi}, {\em Approximated lax pairs for the
  reduced order integration of nonlinear evolution equations}, Journal of
  Computational Physics, 265 (2014), pp.~246--269.

\bibitem{haasdonk2008reduced}
{\sc B.~Haasdonk and M.~Ohlberger}, {\em Reduced basis method for explicit
  finite volume approximations of nonlinear conservation laws}, in Proc. 12th
  International Conference on Hyperbolic Problems: Theory, Numerics,
  Application, 2008.

\bibitem{haasdonk2008FV}
\leavevmode\vrule height 2pt depth -1.6pt width 23pt, {\em Reduced basis method
  for finite volume approximations of parametrized linear evolution equations},
  ESAIM-Mathematical Modelling and Numerical Analysis, 42 (2008), pp.~277--302.

\bibitem{haasdonkExplicit}
{\sc B.~Haasdonk, M.~Ohlberger, and G.~Rozza}, {\em A reduced basis method for
  evolution schemes with parameter-dependent explicit operators}, Electronic
  Transactions on Numerical Analysis, 32 (2008), pp.~145--161.

\bibitem{POD}
{\sc P.~Holmes, J.~Lumley, and G.~Berkooz}, {\em Turbulence, Coherent
  Structures, Dynamical Systems and Symmetry}, Cambridge University Press,
  1996.

\bibitem{iollo2000stability}
{\sc A.~Iollo, S.~Lanteri, and J.~A. Desideri}, {\em Stability properties of
  {POD}--{{G}alerkin} approximations for the compressible {Navier--Stokes}
  equations}, Theoretical and Computational Fluid Dynamics, 13 (2000),
  pp.~377--396.

\bibitem{jolly1991preserving}
{\sc M.~Jolly, I.~Kevrekidis, and E.~Titi}, {\em Preserving dissipation in
  approximate inertial forms for the {K}uramoto-{S}ivashinsky equation},
  Journal of Dynamics and Differential Equations, 3 (1991), pp.~179--197.

\bibitem{kalashnikova2010stability}
{\sc I.~Kalashnikova and M.~Barone}, {\em On the stability and convergence of a
  {G}alerkin reduced order model (rom) of compressible flow with solid wall and
  far-field boundary treatment}, International Journal for Numerical Methods in
  Engineering, 83 (2010), pp.~1345--1375.

\bibitem{LeGresleyThesis}
{\sc P.~A. LeGresley}, {\em Application of proper orthogonal decomposition
  ({POD}) to design decomposition methods}, PhD thesis, Stanford University,
  2006.

\bibitem{lorenzi2016pod}
{\sc S.~Lorenzi, A.~Cammi, L.~Luzzi, and G.~Rozza}, {\em {POD}--{G}alerkin
  method for finite volume approximation of {N}avier--{S}tokes and {RANS}
  equations}, Computer Methods in Applied Mechanics and Engineering, 311
  (2016), pp.~151--179.

\bibitem{ma2002low}
{\sc X.~Ma and G.~E. Karniadakis}, {\em A low-dimensional model for simulating
  three-dimensional cylinder flow}, Journal of Fluid Mechanics, 458 (2002),
  pp.~181--190.

\bibitem{maccormackNote}
{\sc R.~MacCormack}, {\em Numerical computation of compressible viscous flow},
  tech. rep., Lecture notes for AA214b and AA214c (Stanford University), 2007.

\bibitem{marion1989nonlinear}
{\sc M.~Marion and R.~Temam}, {\em Nonlinear {G}alerkin methods}, SIAM Journal
  on Numerical Analysis, 26 (1989), pp.~1139--1157.

\bibitem{noack2005need}
{\sc B.~Noack, P.~Papas, and P.~Monkewitz}, {\em The need for a pressure-term
  representation in empirical {G}alerkin models of incompressible shear flows},
  Journal of Fluid Mechanics, 523 (2005), pp.~339--365.

\bibitem{ohlberger2013nonlinear}
{\sc M.~Ohlberger and S.~Rave}, {\em Nonlinear reduced basis approximation of
  parameterized evolution equations via the method of freezing}, Comptes Rendus
  Mathematique, 351 (2013), pp.~901--906.

\bibitem{prud2002reliable}
{\sc C.~Prud'homme, D.~Rovas, K.~Veroy, L.~Machiels, Y.~Maday, A.~Patera, and
  G.~Turinici}, {\em {Reliable real-time solution of parameterized partial
  differential equations: Reduced-basis output bound methods}}, Journal of
  Fluids Engineering, 124 (2002), pp.~70--80.

\bibitem{reddy2017constrained}
{\sc S.~R. Reddy, B.~A. Freno, P.~G. Cizmas, S.~Gokaltun, D.~McDaniel, and
  G.~S. Dulikravich}, {\em Constrained reduced-order models based on proper
  orthogonal decomposition}, Computer Methods in Applied Mechanics and
  Engineering, 321 (2017), pp.~18--34.

\bibitem{rowley2004mrc}
{\sc C.~Rowley, T.~Colonius, and R.~Murray}, {\em Model reduction for
  compressible flows using {POD} and {{G}alerkin} projection}, Physica D:
  Nonlinear Phenomena, 189 (2004), pp.~115--129.

\bibitem{rozza2007reduced}
{\sc G.~Rozza, D.~B.~P. Huynh, and A.~T. Patera}, {\em {Reduced basis
  approximation and a posteriori error estimation for affinely parametrized
  elliptic coercive partial differential equations}}, Archives of Computational
  Methods in Engineering, 15 (2008), pp.~229--275.

\bibitem{ryckelynck2005phm}
{\sc D.~Ryckelynck}, {\em A priori hyperreduction method: an adaptive
  approach}, Journal of Computational Physics, 202 (2005), pp.~346--366.

\bibitem{san2013proper}
{\sc O.~San and T.~Iliescu}, {\em Proper orthogonal decomposition closure
  models for fluid flows: Burgers equation}, arXiv preprint arXiv:1308.3276,
  (2013).

\bibitem{shen1990long}
{\sc J.~Shen}, {\em Long time stability and convergence for fully discrete
  nonlinear {G}alerkin methods}, Applicable Analysis, 38 (1990), pp.~201--229.

\bibitem{sirisup2004spectral}
{\sc S.~Sirisup and G.~Karniadakis}, {\em A spectral viscosity method for
  correcting the long-term behavior of pod models}, Journal of Computational
  Physics, 194 (2004), pp.~92--116.

\bibitem{sirovich1987tad3}
{\sc L.~Sirovich}, {\em Turbulence and the dynamics of coherent structures.
  {III:} dynamics and scaling}, Quarterly of Applied Mathematics, 45 (1987),
  pp.~583--590.

\bibitem{stabile2017advances}
{\sc G.~Stabile, S.~Hijazi, A.~Mola, S.~Lorenzi, and G.~Rozza}, {\em Advances
  in reduced order modelling for {CFD}: vortex shedding around a circular
  cylinder using a {POD}--{G}alerkin method}, arXiv preprint arXiv:1701.03424,
  (2017).

\bibitem{stabile2017finite}
{\sc G.~Stabile and G.~Rozza}, {\em Finite volume {POD}-{G}alerkin stabilised
  reduced order methods for the parametrised incompressible navier-stokes
  equations}, arXiv preprint arXiv:1710.11580,  (2017).

\bibitem{taddei2015reduced}
{\sc T.~Taddei, S.~Perotto, and A.~Quarteroni}, {\em Reduced basis techniques
  for nonlinear conservation laws}, ESAIM: Mathematical Modelling and Numerical
  Analysis, 49 (2015), pp.~787--814.

\bibitem{volkwein2006algorithm}
{\sc S.~Volkwein and S.~Weiland}, {\em An algorithm for {G}alerkin projections
  in both time and spatial coordinates}, Proc. 17th MTNS,  (2006).

\bibitem{wang2012proper}
{\sc Z.~Wang, I.~Akhtar, J.~Borggaard, and T.~Iliescu}, {\em Proper orthogonal
  decomposition closure models for turbulent flows: a numerical comparison},
  Computer Methods in Applied Mechanics and Engineering, 237 (2012),
  pp.~10--26.

\bibitem{zahr2010comparison}
{\sc M.~J. Zahr, K.~Carlberg, D.~Amsallem, and C.~Farhat}, {\em Comparison of
  model reduction techniques on high-fidelity linear and nonlinear electrical,
  mechanical, and biological systems}, University of California, Berkeley,
  (2010).

\bibitem{zimmermann2014reduced}
{\sc R.~Zimmermann, A.~Vendl, and S.~G{\"o}rtz}, {\em Reduced-order modeling of
  steady flows subject to aerodynamic constraints}, AIAA Journal, 52 (2014).

\end{thebibliography}
\bibliographystyle{siam}
\end{document}